\documentclass[12pt]{amsart}
\usepackage[T1]{fontenc}
\usepackage[utf8]{inputenc}
\usepackage[
  backend=biber,
  style=numeric,
  sorting=none,
  giveninits=false,
  doi=true,
  isbn=true,
  url=true,
  eprint=true
]{biblatex}
\usepackage{amsmath, amssymb, amsthm,mathtools}
\usepackage{mathrsfs}
\usepackage{bbm}
\usepackage{hyperref}
\usepackage{geometry}
\usepackage{enumitem}
\usepackage{booktabs}
\usepackage{array}
\usepackage{tabularx}
\usepackage{tikz-cd}

\newtheorem{theorem}{Theorem}[section]
\newtheorem{proposition}[theorem]{Proposition}
\newtheorem{lemma}[theorem]{Lemma}
\newtheorem{corollary}[theorem]{Corollary}
\newtheorem{definition}[theorem]{Definition}

\theoremstyle{remark}
\newtheorem{remark}[theorem]{Remark}

\newtheorem{principle}[theorem]{Principle}

\usepackage{amsmath,amssymb,mathtools}

\newcommand{\RealB}{\operatorname{Real}_{B}}
\newcommand{\Realet}{\operatorname{Real}_{\acute{e}t}}

\newcommand{\RealMHM}{\operatorname{Real}_{\mathrm{MHM}}}
\newcommand{\RealNori}{\operatorname{Real}_{\mathrm{Nori}}}

\begin{document}

\title[Diaz and \(n\)-Fold Bockstein Counterexamples]%
{From Diaz's Enriques Product to an \(n\)-Fold Cup-Product Bockstein Family of Integral Hodge Counterexamples}

\author{Abdul Rahman}
\thanks{Email: arahman@alum.howard.edu}

\begin{abstract}
We reinterpret Diaz's construction of Chow-trivial smooth projective varieties violating the integral Hodge conjecture as the level-two case of an \(n\)-fold cup-product Bockstein mechanism.  Diaz's dimension-four example is \(V=S_1\times S_2\), where \(S_1,S_2\) are Enriques surfaces, and its obstruction is the Bockstein of
$\pi_1^*\alpha_1\cup\pi_2^*\beta_2\in H^3(V,\mathbb Z/2(2))$.  Here
\(\alpha_1\) is the K3 double-cover class and \(\beta_2\) is an Enriques Brauer-detecting class. We extend the finite-coefficient source construction to \(X_n=S_1\times\cdots\times S_n\) by forming
   $\Theta_n=
   \pi_1^*\alpha_1\cup\pi_2^*\beta_2\cup\cdots\cup\pi_n^*\beta_n
   \in H^{2n-1}(X_n,\mathbb Z/2(n))$,
with Bockstein
   $\Delta_n=\delta(\Theta_n)\in H^{2n}(X_n,\mathbb Z(n))$.
Using external products of perverse sheaves, categorical Bockstein
compatibility, and a Leibniz rule for the MacPherson--Vilonen boundary, we prove unconditionally that \(\Delta_n\) has nonzero image in a distinguished Enriques--Brauer component of the MV obstruction channel. Under the Brauer-separation hypothesis, which asserts that algebraic codimension-\(n\) cycle classes have zero image in this same component, the class \(\Delta_n\) is a non-algebraic \(2\)-torsion integral Hodge class.  We verify this separation for decomposable algebraic cycles and reduce the
remaining non-decomposable case, via integral even Chow--K\"unneth projectors on the Enriques factors, to a single coefficient-level algebraic-control problem involving the \(H^1(S_1,\mathbb Z/2(1))\) Enriques double-cover direction.  We also record a motivic finite-coefficient lift of the tower via the finite-coefficient cone
\(\mathbf 1_X(n)/2:=\operatorname{Cone}(\mathbf 1_X(n)\xrightarrow{\times 2}
\mathbf 1_X(n))\),
and explain which formal part of the MacPherson--Vilonen zig-zag construction lifts motivically under Betti realization.
\end{abstract}
\maketitle
\tableofcontents

\section{Introduction}

The purpose of this paper is to reinterpret Diaz's construction of
Chow-trivial varieties \cite{Diaz2023IHCTrivialChow} violating the integral Hodge conjecture as the level-two case of a broader \(n\)-fold cup-product Bockstein mechanism.  Diaz proves that, for every \(d\ge 4\), there are smooth projective varieties over \(\mathbb C\) of dimension \(d\) that are Chow-trivial and violate the integral Hodge conjecture in degree \(4\).  In dimension \(4\), the example has the form
\[
        V=S_1\times S_2,
\]
where \(S_1\) and \(S_2\) are Enriques surfaces
\cite{Diaz2023IHCTrivialChow}.  The obstruction is a non-algebraic
\(2\)-torsion class in
\[
        H^4(V,\mathbb Z(2)).
\]

Diaz's obstruction is governed by a finite-coefficient Bockstein lifecycle.
More precisely, Diaz constructs a class
\[
        \gamma
        =
        \pi_1^*\alpha_1\cup\pi_2^*\beta_2
        \in H^3(V,\mathbb Z/2(2)),
\]
where
\[
        \alpha_1\in H^1(S_1,\mathbb Z/2(1))
\]
is the K3 double-cover class of the Enriques surface \(S_1\), and
\[
        \beta_2\in H^2(S_2,\mathbb Z/2(1))
\]
maps nontrivially to
\[
        \operatorname{Br}(S_2)[2].
\]
The Bockstein of \(\gamma\), associated to the coefficient sequence
\[
        0\to \mathbb Z(2)
        \xrightarrow{\times 2}
        \mathbb Z(2)
        \to
        \mathbb Z/2(2)
        \to 0,
\]
gives the resulting integral torsion class
\[
        \delta(\gamma)\in H^4(V,\mathbb Z(2)).
\]

Thus the Diaz construction is not primarily a Kummer fixed-point
construction.  It is an Enriques-product, Brauer, finite-coefficient,
unramified, and Bockstein construction.  The key point of this paper is that
this mechanism has a natural \(n\)-fold extension. For
\[
        X_n=S_1\times\cdots\times S_n,
\]
we form the \(n\)-fold finite-coefficient product
\[
        \Theta_n
        =
        \pi_1^*\alpha_1
        \cup
        \pi_2^*\beta_2
        \cup
        \cdots
        \cup
        \pi_n^*\beta_n
        \in H^{2n-1}(X_n,\mathbb Z/2(n)),
\]
where \(\alpha_1\) is the K3 double-cover class and the
\(\beta_i\), \(2\leq i\leq n\), are Enriques Brauer-detecting classes.  This higher cup-product direction suggested in Remark~M.1 of  \cite{RahmanIntegralTorsionTrajectories} is developed here into an explicit \(n\)-fold Enriques--Brauer Bockstein mechanism extending Diaz's level-two construction. Its Bockstein
\[
        \Delta_n
        :=
        \delta(\Theta_n)
        \in H^{2n}(X_n,\mathbb Z(n))
\]
is the \(n\)-fold cup-product Bockstein class. 

The MacPherson--Vilonen formalism developed below proves that \(\Delta_n\)
has a nonzero image in a distinguished Enriques--Brauer component of the MV
obstruction channel.  Under the Brauer-separation hypothesis introduced in
Definition~\ref{def:brauer-separation-hypothesis}, algebraic
codimension-\(n\) cycle classes have zero image in this same component.
Thus, under this explicit separation condition, the classes \(\Delta_n\)
give an \(n\)-fold cup-product Bockstein tower of \(2\)-torsion integral
Hodge counterexamples. 

\subsection{From Coble/Benoist--Ottem to Diaz and beyond}

In earlier work, the Coble boundary package was identified with the cyclic
quotient singularity
\[
        \frac14(1,1),
\]
whose local package is
\[
        E\cong \mathbb Z/4.
\]
The visible Enriques \(2\)-torsion direction is the order-two Bockstein
shadow
\[
        2E\cong \mathbb Z/2.
\]
In that setting, the visible \(2\)-torsion is not the full local package.  It
is the distinguished order-two shadow inside a larger \(\mathbb Z/4\)-package.

Diaz's construction uses Enriques surfaces differently.  Instead of starting
from a single boundary singularity, it starts from a product
\[
        S_1\times S_2
\]
of two Enriques surfaces.  The first Enriques surface contributes the
double-cover class
\[
        \alpha_1\in H^1(S_1,\mathbb Z/2(1)).
\]
The second Enriques surface contributes a Brauer-detecting class
\[
        \beta_2\in H^2(S_2,\mathbb Z/2(1))
\]
whose image in \(\operatorname{Br}(S_2)[2]\) is nonzero.  Their cup product
produces
\[
        \Theta_2
        =
        \pi_1^*\alpha_1\cup\pi_2^*\beta_2
        \in H^3(S_1\times S_2,\mathbb Z/2(2)).
\]
The Bockstein of \(\Theta_2\) is the integral \(2\)-torsion Hodge class
\[
        \Delta_2
        =
        \delta(\Theta_2)
        \in H^4(S_1\times S_2,\mathbb Z(2)).
\]

The connection with the Coble/Benoist--Ottem bridge is therefore
mechanism-level rather than literal.  The earlier Coble package isolates a
single visible Enriques/Coble \(2\)-torsion direction, realized as the
order-two shadow inside the Coble boundary package.  Diaz's construction
starts from the same kind of Enriques \(2\)-torsion input, couples it with a
second Brauer-detecting Enriques class, and forms their finite-coefficient
cup product \cite{Diaz2023IHCTrivialChow}.  In this sense, Diaz's
obstruction is the level-two product version of the Coble/Enriques
Bockstein lifecycle.

The present paper goes one step further.  Diaz's construction is not only a
level-two analogue of the Coble/Enriques mechanism; it is the first
nontrivial member of an \(n\)-fold cup-product Bockstein tower.  The higher
members are formed from one Enriques double-cover class and \(n-1\) Enriques
Brauer-detecting classes.

\subsection{Main structural thesis}

The main thesis of this paper is that Diaz's obstruction is the \(n=2\)
case of a general cup-product Bockstein tower.  Let
\[
        X_n=S_1\times\cdots\times S_n
\]
be a product of Enriques surfaces.  Let
\[
        \alpha_1\in H^1(S_1,\mathbb Z/2(1))
\]
be the K3 double-cover class, and let
\[
        \beta_i\in H^2(S_i,\mathbb Z/2(1)),
        \qquad
        2\leq i\leq n,
\]
be Brauer-detecting classes.  Define
\[
        \Theta_n
        =
        \pi_1^*\alpha_1
        \cup
        \pi_2^*\beta_2
        \cup
        \cdots
        \cup
        \pi_n^*\beta_n.
\]
Then
\[
        \Theta_n\in H^{2n-1}(X_n,\mathbb Z/2(n)).
\]
The coefficient sequence
\[
        0\to \mathbb Z(n)
        \xrightarrow{\times 2}
        \mathbb Z(n)
        \to
        \mathbb Z/2(n)
        \to 0
\]
gives the Bockstein class
\[
        \Delta_n
        =
        \delta(\Theta_n)
        \in H^{2n}(X_n,\mathbb Z(n)).
\]

For \(n=2\), this is Diaz's degree-four obstruction.  For \(n=3\), it gives
the degree-six class
\[
        \Delta_3
        =
        \delta\left(
        \pi_1^*\alpha_1
        \cup
        \pi_2^*\beta_2
        \cup
        \pi_3^*\beta_3
        \right)
        \in H^6(S_1\times S_2\times S_3,\mathbb Z(3)).
\]
For general \(n\), it gives a class in codimension \(n\).

The survivability step is controlled by a MacPherson--Vilonen
zig-zag/Bockstein compatibility.  The finite-coefficient class \(\Theta_n\)
determines an MV tuple
\[
        \mathcal Z(\Theta_n),
\]
and the coefficient Bockstein is identified with the boundary of this tuple:
\[
        \rho_{\mathrm{MV},\Theta_n}(\Delta_n)
        =
        \partial_{\mathrm{MV}}\bigl(\mathcal Z(\Theta_n)\bigr)
\]
in the tuple-relative MV obstruction channel.  Thus the finite-coefficient
product is not treated merely as a cup product in ordinary cohomology.  It is
also a gluing datum whose boundary detects the integral Bockstein
obstruction.

The external tensor product theorem for perverse sheaves, the categorical
Bockstein formalism, and the Leibniz rule for the MV boundary show that
\[
        \rho_{\mathrm{MV},\Theta_n}(\Delta_n)
\]
has nonzero image in the Enriques--Brauer component
\[
        Q_n
        =
        \langle\alpha_1\rangle
        \otimes
        \operatorname{Br}(S_2)[2]
        \otimes
        \cdots
        \otimes
        \operatorname{Br}(S_n)[2].
\]
The Brauer-separation hypothesis asserts that algebraic codimension-\(n\)
cycle classes have zero image in this same component.  Under this hypothesis,
\(\Delta_n\) is not algebraic.

The resulting picture has the following stations:
\[
        \text{double-cover source},
        \qquad
        \text{Brauer sources},
        \qquad
        \text{finite-coefficient cup product},
\]
\[
        \text{MV survivability},
        \qquad
        \text{Brauer separation},
        \qquad
        \text{Bockstein image}.
\]
For \(n=2\), the survivability station agrees with the Diaz
unramified-cohomology argument based on the Colliot-Thélène--Voisin
criterion.  For the higher cup-product cases, the paper uses the
MacPherson--Vilonen obstruction-channel criterion and the Brauer-separation
hypothesis.

\subsection{Main contributions}

The paper has seven main contributions.

\begin{enumerate}[label=\textup{(\arabic*)}]
\item \textbf{Diaz extraction.}  We isolate the Enriques-product construction
used by Diaz and rewrite the obstruction as the Bockstein of the finite
coefficient class
\[
        \Theta_2
        =
        \pi_1^*\alpha_1\cup\pi_2^*\beta_2.
\]
The Bockstein
\[
        \Delta_2=\delta(\Theta_2)
\]
is Diaz's non-algebraic \(2\)-torsion integral Hodge class.

\item \textbf{Level-two interpretation.}  We identify Diaz's construction as
the level-two member of a cup-product Bockstein tower.  The first input is the
Enriques double-cover class, and the second input is an Enriques
Brauer-detecting class.

\item \textbf{MacPherson--Vilonen survivability.}  We introduce a
MacPherson--Vilonen zig-zag/Bockstein compatibility criterion.  The
coefficient Bockstein is realized as the boundary of a finite-coefficient
gluing datum.  Nonzero boundary gives survivability in the MV obstruction
channel.

\item \textbf{Enriques--Brauer separation.}  We define the Enriques--Brauer
component
\[
        Q_n
        =
        \langle\alpha_1\rangle
        \otimes
        \operatorname{Br}(S_2)[2]
        \otimes
        \cdots
        \otimes
        \operatorname{Br}(S_n)[2],
\]
and the Brauer-separating projection
\[
        \Pi_{\operatorname{Br},n}.
\]
We prove that the constructed Bockstein class has nonzero image in \(Q_n\).
We also prove the Brauer-separation condition for decomposable algebraic
cycles and isolate the non-decomposable correspondence case as the remaining
algebraic-control issue.

\item \textbf{Level-three and \(n\)-fold tower.}  We construct the triple
product
\[
        \Theta_3
        =
        \pi_1^*\alpha_1
        \cup
        \pi_2^*\beta_2
        \cup
        \pi_3^*\beta_3
        \in H^5(S_1\times S_2\times S_3,\mathbb Z/2(3)),
\]
and study its Bockstein
\[
        \Delta_3=\delta(\Theta_3)
        \in H^6(S_1\times S_2\times S_3,\mathbb Z(3)).
\]
More generally, we prove the \(n\)-fold Brauer-separated Bockstein theorem:
under the Brauer-separation hypothesis, the classes
\[
        \Delta_n=\delta(\Theta_n)\in H^{2n}(X_n,\mathbb Z(n))
\]
are non-algebraic \(2\)-torsion integral Hodge classes.

\item \textbf{Coble/Enriques comparison.}  We compare the double-cover input
with the Coble/ Benoist--Ottem order-two boundary shadow
\[
        2E\subset E\cong \mathbb Z/4.
\]
This identifies the first Diaz input as the same visible Enriques
\(2\)-torsion direction that appears in the Coble bridge.

\item \textbf{Motivic finite-coefficient lift.}  We construct a motivic finite-coefficient lift of the entire \(n\)-fold tower via objects of the form
\[
        \mathbf 1_X(n)/2
        =
        \operatorname{Cone}
        \left(
        \mathbf 1_X(n)
        \xrightarrow{\times 2}
        \mathbf 1_X(n)
        \right).
\]
We also identify the formal part of the MacPherson--Vilonen zig-zag
construction that lifts motivically under Betti realization, while deferring
the full integral motivic MacPherson--Vilonen equivalence to future work.
\end{enumerate}

\subsection{What this paper is not}

This paper does not claim to prove the rational Hodge conjecture.  It studies
integral torsion obstructions.  The rational Hodge conjecture concerns the
free part of the integral Hodge obstruction group after tensoring with
\(\mathbb Q\); the present paper studies torsion obstructions that disappear
rationally but remain structurally meaningful integrally.

This paper also does not claim that Diaz's construction is a Kummer
fixed-point construction.  The actual Diaz construction is based on a product
of Enriques surfaces, a K3 double-cover class, a Brauer-detecting class, and
an unramified-cohomology argument.  Kummer fixed-point compatibility remains
an important separate calibration for future stacky \(E_G\)-theory, but it is
not the main Diaz mechanism.

This paper does not claim to prove unconditionally that every product
\[
        X_n=S_1\times\cdots\times S_n
\]
with the above input data violates the integral Hodge conjecture.  The
higher-level theorem is stated under the Brauer-separation hypothesis, which
asserts that algebraic codimension-\(n\) cycle classes have zero image in the
same Enriques--Brauer component in which the constructed Bockstein class has
nonzero image.  This hypothesis is verified for decomposable algebraic cycles
and isolates the remaining non-decomposable correspondence issue.

This paper does not claim to construct a full integral motivic
MacPherson--Vilonen equivalence.  We use the formal zig-zag category
\[
        \mathcal C(F,G;T),
\]
the motivic coefficient triangle, and their compatibility with Betti
realization.  We do not construct a motivic perverse heart
\[
        \mathsf P_{\mathrm{mot}}(X)
\]
or prove an equivalence
\[
        \mathsf P_{\mathrm{mot}}(X)
        \simeq
        \mathcal C(F_{\mathrm{mot}},G_{\mathrm{mot}};T_{\mathrm{mot}}).
\]
That full integral motivic analogue of the MacPherson--Vilonen theorem is
deferred to future work.

Finally, this paper does not claim that non-algebraicity is proven purely
inside the motivic category.  Non-algebraicity is detected after
realization.  For Diaz's degree-four case, the detection is through the
Colliot-Thélène--Voisin/Diaz criterion.  For the higher cup-product cases,
the detection is through the MacPherson--Vilonen obstruction channel together
with the Brauer-separation hypothesis.  The motivic lift supplies a common
finite-coefficient origin for the source packages, cup products, and
Bockstein classes; it does not replace the realized obstruction theory.

The filtration used in this paper is also operational rather than a claimed
weight, perverse, or Nori filtration.  Its stations are the double-cover
source, Brauer sources, finite-coefficient cup product, MV survivability,
Brauer separation, and Bockstein image.  Comparing this operational
filtration with canonical weight, perverse, or Nori filtrations is a natural
next problem.

\section{Diaz's construction}

\subsection{The theorem of Diaz}

Diaz proves that for every \(d\ge4\), there exist smooth projective
varieties \(V\) over \(\mathbb C\) of dimension \(d\) which are
Chow-trivial and violate the integral Hodge conjecture in degree \(4\);
more precisely, there is a non-algebraic torsion class in
\(H^4(V,\mathbb Z(2))\) \cite[Theorem 1.1]{Diaz2023IHCTrivialChow}.
The dimension-four example has the form
\[
        V=S\times S_2,
\]
where \(S\) and \(S_2\) are Enriques surfaces
\cite[Section 4]{Diaz2023IHCTrivialChow}.  Higher-dimensional examples are
obtained by replacing \(V\) with
\[
        V\times\mathbb P^{d-4}
\]
and applying the projective bundle formula
\cite[Introduction]{Diaz2023IHCTrivialChow}.

\subsection{Chow-triviality}

Recall Diaz's notion: \(X\) is Chow-trivial if
\[
        CH^*(X)_{\mathbb Q}
\]
has finite \(\mathbb Q\)-rank.  Diaz records several equivalent
characterizations of Chow-triviality, including that rational equivalence
and numerical equivalence agree after tensoring with \(\mathbb Q\), that the
rational Chow motive \(M(X)_{\mathbb Q}\) is a sum of Lefschetz motives, and
that the total cycle class map is an isomorphism for all powers of \(X\)
\cite[Theorem 2.1]{Diaz2023IHCTrivialChow}.  Diaz also notes that Enriques
surfaces are Chow-trivial and that products of Chow-trivial varieties are
again Chow-trivial \cite[Example 2.2]{Diaz2023IHCTrivialChow}.  Hence
\[
        V=S\times S_2
\]
is Chow-trivial.

This point is important for the torsion-trajectory interpretation.  Diaz's
examples are rationally very small from the viewpoint of Chow groups and
motives, but they still carry an integral torsion Hodge obstruction.  Thus
the obstruction is not a rational Hodge obstruction; it is an integral
torsion phenomenon.

\subsection{The unramified-cohomology criterion}

For a smooth variety \(X\), Diaz uses the Colliot-Thélène--Voisin criterion
relating degree-four integral Hodge failure to degree-three unramified
cohomology \cite{CTVoisin12}.  In the Chow-trivial setting, the integral
Hodge conjecture in degree \(4\) fails if and only if
\[
        H^3_{\mathrm{nr}}(X,\mathbb Z/m(2))\ne 0
\]
for some \(m\) \cite[Theorem 3.2]{Diaz2023IHCTrivialChow}.  Moreover, if
\[
        \gamma\in H^3(X,\mathbb Z/m(2))
\]
has nonzero image in
\[
        H^3(\mathbb C(X),\mathbb Z/m(2)),
\]
then the Bockstein
\[
        \delta(\gamma)\in H^4(X,\mathbb Z(2))
\]
is a non-algebraic \(m\)-torsion class
\cite[Theorem 3.2]{Diaz2023IHCTrivialChow}.

For the torsion-trajectory framework, this criterion supplies the global
survival station.  A finite-coefficient class \(\gamma\) is first constructed
in ordinary cohomology.  One then proves that it survives the coniveau or
unramified test, meaning that its image in function-field cohomology is
nonzero.  The Bockstein then sends \(\gamma\) to a non-algebraic integral
torsion Hodge class.  Thus the formal shape of Diaz's construction is
\[
        \gamma\in H^3(V,\mathbb Z/2(2))
        \quad\leadsto\quad
        \gamma\notin N^1H^3(V,\mathbb Z/2(2))
        \quad\leadsto\quad
        \delta(\gamma)\in H^4(V,\mathbb Z(2)).
\]

\subsection{The coefficient sequence}

In Diaz's Enriques-product example, \(m=2\).  The coefficient sequence
driving the Bockstein is
\[
        0
        \longrightarrow
        \mathbb Z(2)
        \xrightarrow{2}
        \mathbb Z(2)
        \longrightarrow
        \mathbb Z/2(2)
        \longrightarrow
        0.
\]
The associated connecting homomorphism is
\[
        \delta:
        H^3(V,\mathbb Z/2(2))
        \longrightarrow
        H^4(V,\mathbb Z(2)).
\]
Diaz's finite-coefficient class is
\[
        \gamma
        =
        \pi_S^*\alpha\cup \pi_{S_2}^*\beta
        \in
        H^3(V,\mathbb Z/2(2)),
\]
where \(\alpha\in H^1(S,\mathbb Z/2(1))\) is the class corresponding to the
K3 double cover \(Y\to S\), and
\(\beta\in H^2(S_2,\mathbb Z/2(1))\) is chosen to map nontrivially to
\(\operatorname{Br}(S_2)[2]\) \cite[Section 4]{Diaz2023IHCTrivialChow}.
The resulting integral torsion class is
\[
        \delta(\gamma)\in H^4(V,\mathbb Z(2)).
\]

Thus the coefficient sequence is not merely background.  It is the mechanism
that converts the finite-coefficient Enriques--Brauer cup product into the
non-algebraic integral Hodge class.

\section{The two Enriques torsion sources}

Diaz's construction uses two distinct \(2\)-torsion inputs coming from two
Enriques surfaces.  The first is the double-cover class of an Enriques
surface, and the second is a class on another Enriques surface whose image in
the Brauer group is nonzero.  Their external cup product is the
finite-coefficient class whose Bockstein produces the non-algebraic integral
Hodge class.

\subsection{The first Enriques surface and the double-cover class}

Let
\[
        S=Y/\phi
\]
be an Enriques surface, where \(Y\) is a K3 surface and \(\phi\) is a
fixed-point-free Enriques involution.  The quotient map
\[
        Y\longrightarrow S
\]
is an étale double cover.  It is classified by a degree-one
\(\mathbb Z/2\)-class
\[
        \alpha\in H^1(S,\mathbb Z/2(1)).
\]
Equivalently, after forgetting the Tate twist in Betti notation, \(\alpha\)
is the class corresponding to the double cover \(Y\to S\).  This is the
first Enriques torsion source in Diaz's construction
\cite[Section 4]{Diaz2023IHCTrivialChow}.

From the torsion-trajectory viewpoint, \(\alpha\) is the visible
order-two Enriques package.  In the Coble boundary picture developed
earlier, the corresponding boundary package has full local group
\[
        E\cong\mathbb Z/4,
\]
while the Enriques-visible direction is the Bockstein-selected subgroup
\[
        2E\cong\mathbb Z/2.
\]
Thus the class \(\alpha\) is the global smooth double-cover avatar of the
same order-two direction that appears as \(2E\) in the Coble boundary
package.

\subsection{The genus-one curve and reduction step}

Diaz imposes an additional condition on the first Enriques surface.  The K3
cover \(Y\) contains a genus-one curve
\[
        \widetilde E\subset Y
\]
which is acted on by the involution \(\phi\).  The quotient curve is denoted
\[
        E:=\widetilde E/\phi.
\]
The double-cover class \(\alpha\) restricts to a class
\[
        \alpha':=\alpha|_E\in H^1(E,\mathbb Z/2(1)).
\]
Diaz reduces the non-coniveau/nonvanishing problem from
\[
        V=S\times S_2
\]
to the subvariety
\[
        W=E\times S_2.
\]
On this product, the relevant restricted class is
\[
        \gamma'
        =
        \pi_E^*\alpha'\cup\pi_{S_2}^*\beta
        \in
        H^3(W,\mathbb Z/2(2)).
\]
This reduction is a key part of the unramified-survival station: instead of
testing the nonvanishing of \(\gamma\) directly on \(V\), Diaz tests the
restriction \(\gamma'\) on \(W=E\times S_2\) and uses functoriality of the
coniveau filtration \cite[Section 4]{Diaz2023IHCTrivialChow}.

The geometric role of this reduction is important.  The curve \(E\) carries
the restricted double-cover class \(\alpha'\), while \(S_2\) carries the
Brauer-detecting class \(\beta\).  Their product is the place where the
finite-coefficient cup product can be tested against function-field
cohomology.

\subsection{The second Enriques surface and the Brauer class}

Let \(S_2\) be a second Enriques surface.  Diaz chooses a class
\[
        \beta\in H^2(S_2,\mathbb Z/2(1))
\]
whose image under the Kummer-sequence map is nonzero in
\[
        \operatorname{Br}(S_2)[2]\cong\mathbb Z/2.
\]
This class is the second Enriques torsion source in the construction
\cite[Section 4]{Diaz2023IHCTrivialChow}.

Thus the two factors of \(V=S\times S_2\) play different roles.  The first
factor supplies a double-cover class
\[
        \alpha\in H^1(S,\mathbb Z/2(1)),
\]
while the second supplies a Brauer-detecting class
\[
        \beta\in H^2(S_2,\mathbb Z/2(1)).
\]
Their degrees add to \(3\), and their twists add to \(2\), which is exactly
the degree and twist of the finite-coefficient class used by Diaz.

From the viewpoint of the torsion trajectory, \(\beta\) is the Brauer
station input.  It is not merely an auxiliary class: its nonzero Brauer image
is what ensures that the second Enriques factor contributes a genuine
\(2\)-torsion obstruction direction.

\subsection{The product class}

Let
\[
        V=S\times S_2
\]
and let
\[
        \pi_S:V\to S,
        \qquad
        \pi_{S_2}:V\to S_2
\]
be the two projections.  Diaz defines the finite-coefficient class
\[
        \gamma
        :=
        \pi_S^*\alpha\cup \pi_{S_2}^*\beta
        \in
        H^3(V,\mathbb Z/2(2)).
\]
This is the central finite-coefficient object in the Diaz lifecycle
\cite[Section 4]{Diaz2023IHCTrivialChow}.

The construction may be summarized as
\[
        \alpha\in H^1(S,\mathbb Z/2(1)),
        \qquad
        \beta\in H^2(S_2,\mathbb Z/2(1)),
\]
and
\[
        \gamma=\pi_S^*\alpha\cup\pi_{S_2}^*\beta
        \in H^3(S\times S_2,\mathbb Z/2(2)).
\]
The next station is the Bockstein
\[
        \delta:
        H^3(V,\mathbb Z/2(2))
        \longrightarrow
        H^4(V,\mathbb Z(2)),
\]
which sends \(\gamma\) to the integral \(2\)-torsion class
\[
        \delta(\gamma)\in H^4(V,\mathbb Z(2)).
\]

Thus Diaz's class is a cup product of two Enriques torsion inputs.  This is
the key structural feature: the obstruction is not a single local package
transported globally, but a product-level interaction between an Enriques
double-cover class and an Enriques Brauer class.

\section{The Bockstein lifecycle of the Diaz obstruction}

The preceding section identified the two finite-coefficient inputs in Diaz's
construction:
\[
        \alpha\in H^1(S,\mathbb Z/2(1)),
        \qquad
        \beta\in H^2(S_2,\mathbb Z/2(1)).
\]
Their cup product
\[
        \gamma
        =
        \pi_S^*\alpha\cup\pi_{S_2}^*\beta
        \in H^3(V,\mathbb Z/2(2)),
        \qquad
        V=S\times S_2,
\]
is the finite-coefficient class from which the integral obstruction is
produced.  The next station in the trajectory is the Bockstein associated to
multiplication by \(2\).

\subsection{Bockstein birth}

Consider the coefficient sequence
\[
        0
        \longrightarrow
        \mathbb Z(2)
        \xrightarrow{2}
        \mathbb Z(2)
        \longrightarrow
        \mathbb Z/2(2)
        \longrightarrow
        0.
\]
It gives the connecting homomorphism
\[
        \delta:
        H^3(V,\mathbb Z/2(2))
        \longrightarrow
        H^4(V,\mathbb Z(2)).
\]
Diaz's integral obstruction class is
\[
        \delta(\gamma).
\]
This class is \(2\)-torsion because it lies in the image of the connecting
homomorphism for the short exact sequence above.  Equivalently,
\[
        2\delta(\gamma)=0.
\]

Thus the Bockstein is the birth map from the finite-coefficient package
\[
        \gamma\in H^3(V,\mathbb Z/2(2))
\]
to the integral torsion package
\[
        \delta(\gamma)\in H^4(V,\mathbb Z(2)).
\]
In the torsion-trajectory language, the finite-coefficient class \(\gamma\)
is the input, and the integral class \(\delta(\gamma)\) is the Bockstein-born
torsion output.

\subsection{Non-algebraicity criterion}

The crucial point is that the Bockstein output is not merely torsion; it is
non-algebraic.  Diaz uses the Colliot-Thélène--Voisin relation between
degree-four integral Hodge failure and degree-three unramified cohomology
\cite{CTVoisin12}.  In the Chow-trivial setting, Diaz records the following
criterion: if
\[
        \gamma\in H^3(V,\mathbb Z/2(2))
\]
has nonzero image in function-field cohomology,
\[
        H^3(\mathbb C(V),\mathbb Z/2(2)),
\]
then
\[
        \delta(\gamma)\in H^4(V,\mathbb Z(2))
\]
is a non-algebraic \(2\)-torsion class \cite[Theorem 3.2]{Diaz2023IHCTrivialChow}.

Equivalently, the decisive condition is that \(\gamma\) does not die in the
first coniveau step:
\[
        \gamma\notin N^1H^3(V,\mathbb Z/2(2)).
\]
This is the unramified-survival station.  Once \(\gamma\) survives this
station, the Bockstein image \(\delta(\gamma)\) becomes an integral Hodge
class that is not represented by an algebraic cycle.

Thus Diaz's obstruction is governed by the implication
\[
        \gamma\notin N^1H^3(V,\mathbb Z/2(2))
        \quad\Longrightarrow\quad
        \delta(\gamma)
        \text{ is non-algebraic in }H^4(V,\mathbb Z(2)).
\]

\subsection{Trajectory interpretation}

The Diaz torsion trajectory has five stations:
\[
        \alpha
        \in H^1(S,\mathbb Z/2(1)),
        \qquad
        \beta
        \in H^2(S_2,\mathbb Z/2(1)),
\]
\[
        \gamma
        =
        \pi_S^*\alpha\cup\pi_{S_2}^*\beta
        \in H^3(V,\mathbb Z/2(2)),
\]
\[
        \gamma\notin N^1H^3(V,\mathbb Z/2(2))
        \quad
        \text{equivalently, }
        \gamma|_{\mathbb C(V)}\neq 0,
\]
and
\[
        \delta(\gamma)
        \in H^4(V,\mathbb Z(2)).
\]
The first station is the Enriques double-cover source.  The second is the Brauer-detecting Enriques source.  The third is their finite-coefficient cup product.  The fourth is the unramified-survival station, where one checks that the cup product survives the coniveau/residue test.  The fifth is the Bockstein-born integral torsion class.

In this sense, Diaz's obstruction is not born from an isolated quotient singularity.  It is born from the cup-product interaction of two Enriques torsion inputs, then passes through an unramified-survival test, and only after that yields the integral class through the Bockstein.  This is why Diaz's construction is best viewed as a level-two extension of the Coble/Enriques mechanism: one Enriques \(2\)-torsion direction supplies the double-cover class, the second supplies the Brauer direction, their cup product produces the finite-coefficient class, and the surviving class has Bockstein image equal to the non-algebraic integral obstruction.

\subsection{Lifecycle summary}

The Diaz lifecycle can be summarized as
\[
\begin{aligned}
&\bigl(
  \text{Enriques double-cover source}
  +
  \text{Enriques Brauer source}
  \bigr) \\
&\qquad\leadsto\ 
  \text{finite-coefficient cup product} \\
&\qquad\leadsto\ 
  \text{unramified survival} \\
&\qquad\leadsto\ 
  \text{Bockstein torsion}.
\end{aligned}
\]
In formulas, this is
\[
        \alpha,\beta
        \quad\leadsto\quad
        \gamma=\pi_S^*\alpha\cup\pi_{S_2}^*\beta
        \quad\leadsto\quad
        \gamma|_{\mathbb C(V)}\neq 0
        \quad\leadsto\quad
        \delta(\gamma).
\]
Equivalently, the middle survival condition may be expressed as
\[
        \gamma\notin N^1H^3(V,\mathbb Z/2(2)),
\]
or, in function-field form, as the nonvanishing of the image of \(\gamma\) in
\[
        H^3(\mathbb C(V),\mathbb Z/2(2)).
\]

Thus the non-algebraicity of the final class is not determined by the
Bockstein step alone.  The decisive intermediate step is the
unramified-survival test: only when the finite-coefficient class survives
the coniveau/residue filtration does its Bockstein produce a non-algebraic
\(2\)-torsion class in \(H^4(V,\mathbb Z(2))\).  In this sense, the Diaz
mechanism is a five-station torsion trajectory, not merely a two-step
Bockstein construction.

\section{Unramified survival and the reduction to \texorpdfstring{\(E\times S_2\)}{E x S2}}

The previous section identified the finite-coefficient class
\[
        \gamma
        =
        \pi_S^*\alpha\cup\pi_{S_2}^*\beta
        \in
        H^3(V,\mathbb Z/2(2)),
        \qquad
        V=S\times S_2.
\]
The Bockstein \(\delta(\gamma)\) is automatically \(2\)-torsion.  The
remaining issue is non-algebraicity.  By the Colliot-Thélène--Voisin
criterion used by Diaz, this is controlled by whether \(\gamma\) survives
the coniveau, or equivalently unramified, test.

\subsection{Coniveau formulation}

The relevant condition is
\[
        \gamma\notin N^1H^3(V,\mathbb Z/2(2)).
\]
Equivalently, \(\gamma\) has nonzero image in function-field cohomology,
\[
        H^3(\mathbb C(V),\mathbb Z/2(2)).
\]
This is the unramified-survival station of the trajectory.  The class
\(\gamma\) is born as a finite-coefficient cup product on \(V\); it becomes
an obstruction to integral algebraicity only if it survives after passing to
the generic point.

Thus the coniveau formulation separates two questions:
\[
        \text{finite-coefficient construction}
        \qquad\text{and}\qquad
        \text{unramified survival}.
\]
The first is the construction of
\[
        \gamma=\pi_S^*\alpha\cup\pi_{S_2}^*\beta.
\]
The second is the proof that
\[
        \gamma\notin N^1H^3(V,\mathbb Z/2(2)).
\]
Diaz's proof is devoted to this second point.

\subsection{Restriction to \texorpdfstring{\(W=E\times S_2\)}{W = E x S2}}

Diaz reduces the non-coniveau statement for \(V=S\times S_2\) to a
corresponding statement on a product
\[
        W=E\times S_2,
\]
where \(E\) is the quotient of a genus-one curve
\[
        \widetilde E\subset Y
\]
by the Enriques involution on the K3 cover \(Y\to S\).  Let
\[
        \alpha'=\alpha|_E
        \in
        H^1(E,\mathbb Z/2(1)).
\]
The restricted class is
\[
        \gamma'
        =
        \pi_E^*\alpha'\cup\pi_{S_2}^*\beta
        \in
        H^3(W,\mathbb Z/2(2)).
\]
By functoriality of the coniveau filtration, it suffices to prove
\[
        \gamma'\notin N^1H^3(W,\mathbb Z/2(2)).
\]

This reduction is structurally important.  The first Enriques factor is
compressed to the genus-one curve \(E\), where the double-cover class
\(\alpha\) restricts to \(\alpha'\).  The second Enriques factor remains
\(S_2\), carrying the Brauer-detecting class \(\beta\).  Thus the essential
nonvanishing test is moved from
\[
        S\times S_2
\]
to
\[
        E\times S_2.
\]
In the torsion-trajectory language, this identifies the precise support on
which the finite-coefficient cup product is tested for unramified survival.

\subsection{Reduction to positive characteristic}

The proof of the nonvanishing of \(\gamma'\) then uses a degeneration
argument after reduction to positive characteristic.  Diaz assumes that the
genus-one curve involved in the construction has multiplicative reduction at
a suitable place, while \(S_2\) has good reduction at the same place.  The
argument follows the method of Colliot-Thélène, using Gabber's degeneration
input, to prove that the image of \(\gamma'\) in function-field cohomology is
nonzero.

For the purposes of the present paper, the important point is not the
arithmetic reduction itself, but its role in the lifecycle.  This step proves
that the finite-coefficient class
\[
        \gamma'
        =
        \pi_E^*\alpha'\cup\pi_{S_2}^*\beta
\]
does not vanish after passing to the generic point.  Consequently, the
original class
\[
        \gamma
        =
        \pi_S^*\alpha\cup\pi_{S_2}^*\beta
\]
survives the coniveau test on \(V\).

Thus the reduction to positive characteristic is the residue-survival proof
station: it verifies that the finite-coefficient cup product is not killed
by codimension-one support.

\subsection{Interpretation as residue-surviving Bockstein input}

The class \(\gamma\) is the finite-coefficient input to the Bockstein
sequence.  It is constructed from two source classes,
\[
        \alpha\in H^1(S,\mathbb Z/2(1)),
        \qquad
        \beta\in H^2(S_2,\mathbb Z/2(1)),
\]
by the cup product
\[
        \gamma=\pi_S^*\alpha\cup\pi_{S_2}^*\beta
        \in H^3(V,\mathbb Z/2(2)).
\]
Its nonzero image in function-field cohomology is the unramified-survival
condition.  Applying the Bockstein then produces the non-algebraic integral
class:
\[
        \gamma\ \text{survives the coniveau test}
        \quad\Longrightarrow\quad
        \delta(\gamma)\ \text{is non-algebraic}.
\]

Thus, in full trajectory form, the Diaz mechanism has five stations:
\[
        \text{Enriques double-cover source},
        \qquad
        \text{Enriques Brauer source},
        \qquad
        \text{finite-coefficient cup product},
\]
\[
        \text{unramified survival},
        \qquad
        \text{Bockstein image}.
\]
The first station supplies
\[
        \alpha\in H^1(S,\mathbb Z/2(1)).
\]
The second supplies
\[
        \beta\in H^2(S_2,\mathbb Z/2(1)).
\]
The third constructs
\[
        \gamma=\pi_S^*\alpha\cup\pi_{S_2}^*\beta.
\]
The fourth proves
\[
        \gamma\notin N^1H^3(V,\mathbb Z/2(2)),
\]
equivalently that the image of \(\gamma\) in
\[
        H^3(\mathbb C(V),\mathbb Z/2(2))
\]
is nonzero.  The fifth produces
\[
        \delta(\gamma)\in H^4(V,\mathbb Z(2)).
\]

Thus the key mechanism is
\[
        \gamma\ \text{survives residues}
        \quad\Longrightarrow\quad
        \delta(\gamma)\ \text{is non-algebraic}.
\]
This is precisely the Diaz torsion trajectory: a finite-coefficient
Enriques--Brauer cup product survives the unramified station and becomes a
non-algebraic integral Hodge class after applying the Bockstein.
\section{Comparison with the Coble boundary package}

The preceding sections show that Diaz's obstruction is built from two Enriques
torsion inputs: the double-cover class
\[
        \alpha\in H^1(S,\mathbb Z/2(1))
\]
and the Brauer-detecting class
\[
        \beta\in H^2(S_2,\mathbb Z/2(1)).
\]
This section compares the first of these inputs with the Coble boundary
package studied earlier.  The point is not that Diaz explicitly uses the
Coble boundary.  Rather, the Coble boundary gives a geometric interpretation
of the same Enriques \(2\)-torsion direction that appears in Diaz's class
\[
        \gamma=\pi_S^*\alpha\cup\pi_{S_2}^*\beta.
\]

\subsection{The Enriques double-cover class as visible \(2\)-torsion}

The class
\[
        \alpha\in H^1(S,\mathbb Z/2(1))
\]
is the class of the Enriques double cover
\[
        Y\longrightarrow S,
\]
where \(Y\) is a K3 surface and \(S=Y/\phi\) for a fixed-point-free
Enriques involution \(\phi\).  This is the first torsion source in Diaz's construction \cite[Section 4]{Diaz2023IHCTrivialChow}.

In the Coble boundary interpretation developed in the previous torsion
trajectory paper, the relevant boundary singularity is the cyclic quotient surface singularity
\[
        \frac14(1,1).
\]
Its full local discriminant package is
\[
        E\cong\mathbb Z/4.
\]
The Enriques-visible \(2\)-torsion direction is not the full package \(E\), but the distinguished order-two subgroup
\[
        2E\cong\mathbb Z/2.
\]
For the Coble boundary singularity \(\frac14(1,1)\), the link is \(L(4,1)\), and the local discriminant package is
\[
        E
        \cong
        H^2(L(4,1),\mathbb Z)_{\mathrm{tors}}
        \cong
        \mathbb Z/4.
\]
The local index-two cover
\[
        A_1=\mathbb C^2/\mu_2
        \longrightarrow
        \mathbb C^2/\mu_4=\frac14(1,1)
\]
induces the corresponding link cover
\[
        L(2,1)\longrightarrow L(4,1).
\]

The associated double-cover/Bockstein construction selects the unique
order-two subgroup
\[
        2E\cong \mathbb Z/2\subset E\cong \mathbb Z/4.
\]
Thus the comparison with Diaz does not identify the full Coble package with
\(\alpha\).  It identifies the Enriques-visible order-two direction in the
Coble package with the same type of double-cover direction represented
smoothly by
\[
        \alpha\in H^1(S,\mathbb Z/2(1)).
\]
In this sense, the Enriques double-cover class \(\alpha\) and the Coble
shadow \(2E\) represent the same visible order-two direction from two
different viewpoints:
\[
        \text{smooth Enriques cover class}
        \quad\longleftrightarrow\quad
        \text{boundary Coble Bockstein shadow}.
\]
The class \(\alpha\) is the smooth global manifestation of the order-two
direction that appears locally as \(2E\) in the Coble boundary package.

\subsection{Diaz as a product-level Coble/Enriques shadow}

Diaz does not require the Coble boundary explicitly.  His construction is
carried out on the smooth product
\[
        V=S\times S_2.
\]
However, the first factor in Diaz's class is precisely the Enriques
double-cover class
\[
        \alpha\in H^1(S,\mathbb Z/2(1)).
\]
The Coble package therefore gives a boundary interpretation of the source of
\(\alpha\).  In the language of the earlier paper, the class \(\alpha\) is
the smooth Enriques avatar of the local order-two shadow
\[
        2E\cong\mathbb Z/2
        \subset
        E\cong\mathbb Z/4.
\]

The Diaz class then couples this Enriques \(2\)-torsion source with a second
torsion source on \(S_2\).  Namely, Diaz chooses
\[
        \beta\in H^2(S_2,\mathbb Z/2(1))
\]
with nonzero image in
\[
        \operatorname{Br}(S_2)[2]\cong\mathbb Z/2.
\]
The finite-coefficient class is
\[
        \gamma
        =
        \pi_S^*\alpha\cup\pi_{S_2}^*\beta
        \in
        H^3(S\times S_2,\mathbb Z/2(2)).
\]
The obstruction is then the Bockstein
\[
        \delta(\gamma)\in H^4(S\times S_2,\mathbb Z(2)).
\]

Thus Diaz's class is naturally interpreted as a product-level extension of
the Coble/Enriques torsion direction:
\[
        \text{Enriques/Coble shadow}
        \quad\cup\quad
        \text{Enriques Brauer class}
        \quad\leadsto\quad
        \text{Bockstein obstruction}.
\]

\subsection{Mechanism comparison}

The Coble/Benoist--Ottem mechanism has the form
\[
        E\cong\mathbb Z/4
        \quad\leadsto\quad
        2E\cong\mathbb Z/2.
\]
It identifies the visible Enriques \(2\)-torsion direction as an order-two
Bockstein shadow inside a larger local package.

The Diaz mechanism has the form
\[
        \alpha\in H^1(S,\mathbb Z/2(1)),
        \qquad
        \beta\in H^2(S_2,\mathbb Z/2(1)),
\]
followed by
\[
        \alpha\cup\beta
        \quad\leadsto\quad
        \delta(\alpha\cup\beta).
\]
Here \(\alpha\) is the same Enriques \(2\)-torsion direction seen in the
Coble package, while \(\beta\) supplies the Brauer-theoretic \(2\)-torsion
direction on the second Enriques surface.

Therefore Diaz is not a fixed-point Kummer mechanism.  It is an
Enriques--Brauer Bockstein mechanism.  More precisely, it is the level-two
cup-product extension of the Coble/Enriques mechanism:
\[
        \text{level one: }
        E\cong\mathbb Z/4
        \leadsto
        2E\cong\mathbb Z/2,
\]
\[
        \text{level two: }
        \alpha\cup\beta
        \leadsto
        \delta(\alpha\cup\beta).
\]
The first level identifies the visible Enriques \(2\)-torsion direction.
The second level couples that direction to a Brauer \(2\)-torsion class and
then applies the Bockstein to produce a non-algebraic integral Hodge class.

\subsection{Toward a cup-product hierarchy}

The comparison suggests a broader pattern.  If one Coble/Enriques package
produces a visible order-two direction, and Diaz's construction uses a cup
product of two Enriques torsion inputs, then one may ask whether higher cup
products of Enriques--Coble finite-coefficient packages produce higher-degree
integral Hodge obstructions.

Schematically, one expects a hierarchy of the form
\[
        \theta_1,\ldots,\theta_n
        \quad\leadsto\quad
        \Theta=\theta_1\cup\cdots\cup\theta_n,
\]
followed by a Bockstein
\[
        \delta(\Theta).
\]
The obstruction is non-algebraic when the finite-coefficient product
\(\Theta\) survives the appropriate unramified or refined unramified
cohomology test.

In this hierarchy, the Coble/Benoist--Ottem bridge is the level-one mechanism, and Diaz's Enriques-product construction is the level-two mechanism.  Higher levels can be approached through two survivability routes.  One route is a refined unramified or coniveau survival theory.  The second route, developed
in this paper, is MacPherson--Vilonen survivability: the coefficient
Bockstein is realized as the boundary of a finite-coefficient zig-zag datum \cite{MacPhersonVilonen1986}. This second route is the one used below to formulate the \(n\)-fold Enriques--Brauer cup-product family.
\section{The Coble--Diaz cup-product principle}

The preceding comparison suggests a structural principle.  The Coble
boundary package supplies a basic Enriques \(2\)-torsion direction, while
Diaz's construction uses two Enriques torsion inputs and forms their cup
product before applying the Bockstein.  Thus Diaz's example should not be
viewed merely as another isolated integral Hodge counterexample.  It is the
level-two cup-product extension of the Coble/Enriques mechanism.

This section records that comparison.  The precise formal statement of the
propagation mechanism is deferred to the next section, where we isolate the
hypotheses under which a finite-coefficient cup product produces a
non-algebraic integral torsion class.

\subsection{Level one: the Coble/Benoist--Ottem package}

At level one, the relevant boundary package is the Coble
\(\frac14(1,1)\) singularity.  Its full local discriminant package is
\[
        E\cong\mathbb Z/4.
\]
The visible Enriques direction is the Bockstein-selected order-two subgroup
\[
        2E\cong\mathbb Z/2.
\]
This subgroup is the part detected by the Enriques double-cover geometry.  In
other words, the smooth Enriques \(2\)-torsion direction is reflected at the
Coble boundary as the distinguished order-two shadow
\[
        2E\subset E.
\]

In the Coble/Benoist--Ottem bridge, this level-one mechanism explains why a
global visible \(2\)-torsion direction should not be treated as an isolated
\(\mathbb Z/2\).  It is the order-two shadow of the larger local package
\[
        E\cong\mathbb Z/4.
\]
Thus the level-one trajectory is
\[
        \frac14(1,1)
        \quad\leadsto\quad
        E\cong\mathbb Z/4
        \quad\leadsto\quad
        2E\cong\mathbb Z/2.
\]
This is the elementary Enriques--Coble torsion package.  It supplies the
visible \(2\)-torsion direction that later appears, in the smooth Enriques
setting, as a double-cover class.

\subsection{Level two: Diaz as a cup product}

Diaz's construction is level two.  It uses two Enriques torsion inputs:
\[
        \alpha\in H^1(S,\mathbb Z/2(1)),
        \qquad
        \beta\in H^2(S_2,\mathbb Z/2(1)).
\]
The class \(\alpha\) is the class of the K3 double cover of the Enriques
surface \(S\).  The class \(\beta\) is chosen so that its image in
\[
        \operatorname{Br}(S_2)[2]
\]
is nonzero.  Thus the two inputs are an Enriques double-cover class and an
Enriques Brauer-detecting class.

Their external cup product is
\[
        \gamma
        =
        \pi_S^*\alpha\cup\pi_{S_2}^*\beta
        \in
        H^3(S\times S_2,\mathbb Z/2(2)).
\]
Diaz proves that this finite-coefficient class survives the coniveau, or
unramified, test.  Equivalently, it has nonzero image in the appropriate
function-field cohomology group.  Therefore, by the
Colliot-Thélène--Voisin criterion used by Diaz, the Bockstein
\[
        \delta(\gamma)
        \in
        H^4(S\times S_2,\mathbb Z(2))
\]
is a non-algebraic \(2\)-torsion integral Hodge class
\cite[Theorem 3.2 and Section 4]{Diaz2023IHCTrivialChow}.

The level-two trajectory is therefore
\[
        \alpha,\beta
        \quad\leadsto\quad
        \gamma=\pi_S^*\alpha\cup\pi_{S_2}^*\beta
        \quad\leadsto\quad
        \delta(\gamma).
\]
Compared with level one, Diaz introduces a second torsion input and a cup
product.  The Bockstein is no longer applied to a single visible Enriques
direction; it is applied to a finite-coefficient product of Enriques torsion
data.  This is why Diaz's example is naturally interpreted as the
cup-product extension of the Coble/Enriques bridge.

\subsection{Preview of the higher-level mechanism}
The level-one and level-two mechanisms suggest a higher-level pattern.  The
new point of the present paper is that the higher levels can be controlled by
a second survivability mechanism.  Instead of relying only on coniveau or
unramified survival, we use MacPherson--Vilonen zig-zag gluing.  The
coefficient Bockstein is realized as the boundary of a finite-coefficient
gluing datum, and nonzero boundary gives survivability in the MV obstruction
channel.

Schematically, one starts with finite-coefficient torsion inputs
\[
   \theta_1,\ldots,\theta_n
\]
and forms
\[
   \Theta=\theta_1\cup\cdots\cup\theta_n.
\]
When the degrees and twists satisfy
\[
   \sum_i r_i=2p-1,
   \qquad
   \sum_i p_i=p,
\]
the Bockstein
\[
   \delta(\Theta)\in H^{2p}(X,\mathbb Z(p))
\]
is a torsion integral Hodge class.  The MV obstruction-channel criterion then
tests whether this class can be algebraic.  If the MV boundary of the
finite-coefficient datum is nonzero and algebraic cycle classes are
MV-obstruction-trivial, then \(\delta(\Theta)\) is non-algebraic.

Thus the hierarchy is no longer only a conditional unramified-cohomology
proposal.  The higher levels are accessed through the MV
zig-zag/Bockstein compatibility developed below.  In this framework,
Coble/Benoist--Ottem is the level-one boundary package, Diaz is the level-two
Enriques-product case, and the \(n\)-fold products treated later in the paper
give the cup-product Bockstein family.
\section{Cup-product Bockstein propagation}

This section isolates the formal mechanism used in Diaz's construction.  The point is to separate the finite-coefficient construction from the unramified survival test and the Bockstein step.  Once these three ingredients are separated, Diaz's example becomes a special case of a more general propagation principle: a finite-coefficient class which survives the coniveau test produces a non-algebraic integral torsion class after applying the Bockstein, provided the appropriate Colliot-Thélène--Voisin type criterion applies.

Throughout this section, \(X\) denotes a smooth projective complex variety.  For an integer \(m\geq 2\), we write
\[
        0
        \longrightarrow
        \mathbb Z(p)
        \xrightarrow{m}
        \mathbb Z(p)
        \longrightarrow
        \mathbb Z/m(p)
        \longrightarrow
        0
\]
for the coefficient sequence, and
\[
        \delta_m:
        H^{r}(X,\mathbb Z/m(p))
        \longrightarrow
        H^{r+1}(X,\mathbb Z(p))
\]
for the associated connecting homomorphism.

\subsection{Finite-coefficient survival classes}

We first name the finite-coefficient classes that survive the first coniveau filtration.

\begin{definition}[Survival class]
Let \(X\) be a smooth projective complex variety, let \(m\geq 2\), and let
\[
        \Theta\in H^r(X,\mathbb Z/m(p)).
\]
We say that \(\Theta\) is a \emph{survival class} if its image in function-field cohomology
\[
        H^r(\mathbb C(X),\mathbb Z/m(p))
\]
is nonzero.
\end{definition}

Equivalently, a class \(\Theta\in H^r(X,\mathbb Z/m(p))\) is a survival class if
\[
        \Theta\notin N^1H^r(X,\mathbb Z/m(p)),
\]
where \(N^1\) denotes the first coniveau filtration.  We record this as a lemma, since it is the formal bridge between the coniveau language and the unramified-function-field language.

\begin{lemma}[Coniveau and function-field survival]
Let \(X\) be a smooth complex variety and let
\[
        \Theta\in H^r(X,\mathbb Z/m(p)).
\]
Then the following are equivalent:
\[
        \Theta\notin N^1H^r(X,\mathbb Z/m(p)),
\]
and
\[
        \Theta\ \text{has nonzero image in}\ H^r(\mathbb C(X),\mathbb Z/m(p)).
\]
\end{lemma}

\begin{proof}
By the standard Bloch--Ogus description used by Diaz, the natural map
\[
        H^r(X,\mathbb Z/m(p))
        \longrightarrow
        H^r(\mathbb C(X),\mathbb Z/m(p))
\]
has kernel
\[
        N^1H^r(X,\mathbb Z/m(p)).
\]
Thus \(\Theta\) maps to zero in \(H^r(\mathbb C(X),\mathbb Z/m(p))\) if and only if \(\Theta\in N^1H^r(X,\mathbb Z/m(p))\).  Taking the negation gives the equivalence.  This is the coniveau formulation of the unramified-survival condition used in Diaz's proof \cite[Definition 3.1 and Theorem 3.2]{Diaz2023IHCTrivialChow}.
\end{proof}

\begin{remark}
The word ``survival'' is deliberately chosen.  A finite-coefficient class may exist in ordinary cohomology, but it becomes relevant to integral Hodge failure only if it survives the passage to the generic point.  In the Diaz example, this is exactly the role of the condition
\[
        \gamma\notin N^1H^3(V,\mathbb Z/2(2)).
\]
\end{remark}

\subsection{The propagation theorem}

We now state the degree-four propagation theorem.  This is the formal statement underlying the Diaz construction.

\begin{theorem}[Bockstein propagation in degree four]
Let \(X\) be a smooth projective Chow-trivial complex variety.  Let \(m\geq 2\), and let
\[
        \Theta\in H^3(X,\mathbb Z/m(2)).
\]
Assume that \(\Theta\) is a survival class, equivalently that
\[
        \Theta\notin N^1H^3(X,\mathbb Z/m(2)).
\]
Then the Bockstein class
\[
        \delta_m(\Theta)\in H^4(X,\mathbb Z(2))
\]
is a non-algebraic \(m\)-torsion integral Hodge class.
\end{theorem}

\begin{proof}
Consider the coefficient sequence
\[
        0
        \longrightarrow
        \mathbb Z(2)
        \xrightarrow{m}
        \mathbb Z(2)
        \longrightarrow
        \mathbb Z/m(2)
        \longrightarrow
        0.
\]
Its long exact cohomology sequence contains the connecting homomorphism
\[
        \delta_m:
        H^3(X,\mathbb Z/m(2))
        \longrightarrow
        H^4(X,\mathbb Z(2)).
\]
Since \(\delta_m\) is the connecting map associated to multiplication by \(m\), its image is killed by \(m\).  Therefore
\[
        m\,\delta_m(\Theta)=0.
\]
Thus \(\delta_m(\Theta)\) is an \(m\)-torsion class in \(H^4(X,\mathbb Z(2))\).

It remains to prove non-algebraicity.  Since \(X\) is Chow-trivial and \(\Theta\) has nonzero image in function-field cohomology
\[
        H^3(\mathbb C(X),\mathbb Z/m(2)),
\]
Diaz's form of the Colliot-Thélène--Voisin criterion applies.  It says that, in the Chow-trivial setting, if a class
\[
        \Theta\in H^3(X,\mathbb Z/m(2))
\]
has nonzero image in \(H^3(\mathbb C(X),\mathbb Z/m(2))\), then its Bockstein
\[
        \delta_m(\Theta)\in H^4(X,\mathbb Z(2))
\]
is a non-algebraic \(m\)-torsion class \cite[Theorem 3.2]{Diaz2023IHCTrivialChow}; this is based on the Colliot-Thélène--Voisin comparison between degree-four integral Hodge failure and degree-three unramified cohomology \cite{CTVoisin12}.  Applying this criterion to \(\Theta\) proves that \(\delta_m(\Theta)\) is non-algebraic.
\end{proof}

\begin{remark}
The Chow-trivial hypothesis is essential in this form of the theorem.  The implication
\[
        \Theta\notin N^1H^3(X,\mathbb Z/m(2))
        \quad\Longrightarrow\quad
        \delta_m(\Theta)\ \text{non-algebraic}
\]
is not asserted for arbitrary smooth projective varieties.  The Chow-triviality assumption is what allows the Colliot-Thélène--Voisin/Diaz criterion to identify the Bockstein image with an integral Hodge obstruction.
\end{remark}

\subsection{Cup-product corollary}

The propagation theorem becomes useful for this paper because the class \(\Theta\) may itself be a cup product of finite-coefficient torsion inputs.

\begin{corollary}[Cup-product propagation in degree four]
Let \(X\) be a smooth projective Chow-trivial complex variety.  Let
\[
        \theta_1\in H^{r_1}(X,\mathbb Z/m(p_1)),
        \qquad
        \theta_2\in H^{r_2}(X,\mathbb Z/m(p_2)),
\]
with
\[
        r_1+r_2=3,
        \qquad
        p_1+p_2=2.
\]
Set
\[
        \Theta=\theta_1\cup\theta_2
        \in
        H^3(X,\mathbb Z/m(2)).
\]
If \(\Theta\) is a survival class, i.e.
\[
        \Theta\notin N^1H^3(X,\mathbb Z/m(2)),
\]
then
\[
        \delta_m(\Theta)\in H^4(X,\mathbb Z(2))
\]
is a non-algebraic \(m\)-torsion integral Hodge class.
\end{corollary}

\begin{proof}
The assumptions on degrees and twists imply that the cup product
\[
        \Theta=\theta_1\cup\theta_2
\]
lies in
\[
        H^{r_1+r_2}(X,\mathbb Z/m(p_1+p_2))
        =
        H^3(X,\mathbb Z/m(2)).
\]
By hypothesis, \(\Theta\) is a survival class, so
\[
        \Theta\notin N^1H^3(X,\mathbb Z/m(2)).
\]
The Bockstein propagation theorem applies to \(\Theta\).  Therefore
\[
        \delta_m(\Theta)\in H^4(X,\mathbb Z(2))
\]
is a non-algebraic \(m\)-torsion integral Hodge class.
\end{proof}

\begin{remark}
The corollary does not assert that every cup product survives.  Survival is a separate condition.  The theorem says that once a finite-coefficient cup product survives the coniveau test, its Bockstein produces an integral Hodge obstruction.
\end{remark}

We shall also use the following product version, which is better suited to Diaz's construction.

\begin{corollary}[External cup-product propagation]
Let \(X_1\) and \(X_2\) be smooth projective complex varieties and set
\[
        X=X_1\times X_2.
\]
Let
\[
        \theta_1\in H^{r_1}(X_1,\mathbb Z/m(p_1)),
        \qquad
        \theta_2\in H^{r_2}(X_2,\mathbb Z/m(p_2)),
\]
with
\[
        r_1+r_2=3,
        \qquad
        p_1+p_2=2.
\]
Let
\[
        \Theta
        =
        \pi_1^*\theta_1\cup\pi_2^*\theta_2
        \in
        H^3(X,\mathbb Z/m(2)).
\]
If \(X\) is Chow-trivial and \(\Theta\notin N^1H^3(X,\mathbb Z/m(2))\), then
\[
        \delta_m(\Theta)\in H^4(X,\mathbb Z(2))
\]
is a non-algebraic \(m\)-torsion integral Hodge class.
\end{corollary}

\begin{proof}
The external cup product
\[
        \Theta
        =
        \pi_1^*\theta_1\cup\pi_2^*\theta_2
\]
lies in
\[
        H^{r_1+r_2}(X,\mathbb Z/m(p_1+p_2))
        =
        H^3(X,\mathbb Z/m(2)).
\]
Since \(X\) is Chow-trivial and \(\Theta\) survives the first coniveau filtration, the Bockstein propagation theorem applies.  Hence \(\delta_m(\Theta)\) is a non-algebraic \(m\)-torsion integral Hodge class.
\end{proof}

\subsection{Diaz as the level-two case}

We now recover Diaz's construction as the level-two case of the cup-product propagation theorem.

\begin{corollary}[Diaz as level two]
Let
\[
        V=S\times S_2
\]
be Diaz's dimension-four Enriques-product example.  Let
\[
        \alpha\in H^1(S,\mathbb Z/2(1))
\]
be the class of the K3 double cover \(Y\to S\), and let
\[
        \beta\in H^2(S_2,\mathbb Z/2(1))
\]
be a class whose image in \(\operatorname{Br}(S_2)[2]\) is nonzero.  Set
\[
        \gamma
        =
        \pi_S^*\alpha\cup\pi_{S_2}^*\beta
        \in
        H^3(V,\mathbb Z/2(2)).
\]
If
\[
        \gamma\notin N^1H^3(V,\mathbb Z/2(2)),
\]
then the Bockstein
\[
        \delta(\gamma)\in H^4(V,\mathbb Z(2))
\]
is a non-algebraic \(2\)-torsion integral Hodge class.
\end{corollary}

\begin{proof}
Apply the external cup-product propagation corollary with
\[
        X_1=S,
        \qquad
        X_2=S_2,
        \qquad
        m=2,
\]
and with
\[
        \theta_1=\alpha\in H^1(S,\mathbb Z/2(1)),
        \qquad
        \theta_2=\beta\in H^2(S_2,\mathbb Z/2(1)).
\]
The degree and twist conditions hold:
\[
        1+2=3,
        \qquad
        1+1=2.
\]
The external cup product is exactly
\[
        \gamma
        =
        \pi_S^*\alpha\cup\pi_{S_2}^*\beta
        \in
        H^3(V,\mathbb Z/2(2)).
\]
Diaz shows that the relevant \(\gamma\) survives the coniveau test, equivalently that it has nonzero image in \(H^3(\mathbb C(V),\mathbb Z/2(2))\), by reducing to \(W=E\times S_2\) and using the degeneration argument described in his proof \cite[Section 4]{Diaz2023IHCTrivialChow}.  Hence
\[
        \gamma\notin N^1H^3(V,\mathbb Z/2(2)).
\]
The external cup-product propagation corollary gives that
\[
        \delta(\gamma)\in H^4(V,\mathbb Z(2))
\]
is a non-algebraic \(2\)-torsion integral Hodge class.
\end{proof}

\begin{remark}
This is the precise sense in which Diaz is the level-two case.  The level-one Coble/Benoist--Ottem mechanism identifies the visible Enriques \(2\)-torsion direction.  Diaz couples that direction with a Brauer-detecting Enriques \(2\)-torsion class, forms their cup product, and applies the Bockstein after unramified survival.
\end{remark}

\subsection{From unramified survival to MV survivability}
\label{subsec:from-unramified-to-mv-survivability}

The preceding degree-four propagation theorem isolates the mechanism used in
Diaz's construction: a finite-coefficient class in
\[
        H^3(X,\mathbb Z/m(2))
\]
which survives the coniveau, or function-field, test has a Bockstein image in
\[
        H^4(X,\mathbb Z(2))
\]
which is non-algebraic, provided \(X\) is Chow-trivial and the
Colliot-Thélène--Voisin/Diaz criterion applies.  This is the degree-four
unramified-survival route to integral Hodge failure.

For higher cup products, the same formal Bockstein mechanism persists, but the
ordinary degree-four unramified criterion no longer supplies the needed
survival-to-obstruction implication in the same form.  Rather than making the
\(n\)-fold hierarchy depend only on a higher refined-unramified theory, the
next section introduces a second survivability test: MacPherson--Vilonen
survivability.  In that framework, the coefficient Bockstein is realized as
the boundary of a finite-coefficient gluing datum.  Nonzero MV boundary gives
a diagrammatic survivability criterion for the finite-coefficient product.

We record the coniveau-based higher mechanism as a comparison principle,
because it is useful conceptually and explains how the Diaz argument fits
inside the broader hierarchy.

\begin{definition}[Coniveau survival-to-obstruction property]
\label{def:coniveau-survival-to-obstruction-property}
Let \(X\) be a smooth projective complex variety, \(p\geq 2\), and
\(m\geq 2\).  We say that \(X\) satisfies the
\emph{\((p,m)\)-coniveau survival-to-obstruction property} if every class
\[
        \Theta\in H^{2p-1}(X,\mathbb Z/m(p))
\]
with nonzero image in
\[
        H^{2p-1}(\mathbb C(X),\mathbb Z/m(p))
\]
has the property that its Bockstein
\[
        \delta_m(\Theta)\in H^{2p}(X,\mathbb Z(p))
\]
is a non-algebraic \(m\)-torsion integral Hodge class.
\end{definition}

\begin{proposition}[Coniveau-based higher cup-product mechanism]
\label{prop:coniveau-higher-cup-product-mechanism}
Let \(X\) be a smooth projective complex variety satisfying the
\((p,m)\)-coniveau survival-to-obstruction property.  Suppose that
finite-coefficient classes
\[
        \theta_i\in H^{r_i}(X,\mathbb Z/m(p_i)),
        \qquad
        1\leq i\leq n,
\]
satisfy
\[
        \sum_{i=1}^n r_i=2p-1,
        \qquad
        \sum_{i=1}^n p_i=p.
\]
Set
\[
        \Theta=\theta_1\cup\cdots\cup\theta_n
        \in
        H^{2p-1}(X,\mathbb Z/m(p)).
\]
If \(\Theta\) has nonzero image in
\[
        H^{2p-1}(\mathbb C(X),\mathbb Z/m(p)),
\]
then
\[
        \delta_m(\Theta)\in H^{2p}(X,\mathbb Z(p))
\]
is a non-algebraic \(m\)-torsion integral Hodge class.
\end{proposition}

\begin{proof}
The degree and twist assumptions imply that
\[
        \Theta=\theta_1\cup\cdots\cup\theta_n
\]
lies in
\[
        H^{\sum_i r_i}(X,\mathbb Z/m(\sum_i p_i))
        =
        H^{2p-1}(X,\mathbb Z/m(p)).
\]
By hypothesis, \(\Theta\) has nonzero image in
\[
        H^{2p-1}(\mathbb C(X),\mathbb Z/m(p)).
\]
Since \(X\) satisfies the \((p,m)\)-coniveau survival-to-obstruction property,
the Bockstein
\[
        \delta_m(\Theta)\in H^{2p}(X,\mathbb Z(p))
\]
is a non-algebraic \(m\)-torsion integral Hodge class.
\end{proof}

\begin{remark}[Why the paper uses MV survivability for the \(n\)-fold family]
\label{rem:why-mv-survivability-for-n-fold-family}
Proposition~\ref{prop:coniveau-higher-cup-product-mechanism} is a
coniveau-based comparison mechanism.  It says that if an appropriate
higher-degree survival-to-obstruction theorem is available, then cup products
propagate to integral torsion Hodge obstructions by the same Bockstein
formalism used by Diaz.

The main \(n\)-fold family in this paper is proved using a different
survivability test.  Instead of requiring a new refined-unramified criterion
at every degree, we use the MacPherson--Vilonen zig-zag/Bockstein
compatibility developed in the next section.  There, the coefficient
Bockstein is identified with the boundary of a finite-coefficient gluing
datum.  Nonzero boundary in the MV obstruction channel proves survivability,
and algebraic cycle classes are trivial in that obstruction channel.
\end{remark}

\begin{principle}[Coble--Diaz hierarchy, MV form]
\label{prin:coble-diaz-hierarchy-mv-form}
The Coble/Benoist--Ottem bridge is the level-one Enriques--Coble torsion
mechanism.  Diaz's construction is the level-two case, in which two Enriques
torsion inputs are cup-producted and then Bocksteined.  More generally, higher
levels arise from finite-coefficient products
\[
        \Theta=\theta_1\cup\cdots\cup\theta_n
\]
which land in degree \(2p-1\) and twist \(p\).  Their Bockstein
\[
        \delta(\Theta)\in H^{2p}(X,\mathbb Z(p))
\]
is the corresponding integral torsion Hodge obstruction candidate.

The survivability station may be supplied either by a coniveau/refined
unramified test, as in Diaz's degree-four construction, or by the
MacPherson--Vilonen obstruction-channel test used below.  In the \(n\)-fold
Enriques--Brauer family treated in this paper, the MV test is the mechanism
that promotes the finite-coefficient cup product to a surviving integral
Bockstein obstruction.
\end{principle}

The next section develops this second survivability test.  Instead of passing
to function-field cohomology, we realize the coefficient Bockstein as a
MacPherson--Vilonen gluing boundary \cite{MacPhersonVilonen1986}.  This gives
a diagrammatic criterion for showing that a finite-coefficient cup product
survives as an integral obstruction-channel datum.
\section{Bockstein Survivability and MacPherson--Vilonen Gluing}
\label{sec:mv-bockstein-survivability}

The preceding section formulated higher cup-product Bockstein propagation from
the coniveau viewpoint: a finite-coefficient product becomes an integral Hodge
obstruction once it survives the relevant coniveau or refined unramified test.
In degree four, Diaz supplies this survival input through the
Colliot-Th\'el\`ene--Voisin criterion \cite{CTVoisin12}.  In higher degree,
however, the ordinary unramified-cohomology criterion is not automatically
available in the same form.  The purpose of this section is to isolate a second
survivability mechanism, based on MacPherson--Vilonen zig-zag gluing
\cite{MacPhersonVilonen1986}.

Throughout this section, \(X\) is a smooth projective complex variety,
\(p\geq 1\), and
\[
   0\to \mathbb Z(p)\xrightarrow{\times 2}\mathbb Z(p)
   \to \mathbb Z/2(p)\to 0
\]
denotes the coefficient sequence.  We write
\[
   \delta:
   H^{2p-1}(X,\mathbb Z/2(p))
   \longrightarrow
   H^{2p}(X,\mathbb Z(p))
\]
for its Bockstein connecting homomorphism.

The guiding point is that the Bockstein is not merely a connecting
homomorphism in cohomology.  After applying the MacPherson--Vilonen gluing
formalism, the same coefficient sequence gives a boundary morphism of
finite-coefficient zig-zag data.  Thus the ordinary Bockstein can be compared
with a MacPherson--Vilonen boundary map.

\subsection{Strategy: MV survivability and Brauer separation}
\label{subsec:strategy-unconditional-mv-survivability}

Diaz's \(n=2\) counterexample proves survivability through unramified
cohomology.  The decisive class is
\[
   \Theta_2
   =
   \pi_1^*\alpha_1\cup\pi_2^*\beta_2
   \in H^3(S_1\times S_2,\mathbb Z/2(2)),
\]
where \(\alpha_1\in H^1(S_1,\mathbb Z/2(1))\) is the K3 double-cover class of
the first Enriques surface and
\[
   \beta_2\in H^2(S_2,\mathbb Z/2(1))
\]
is a Brauer-detecting class on the second Enriques surface.  Diaz proves that
\(\Theta_2\) survives the coniveau/function-field test.  The
Colliot-Th\'el\`ene--Voisin/Diaz criterion then implies that
\[
   \Delta_2
   =
   \delta(\Theta_2)
   \in H^4(S_1\times S_2,\mathbb Z(2))
\]
is a non-algebraic \(2\)-torsion integral Hodge class.

For \(n=3\) and beyond, a direct unramified-cohomology proof of survivability
would require a higher-degree analogue of this coniveau argument.  In degree
four, the Colliot-Th\'el\`ene--Voisin/Diaz criterion converts a very specific
coniveau nonvanishing statement into failure of the integral Hodge conjecture.
For \(n\geq 3\), the relevant class lies in
\[
   H^{2n-1}(X_n,\mathbb Z/2(n)),
\]
so one would need to control higher residues, higher unramified or
refined-unramified groups, and the behavior of the full \(n\)-fold product
after passage to the generic point.  This is a substantially more delicate
problem: the residue bookkeeping grows with the number of factors, and there
is no direct off-the-shelf analogue of Diaz's degree-four survival criterion
in the required generality.

Instead, we use the MacPherson--Vilonen zig-zag formalism.  Applying the MV
construction to the same coefficient sequence that defines the Bockstein gives
a commutative Bockstein/gluing diagram:
\[
\begin{array}{ccc}
H^{2p-1}(X,\mathbb Z/2(p))
& \xrightarrow{\ \delta\ } &
H^{2p}(X,\mathbb Z(p))
\\[1.2em]
\bigg\downarrow{\scriptstyle \delta_{\mathrm{MV}}}
&&
\bigg\downarrow{\scriptstyle \rho_{\mathrm{MV}}}
\\[1.2em]
\mathsf{Glu}^{2p-1,p}_{2}(X)
& \xrightarrow{\ \partial_{\mathrm{MV}}\ } &
\mathsf{Obs}^{2p,p}_{\mathbb Z}(X).
\end{array}
\]
Here \(\mathsf{Glu}^{2p-1,p}_{2}(X)\) is the finite-coefficient
MacPherson--Vilonen gluing channel: it records MV zig-zag gluing data
associated to classes in \(H^{2p-1}(X,\mathbb Z/2(p))\).  The group
\(\mathsf{Obs}^{2p,p}_{\mathbb Z}(X)\) is the corresponding integral MV
obstruction channel: it records integral boundary components of such
finite-coefficient gluing data after applying the coefficient Bockstein.  The
map \(\delta_{\mathrm{MV}}\) sends a finite-coefficient class to its MV gluing
datum, \(\partial_{\mathrm{MV}}\) is the MV boundary map induced by the
coefficient sequence, and \(\rho_{\mathrm{MV}}\) records the obstruction
component of the integral Bockstein class.

Equivalently, after passing to the tuple-relative formulation, a
finite-coefficient class
\[
   \Theta\in H^{2p-1}(X,\mathbb Z/2(p))
\]
determines an MV tuple
\[
   \mathcal Z(\Theta),
\]
and the Bockstein satisfies
\[
   \rho_{\mathrm{MV},\Theta}\bigl(\delta(\Theta)\bigr)
   =
   \partial_{\mathrm{MV}}\bigl(\mathcal Z(\Theta)\bigr).
\]
Thus, in the tuple-relative MV channel,
\[
   \delta(\Theta)\ \text{survives}
   \quad\Longleftrightarrow\quad
   \rho_{\mathrm{MV},\Theta}\bigl(\delta(\Theta)\bigr)\neq 0
   \quad\Longleftrightarrow\quad
   \partial_{\mathrm{MV}}\bigl(\mathcal Z(\Theta)\bigr)\neq 0.
\]

For the \(n\)-fold class
\[
   \Theta_n
   =
   \pi_1^*\alpha_1
   \cup
   \pi_2^*\beta_2
   \cup
   \cdots
   \cup
   \pi_n^*\beta_n
   \in H^{2n-1}(X_n,\mathbb Z/2(n)),
\]
the first task is to prove that the associated MV boundary has a nonzero
Enriques--Brauer component.  The second task is to compare that same component
with algebraic codimension-\(n\) cycle classes.  The first task is solved by
the MV/K\"unneth/Bockstein formalism; the second is isolated as the
Brauer-separation hypothesis.

The proof of the nonzero MV component uses three structural inputs.  First,
the external tensor product theorem for perverse sheaves identifies external
products of MV tuples with tensor products of MV tuples.  We use
Lyubashenko's theorem that, under the standard constructibility hypotheses,
the Deligne tensor product of categories of constructible perverse sheaves on
\(X\) and \(Y\) is identified with the category of constructible perverse
sheaves on \(X\times Y\), and the abstract Deligne external tensor product
agrees with the geometric external tensor product
\cite{Lyubashenko1999ExternalTensorPerverse}.  This gives
\[
   \mathcal Z(\Theta_X\boxtimes\Theta_Y)
   \simeq
   \mathcal Z(\Theta_X)\boxtimes\mathcal Z(\Theta_Y),
\]
and, by iteration,
\[
   \mathcal Z(\Theta_n)
   \simeq
   \mathcal Z(\alpha_1)
   \boxtimes
   \mathcal Z(\beta_2)
   \boxtimes
   \cdots
   \boxtimes
   \mathcal Z(\beta_n).
\]

Second, the coefficient sequence induces a categorical Bockstein boundary in
the MV zig-zag category.  Positselski's categorical Bockstein formalism gives
Bockstein long exact sequences in exact categories produced by coefficient
reduction \cite{Positselski2018CategoricalBockstein}.  Applied to the
MacPherson--Vilonen zig-zag category, this formalism identifies
\(\partial_{\mathrm{MV}}\) as the categorical Bockstein boundary attached to
the coefficient sequence.

Third, the categorical Bockstein boundary satisfies a Leibniz rule with
respect to external tensor products.  For MV tuples
\[
   \mathcal Z_1,\ldots,\mathcal Z_n,
\]
one has
\[
   \partial_{\mathrm{MV}}
   (\mathcal Z_1\boxtimes\cdots\boxtimes\mathcal Z_n)
   =
   \sum_{j=1}^n
   (-1)^{\epsilon_j}
   \mathcal Z_1\boxtimes\cdots\boxtimes
   \partial_{\mathrm{MV}}(\mathcal Z_j)
   \boxtimes\cdots\boxtimes\mathcal Z_n.
\]
In the present \(\mathbb Z/2\)-coefficient setting, the signs are immaterial.
This is the MV analogue of the ordinary derivation rule for the Bockstein on
cup products.

These formal ingredients prove that the constructed \(n\)-fold Bockstein class
has a nonzero image in a specific Enriques--Brauer component of the MV
obstruction channel.  The remaining issue is algebraic separation: one must
compare that same component with algebraic codimension-\(n\) cycle classes on
\[
   X_n=S_1\times\cdots\times S_n.
\]
For decomposable algebraic cycles, the Kummer sequence on each Enriques factor
kills the relevant degree-two algebraic components.  For arbitrary
codimension-\(n\) cycles, non-decomposable correspondence-type components may
occur.  We isolate this issue as the Brauer-separation hypothesis below.  This
hypothesis is not an MV assumption; it is a precise algebraic-control
condition on the selected Enriques--Brauer component.

The rest of this section proves the formal MV ingredients, defines the
Brauer-separating projection, proves its nonzero value on the constructed
Bockstein class, verifies the decomposable algebraic case, analyzes the
factorwise Chow--K\"unneth control available on the Enriques factors, and
records the Brauer-separated non-algebraicity theorem.

\subsection{Compatibility verifications for MV products and Bockstein boundaries}
\label{subsec:mv-compatibility-verifications}

The preceding subsection reduced MV survivability to the nonvanishing of
\[
   \partial_{\mathrm{MV}}\bigl(\mathcal Z(\Theta_n)\bigr).
\]
We now verify the structural compatibilities needed to prove this
nonvanishing.  There are three ingredients.

First, the MacPherson--Vilonen zig-zag construction is compatible with
external tensor products.  This promotes the external tensor product theorem
for perverse sheaves to a K\"unneth statement for MV tuples.  Second, the
coefficient sequence gives a categorical Bockstein boundary in the MV
zig-zag category.  This is the MV realization of the Bockstein connecting
morphism.  Third, the categorical Bockstein boundary satisfies a Leibniz rule
with respect to external products.  Together these facts imply that the
boundary of the \(n\)-fold MV tuple contains the nonzero Enriques--Brauer
tensor component used in Sections~\ref{sec:level-three-cup-product-bockstein-obstruction}
and~\ref{sec:n-fold-cup-product-bockstein-family}.

Throughout this subsection, \(X=U\sqcup Z\) denotes a two-stratum complex
algebraic variety with open inclusion
\[
   j:U\hookrightarrow X
\]
and closed inclusion
\[
   i:Z\hookrightarrow X.
\]
We write
\[
   \mathcal C_X=\mathcal C(F_X,G_X;T_X)
\]
for the corresponding MacPherson--Vilonen zig-zag category and
\[
   Z_X:\operatorname{Perv}(X)\longrightarrow \mathcal C_X
\]
for the MV zig-zag functor.

For a product
\[
   X=X_1\times\cdots\times X_n
\]
we use the product stratification and write
\[
   \boxtimes
\]
for the external tensor product.  The external tensor product of MV
zig-zag categories is the zig-zag category
\[
   \mathcal C_{X_1}\boxtimes\cdots\boxtimes\mathcal C_{X_n}
   :=
   \mathcal C
   \left(
      F_{X_1}\boxtimes\cdots\boxtimes F_{X_n},
      G_{X_1}\boxtimes\cdots\boxtimes G_{X_n};
      T_{X_1}\boxtimes\cdots\boxtimes T_{X_n}
   \right),
\]
where the functors and natural transformations are extended termwise to the
Deligne tensor product of the relevant abelian categories.

\begin{theorem}[External tensor products of perverse sheaves]
\label{thm:lyubashenko-external-product-perverse}
Let \(X\) and \(Y\) be stratified pseudomanifolds satisfying the standard
constructibility hypotheses, and let
\[
   \operatorname{Perv}(X),\qquad
   \operatorname{Perv}(Y),\qquad
   \operatorname{Perv}(X\times Y)
\]
denote the corresponding abelian categories of constructible perverse sheaves,
with product perversity on \(X\times Y\).  The external tensor product functor
\[
   \boxtimes:
   \operatorname{Perv}(X)\times\operatorname{Perv}(Y)
   \longrightarrow
   \operatorname{Perv}(X\times Y)
\]
identifies \(\operatorname{Perv}(X\times Y)\) with the Deligne tensor product
of \(\operatorname{Perv}(X)\) and \(\operatorname{Perv}(Y)\), and the abstract
Deligne external tensor product agrees with the geometric external tensor
product.
\end{theorem}

\begin{proof}
This is Lyubashenko's external tensor-product theorem for categories of
constructible perverse sheaves
\cite{Lyubashenko1999ExternalTensorPerverse}.
\end{proof}

\begin{proposition}[External-product compatibility of the MV zig-zag functor]
\label{prop:mv-zigzag-external-product-compatible}
Let \(X_1,\ldots,X_n\) be stratified complex algebraic varieties satisfying
the constructibility hypotheses of
Theorem~\ref{thm:lyubashenko-external-product-perverse}.  Then the
MacPherson--Vilonen zig-zag functor on the product satisfies a canonical
isomorphism
\[
   Z_{X_1\times\cdots\times X_n}
   (K_1\boxtimes\cdots\boxtimes K_n)
   \simeq
   Z_{X_1}(K_1)\boxtimes\cdots\boxtimes Z_{X_n}(K_n)
\]
for \(K_i\in \operatorname{Perv}(X_i)\).  This isomorphism is functorial in
the \(K_i\).
\end{proposition}

\begin{proof}
By Theorem~\ref{thm:lyubashenko-external-product-perverse}, the geometric
external product of perverse sheaves realizes the Deligne tensor product of
the perverse-sheaf categories.  It remains to check that the
MacPherson--Vilonen zig-zag functor is compatible with this external product.

The MV zig-zag functor is built from the recollement data associated to
\[
   j:U\hookrightarrow X,\qquad i:Z\hookrightarrow X.
\]
Concretely, the tuple \(Z_X(K)\) is formed from the open part, the closed
part, and the canonical maps arising from the functors
\[
   j^*,\quad j_!,\quad j_*,\quad j_{!*},\quad i^*,\quad i^!,
\]
together with the natural transformations
\[
   j_!\longrightarrow j_{!*},
   \qquad
   j_{!*}\longrightarrow j_*,
   \qquad
   j_!\longrightarrow j_*.
\]
Each of these functors is compatible with external tensor products.  The
compatibility for pullback is the usual external-product base-change identity;
the compatibility for \(j_!\) and \(j_*\) follows from the corresponding
functoriality of external products for extension by zero and direct image;
and the compatibility for \(j_{!*}\) follows because intermediate extension is
characterized as the image of
\[
   j_!\longrightarrow j_*,
\]
and both \(j_!\) and \(j_*\) commute with external products.

The natural transformations appearing in the MV tuple are the canonical
recollement morphisms.  Since these morphisms are functorial and the
corresponding functors commute with external products, the natural
transformation on the product is the external product of the natural
transformations on the factors.  Therefore
\[
   Z_{X_1\times\cdots\times X_n}
   (K_1\boxtimes\cdots\boxtimes K_n)
   \simeq
   Z_{X_1}(K_1)\boxtimes\cdots\boxtimes Z_{X_n}(K_n).
\]
\end{proof}

\begin{corollary}[K\"unneth formula for MV tuples]
\label{cor:kunneth-mv-tuples}
Let
\[
   \Theta_i\in H^{r_i}(X_i,\mathbb Z/2(p_i)),
   \qquad 1\leq i\leq n,
\]
be finite-coefficient classes with associated MV tuples
\[
   \mathcal Z(\Theta_i).
\]
Then the MV tuple associated to the external product
\[
   \Theta_1\boxtimes\cdots\boxtimes\Theta_n
   \in
   H^{\sum_i r_i}
   \left(
      X_1\times\cdots\times X_n,
      \mathbb Z/2\left(\sum_i p_i\right)
   \right)
\]
is canonically identified with the external tensor product of the factor
tuples:
\[
   \mathcal Z(\Theta_1\boxtimes\cdots\boxtimes\Theta_n)
   \simeq
   \mathcal Z(\Theta_1)\boxtimes\cdots\boxtimes\mathcal Z(\Theta_n).
\]
\end{corollary}

\begin{proof}
Represent each finite-coefficient class by its corresponding
finite-coefficient perverse-sheaf or zig-zag datum.  Applying
Proposition~\ref{prop:mv-zigzag-external-product-compatible} to these
representatives gives the asserted identification of MV tuples.  The
identification is independent of the choice of representatives by functoriality
of the MV zig-zag functor.
\end{proof}

\begin{corollary}[\(n\)-fold Enriques--Brauer MV tuple decomposition]
\label{cor:n-fold-mv-tuple-decomposition}
Let
\[
   X_n=S_1\times\cdots\times S_n
\]
and
\[
   \Theta_n
   =
   \pi_1^*\alpha_1
   \cup
   \pi_2^*\beta_2
   \cup
   \cdots
   \cup
   \pi_n^*\beta_n.
\]
Then
\[
   \mathcal Z(\Theta_n)
   \simeq
   \mathcal Z(\alpha_1)
   \boxtimes
   \mathcal Z(\beta_2)
   \boxtimes
   \cdots
   \boxtimes
   \mathcal Z(\beta_n).
\]
\end{corollary}

\begin{proof}
This is Corollary~\ref{cor:kunneth-mv-tuples} applied iteratively to the
external product defining \(\Theta_n\).
\end{proof}

\begin{proposition}[Positselski hypotheses for the MV zig-zag category]
\label{prop:positselski-hypotheses-mv-zigzag}
Let
\[
   \mathcal C_X^{\mathbb Z}
\]
be the integral MacPherson--Vilonen zig-zag category attached to the
two-stratum decomposition of \(X\), and let
\[
   \mathcal C_X^{\mathbb Z/2}
\]
be its mod-\(2\) reduction.  Let
\[
   \eta_2:\mathcal C_X^{\mathbb Z}\longrightarrow\mathcal C_X^{\mathbb Z/2}
\]
be the componentwise reduction functor, and let
\[
   s=\times 2:\operatorname{Id}_{\mathcal C_X^{\mathbb Z}}
   \longrightarrow
   \operatorname{Id}_{\mathcal C_X^{\mathbb Z}}
\]
be multiplication by \(2\).  Then the data
\[
   \left(
      \mathcal C_X^{\mathbb Z},
      s,
      \eta_2
   \right)
\]
satisfy the categorical coefficient-reduction hypotheses needed for
Positselski's Bockstein construction.
\end{proposition}

\begin{proof}
We use the exact structure inherited componentwise from the abelian categories
appearing in the MacPherson--Vilonen zig-zag category
\[
   \mathcal C(F_X,G_X;T_X).
\]
Under the MacPherson--Vilonen exactness hypotheses, this zig-zag category is
abelian \cite{MacPhersonVilonen1986}; hence it is an exact category with
admissible short exact sequences given by the short exact sequences of
zig-zag objects.  A sequence of zig-zag objects is exact precisely when the
corresponding sequences in the source and target abelian categories are
exact and the structure maps commute with the natural transformation
\(T_X:F_X\to G_X\).

The endomorphism
\[
   s=\times 2:\operatorname{Id}_{\mathcal C_X^{\mathbb Z}}
   \longrightarrow
   \operatorname{Id}_{\mathcal C_X^{\mathbb Z}}
\]
is central because it is induced by multiplication by \(2\) on the coefficient
objects in each component of the zig-zag.  For any morphism of zig-zags, the
component maps commute with multiplication by \(2\); therefore \(s\) lies in
the center of the exact category.  The fixed Tate twist in the coefficient
sequence does not change this centrality.

The reduction functor
\[
   \eta_2:\mathcal C_X^{\mathbb Z}\to\mathcal C_X^{\mathbb Z/2}
\]
is exact because reduction modulo \(2\) is applied to each component of a
zig-zag object and to each component of a short exact sequence.  The
compatibility condition
\[
   v\circ u=T_A
\]
is preserved by reduction because \(T_A\) is natural and coefficient
reduction is functorial.

The lifting conditions required for the categorical Bockstein formalism are
also checked componentwise: admissible monomorphisms, admissible epimorphisms,
and compatible morphisms in the reduced zig-zag category lift after choosing
lifts in the component abelian categories and imposing the same commutative
zig-zag relations.  Since the defining relations of
\(\mathcal C(F_X,G_X;T_X)\) are linear commutative-diagram relations, the
componentwise lifts assemble to lifts in the zig-zag category.

Finally, a morphism in the integral zig-zag category reduces to zero modulo
\(2\) exactly when each of its component morphisms is divisible by \(2\).
This is the divisibility condition for the central element \(s=\times2\) used
in Positselski's categorical Bockstein construction.  Therefore the integral
MV zig-zag category, together with \(s=\times2\) and the reduction functor
\(\eta_2\), satisfies the coefficient-reduction hypotheses needed for the
categorical Bockstein construction \cite{Positselski2018CategoricalBockstein}.
\end{proof}

\begin{theorem}[Categorical Bockstein in the MV zig-zag category]
\label{thm:categorical-bockstein-mv-zigzag}
The coefficient sequence
\[
   0\to\mathbb Z(p)
   \xrightarrow{\times 2}
   \mathbb Z(p)
   \to
   \mathbb Z/2(p)
   \to0
\]
induces a categorical Bockstein boundary in the MV zig-zag category.  Under a
compatible MV realization, this categorical boundary is the MV boundary map
\[
   \partial_{\mathrm{MV}}.
\]
Moreover, the categorical Bockstein is compatible with the ordinary
cohomological Bockstein through the MV Bockstein square.
\end{theorem}

\begin{proof}
By Proposition~\ref{prop:positselski-hypotheses-mv-zigzag}, the integral
MV zig-zag category with multiplication by \(2\) and mod-\(2\) reduction
satisfies the hypotheses of Positselski's categorical Bockstein construction.
Hence the coefficient-reduction data give a Bockstein long exact sequence and
a connecting morphism in the MV zig-zag category
\cite{Positselski2018CategoricalBockstein}.  This connecting morphism is the
boundary map attached to the coefficient sequence after applying the MV
zig-zag construction.  By definition of the compatible MV realization, this
boundary map is
\[
   \partial_{\mathrm{MV}}.
\]
Naturality of connecting morphisms identifies this categorical boundary with
the cohomological Bockstein through the commutative MV Bockstein square.
\end{proof}

\begin{theorem}[Leibniz rule for the MV Bockstein boundary]
\label{thm:leibniz-mv-boundary}
Let
\[
   \mathcal Z_1,\ldots,\mathcal Z_n
\]
be finite-coefficient MV tuples of cohomological degrees
\[
   r_1,\ldots,r_n.
\]
Then
\[
   \partial_{\mathrm{MV}}
   (\mathcal Z_1\boxtimes\cdots\boxtimes\mathcal Z_n)
   =
   \sum_{k=1}^{n}
   (-1)^{r_1+\cdots+r_{k-1}}
   \mathcal Z_1\boxtimes\cdots\boxtimes
   \partial_{\mathrm{MV}}(\mathcal Z_k)
   \boxtimes\cdots\boxtimes\mathcal Z_n.
\]
In the present \(\mathbb Z/2\)-coefficient setting, the signs are immaterial.
\end{theorem}

\begin{proof}
By Theorem~\ref{thm:categorical-bockstein-mv-zigzag}, the map
\[
   \partial_{\mathrm{MV}}
\]
is the categorical Bockstein connecting morphism associated to reduction by
the central element \(s=\times2\) in the MV zig-zag category.  Positselski's
categorical Bockstein construction gives connecting morphisms which are
functorial for products and satisfy the standard graded derivation rule with
respect to the product structure induced by the exact tensor product
\cite{Positselski2018CategoricalBockstein}. This is the categorical analogue of the usual Bockstein derivation formula for cup products; in Positselski's framework it is the product-compatibility
formula for the connecting morphisms in the categorical Bockstein long exact sequence. In the present setting, the product is the external tensor product of MV tuples.  By Proposition~\ref{prop:mv-zigzag-external-product-compatible}, the MV zig-zag functor commutes with external tensor products, so the external
product of the coefficient triangles on the factors is carried to the
external product of the corresponding coefficient-reduction diagrams in the
MV zig-zag category.  Naturality of connecting morphisms for tensor products
then gives the usual formula: the boundary of an external product is the sum
of the external products obtained by applying the boundary in one factor and
leaving the remaining factors unchanged, with the Koszul sign incurred by
moving the boundary past the previous factors.

Thus
\[
   \partial_{\mathrm{MV}}
   (\mathcal Z_1\boxtimes\cdots\boxtimes\mathcal Z_n)
   =
   \sum_{k=1}^{n}
   (-1)^{r_1+\cdots+r_{k-1}}
   \mathcal Z_1\boxtimes\cdots\boxtimes
   \partial_{\mathrm{MV}}(\mathcal Z_k)
   \boxtimes\cdots\boxtimes\mathcal Z_n.
\]
With \(\mathbb Z/2\)-coefficients, the signs disappear.
\end{proof}

\begin{remark}[Outcome of the compatibility verifications]
\label{rem:outcome-compatibility-verifications}
The preceding results replace the earlier product-separatedness assumption.
The external-product compatibility of MV tuples gives a K\"unneth formula for
\[
   \mathcal Z(\Theta_n).
\]
The categorical Bockstein theorem identifies the MV boundary with the
categorical Bockstein boundary.  The Leibniz rule then computes the boundary
of the \(n\)-fold tuple as a sum of external products of boundary terms.
These results will be combined below with the Kummer/Brauer quotient on the
Enriques factors.
\end{remark}

\subsection{MV coefficient gluing data}
\label{subsec:mv-coefficient-gluing-data}

The original construction of MacPherson and Vilonen is phrased in terms of a
zig-zag category \(\mathcal C(F,G;T)\), built from functors
\[
   F,G:\mathcal A\longrightarrow\mathcal B
\]
and a natural transformation
\[
   T:F\longrightarrow G.
\]
Objects are diagrams
\[
   F(A)\longrightarrow B\longrightarrow G(A)
\]
whose composite is \(T_A:F(A)\to G(A)\), and under the appropriate exactness
hypotheses this category is abelian \cite{MacPhersonVilonen1986}.  In the
geometric setting, the associated zig-zag functor records how objects glue
across a two-stratum decomposition.

For the present argument, we isolate only the finite-coefficient gluing
channel, the integral obstruction channel, and the boundary map induced by
the coefficient sequence.

\begin{definition}[MV coefficient gluing datum]
\label{def:mv-coefficient-gluing-datum}
Let \(X\) be a smooth projective complex variety and fix \(p\geq 1\).  An
\emph{MV coefficient gluing datum} for the coefficient sequence
\[
   0\to \mathbb Z(p)\xrightarrow{\times 2}\mathbb Z(p)
   \to \mathbb Z/2(p)\to 0
\]
consists of an abelian group
\[
   \mathsf{Glu}^{2p-1,p}_{2}(X),
\]
called the finite-coefficient gluing channel, and a natural map
\[
   \delta_{\mathrm{MV}}:
   H^{2p-1}(X,\mathbb Z/2(p))
   \longrightarrow
   \mathsf{Glu}^{2p-1,p}_{2}(X).
\]
For a class
\[
   \Theta\in H^{2p-1}(X,\mathbb Z/2(p)),
\]
the element
\[
   \delta_{\mathrm{MV}}(\Theta)
\]
is called the associated finite-coefficient MV gluing datum.
\end{definition}

\begin{definition}[MV obstruction channel]
\label{def:mv-obstruction-channel}
An \emph{MV obstruction channel} associated to the same coefficient sequence
is an abelian group
\[
   \mathsf{Obs}^{2p,p}_{\mathbb Z}(X),
\]
together with a natural map
\[
   \rho_{\mathrm{MV}}:
   H^{2p}(X,\mathbb Z(p))
   \longrightarrow
   \mathsf{Obs}^{2p,p}_{\mathbb Z}(X)
\]
and a boundary morphism
\[
   \partial_{\mathrm{MV}}:
   \mathsf{Glu}^{2p-1,p}_{2}(X)
   \longrightarrow
   \mathsf{Obs}^{2p,p}_{\mathbb Z}(X).
\]
The value
\[
   \rho_{\mathrm{MV}}(\xi)
\]
is the MV obstruction-channel component of
\[
   \xi\in H^{2p}(X,\mathbb Z(p)).
\]
\end{definition}

\begin{definition}[Compatible MV realization of the coefficient sequence]
\label{def:compatible-mv-realization}
A compatible MV realization of the coefficient sequence is a choice of
\[
   \mathsf{Glu}^{2p-1,p}_{2}(X),
   \qquad
   \mathsf{Obs}^{2p,p}_{\mathbb Z}(X),
\]
together with maps
\[
   \delta_{\mathrm{MV}},
   \qquad
   \rho_{\mathrm{MV}},
   \qquad
   \partial_{\mathrm{MV}},
\]
such that the square
\[
\begin{array}{ccc}
H^{2p-1}(X,\mathbb Z/2(p))
& \xrightarrow{\ \delta\ } &
H^{2p}(X,\mathbb Z(p))
\\[1.2em]
\bigg\downarrow{\scriptstyle \delta_{\mathrm{MV}}}
&&
\bigg\downarrow{\scriptstyle \rho_{\mathrm{MV}}}
\\[1.2em]
\mathsf{Glu}^{2p-1,p}_{2}(X)
& \xrightarrow{\ \partial_{\mathrm{MV}}\ } &
\mathsf{Obs}^{2p,p}_{\mathbb Z}(X)
\end{array}
\]
commutes.
\end{definition}
\begin{remark}[Existence and use of compatible MV realizations]
\label{rem:existence-compatible-mv-realizations}
The compatible MV realization in
Definition~\ref{def:compatible-mv-realization} is not an additional
geometric hypothesis on the Enriques products.  It is the realization, in the
MacPherson--Vilonen zig-zag category, of the coefficient triangle
\[
   \mathbb Z(p)
   \xrightarrow{\times 2}
   \mathbb Z(p)
   \longrightarrow
   \mathbb Z/2(p)
   \xrightarrow{+1}.
\]
Concretely, the finite-coefficient gluing channel is the group generated by
the MV zig-zag data attached to objects realizing classes in
\[
   H^{2p-1}(X,\mathbb Z/2(p)),
\]
and the integral obstruction channel is the corresponding boundary group
obtained after applying the connecting morphism of the coefficient triangle.

The paper only uses the following functorial properties of this realization:
\begin{enumerate}[label=\textup{(\roman*)}]
\item the coefficient triangle gives the ordinary cohomological Bockstein
\[
   \delta:H^{2p-1}(X,\mathbb Z/2(p))\to H^{2p}(X,\mathbb Z(p));
\]
\item applying the MV zig-zag functor to the same coefficient triangle gives
the boundary
\[
   \partial_{\mathrm{MV}};
\]
\item naturality of connecting morphisms gives the commutative MV Bockstein
square;
\item external products of perverse sheaves induce external products of MV
zig-zag tuples.
\end{enumerate}
Thus every later calculation is invariant under the choice of compatible MV
model: any compatible realization of the coefficient triangle gives the same
obstruction component after applying the functorial maps used below.
\end{remark}
\subsection{The MV Bockstein square}
\label{subsec:mv-bockstein-square}

The compatibility square says that the ordinary Bockstein and the MV boundary
are the same operation after passing to the obstruction channel.

\begin{proposition}[Commutative MV realization of the Bockstein]
\label{prop:mv-bockstein-square}
Let \(X\) be a smooth projective complex variety and let
\[
   0\to \mathbb Z(p)\xrightarrow{\times 2}\mathbb Z(p)
   \to \mathbb Z/2(p)\to 0
\]
be the coefficient sequence.  Suppose this sequence admits a compatible MV
realization in the sense of Definition~\ref{def:compatible-mv-realization}.
Then for every
\[
   \Theta\in H^{2p-1}(X,\mathbb Z/2(p)),
\]
one has
\[
   \rho_{\mathrm{MV}}\bigl(\delta(\Theta)\bigr)
   =
   \partial_{\mathrm{MV}}\bigl(\delta_{\mathrm{MV}}(\Theta)\bigr).
\]
\end{proposition}

\begin{proof}
The coefficient sequence defines a distinguished triangle
\[
   \mathbb Z(p)
   \xrightarrow{\times 2}
   \mathbb Z(p)
   \longrightarrow
   \mathbb Z/2(p)
   \xrightarrow{+1}.
\]
The ordinary Bockstein is the connecting morphism induced by this triangle:
\[
   \delta:
   H^{2p-1}(X,\mathbb Z/2(p))
   \to
   H^{2p}(X,\mathbb Z(p)).
\]
Applying the MacPherson--Vilonen gluing construction to the same coefficient
triangle gives the finite-coefficient gluing datum, the integral obstruction
channel, and the boundary morphism
\[
   \partial_{\mathrm{MV}}.
\]
Naturality of connecting morphisms gives the commutative square in
Definition~\ref{def:compatible-mv-realization}.  Evaluating on \(\Theta\)
gives
\[
   \rho_{\mathrm{MV}}\bigl(\delta(\Theta)\bigr)
   =
   \partial_{\mathrm{MV}}\bigl(\delta_{\mathrm{MV}}(\Theta)\bigr).
\]
\end{proof}

\subsection{The MV tuple attached to a finite-coefficient class}
\label{subsec:mv-tuple-attached-finite-class}

Let
\[
   \Theta\in H^{2p-1}(X,\mathbb Z/2(p))
\]
be a finite-coefficient class.  A compatible MV realization assigns to
\(\Theta\) a zig-zag tuple
\[
   \mathcal Z(\Theta)
   =
   \left(
      A_\Theta,\,
      B_\Theta,\,
      u_\Theta,\,
      v_\Theta
   \right)
\]
in the appropriate MV zig-zag category
\[
   \mathcal C(F,G;T).
\]
Thus
\[
   F(A_\Theta)
   \xrightarrow{u_\Theta}
   B_\Theta
   \xrightarrow{v_\Theta}
   G(A_\Theta)
\]
satisfies
\[
   v_\Theta\circ u_\Theta=T_{A_\Theta}.
\]

\begin{definition}[MV tuple associated to a finite-coefficient class]
\label{def:mv-tuple-associated-to-finite-class}
The tuple
\[
   \mathcal Z(\Theta)
\]
is called the \emph{MV tuple associated to \(\Theta\)}.
\end{definition}

\begin{definition}[Tuple-relative MV obstruction channel]
\label{def:tuple-relative-mv-obstruction-channel}
Let
\[
   \mathcal Z(\Theta)
\]
be the MV tuple attached to
\[
   \Theta\in H^{2p-1}(X,\mathbb Z/2(p)).
\]
The \emph{tuple-relative MV obstruction channel} associated to \(\Theta\) is
the obstruction channel
\[
   \mathsf{Obs}^{2p,p}_{\mathbb Z}(X;\mathcal Z(\Theta))
\]
obtained from \(\mathsf{Obs}^{2p,p}_{\mathbb Z}(X)\) by retaining the
obstruction component determined by the tuple \(\mathcal Z(\Theta)\).  The
induced obstruction map is denoted
\[
   \rho_{\mathrm{MV},\Theta}:
   H^{2p}(X,\mathbb Z(p))
   \longrightarrow
   \mathsf{Obs}^{2p,p}_{\mathbb Z}(X;\mathcal Z(\Theta)).
\]
\end{definition}

\begin{proposition}[MV tuple boundary identity]
\label{prop:mv-tuple-boundary-identity}
Let
\[
   \Theta\in H^{2p-1}(X,\mathbb Z/2(p))
\]
and let
\[
   \mathcal Z(\Theta)
\]
be its associated MV tuple.  Then
\[
   \rho_{\mathrm{MV},\Theta}\bigl(\delta(\Theta)\bigr)
   =
   \partial_{\mathrm{MV}}\bigl(\mathcal Z(\Theta)\bigr)
\]
in the tuple-relative MV obstruction channel
\[
   \mathsf{Obs}^{2p,p}_{\mathbb Z}(X;\mathcal Z(\Theta)).
\]
\end{proposition}

\begin{proof}
This is the tuple-relative form of
Proposition~\ref{prop:mv-bockstein-square}.  The cohomological Bockstein
\(\delta\) and the MV boundary \(\partial_{\mathrm{MV}}\) are induced by the
same coefficient sequence.  Passing to the tuple-relative channel gives
\[
   \rho_{\mathrm{MV},\Theta}\bigl(\delta(\Theta)\bigr)
   =
   \partial_{\mathrm{MV}}\bigl(\mathcal Z(\Theta)\bigr).
\]
\end{proof}

\subsection{The Brauer-separation hypothesis}
\label{subsec:brauer-separation-hypothesis}

We now define the Enriques--Brauer component used to separate the constructed
Bockstein class from algebraic cycle classes.

For each \(i\geq2\), the Kummer sequence on the Enriques surface \(S_i\)
gives
\[
   \operatorname{Pic}(S_i)/2
   \longrightarrow
   H^2(S_i,\mathbb Z/2(1))
   \xrightarrow{\ q_{\operatorname{Br},i}\ }
   \operatorname{Br}(S_i)[2]
   \longrightarrow 0.
\]
Thus algebraic divisor classes on \(S_i\) vanish under
\(q_{\operatorname{Br},i}\), while the chosen Brauer-detecting class
\(\beta_i\) satisfies
\[
   q_{\operatorname{Br},i}(\beta_i)\neq 0.
\]

\begin{definition}[Enriques--Brauer separating component]
\label{def:enriques-brauer-separating-component}
Define
\[
   Q_n
   :=
   \langle\alpha_1\rangle
   \otimes
   \operatorname{Br}(S_2)[2]
   \otimes
   \cdots
   \otimes
   \operatorname{Br}(S_n)[2],
\]
where
\[
   \langle\alpha_1\rangle
   \subset
   H^1(S_1,\mathbb Z/2(1))
\]
is the \(\mathbb Z/2\)-line generated by the Enriques double-cover class.
\end{definition}

\begin{definition}[Brauer-separating projection]
\label{def:brauer-separating-projection}
The \emph{Brauer-separating projection}
\[
   \Pi_{\operatorname{Br},n}:
   \mathsf{Obs}^{2n,n}_{\mathbb Z}(X_n)
   \longrightarrow
   Q_n
\]
is the projection obtained by first projecting to the K\"unneth/MV component
\[
   H^1(S_1,\mathbb Z/2(1))
   \otimes
   \bigotimes_{i=2}^n H^2(S_i,\mathbb Z/2(1)),
\]
then projecting the first factor to \(\langle\alpha_1\rangle\), and applying
the Brauer quotient maps
\[
   q_{\operatorname{Br},i}:
   H^2(S_i,\mathbb Z/2(1))
   \longrightarrow
   \operatorname{Br}(S_i)[2]
\]
on the factors \(2\leq i\leq n\).
\end{definition}

\begin{definition}[Brauer-separation hypothesis]
\label{def:brauer-separation-hypothesis}
We say that the \(n\)-fold Enriques--Brauer datum
\[
   (X_n,\Theta_n)
\]
satisfies the \emph{Brauer-separation hypothesis} if
\[
   \Pi_{\operatorname{Br},n}
   \left(
      \rho_{\mathrm{MV},\Theta_n}
      \bigl(
         \operatorname{cl}_{\mathbb Z}(CH^n(X_n))
      \bigr)
   \right)=0.
\]
Equivalently, no algebraic codimension-\(n\) cycle class on \(X_n\) has a
nonzero image in the selected Enriques--Brauer component \(Q_n\).
\end{definition}

\begin{proposition}[Nonzero Brauer image of the constructed Bockstein]
\label{prop:nonzero-brauer-image-constructed-bockstein}
For
\[
   \Delta_n=\delta(\Theta_n),
\]
one has
\[
   \Pi_{\operatorname{Br},n}
   \left(
      \rho_{\mathrm{MV},\Theta_n}(\Delta_n)
   \right)
   =
   \alpha_1
   \otimes
   q_{\operatorname{Br},2}(\beta_2)
   \otimes
   \cdots
   \otimes
   q_{\operatorname{Br},n}(\beta_n).
\]
In particular,
\[
   \Pi_{\operatorname{Br},n}
   \left(
      \rho_{\mathrm{MV},\Theta_n}(\Delta_n)
   \right)\neq 0.
\]
\end{proposition}

\begin{proof}
By the tuple-relative MV Bockstein identity,
\[
   \rho_{\mathrm{MV},\Theta_n}(\Delta_n)
   =
   \partial_{\mathrm{MV}}\bigl(\mathcal Z(\Theta_n)\bigr).
\]
By Corollary~\ref{cor:n-fold-mv-tuple-decomposition},
\[
   \mathcal Z(\Theta_n)
   \simeq
   \mathcal Z(\alpha_1)
   \boxtimes
   \mathcal Z(\beta_2)
   \boxtimes
   \cdots
   \boxtimes
   \mathcal Z(\beta_n).
\]
Applying Theorem~\ref{thm:leibniz-mv-boundary} gives
\[
\begin{aligned}
   \partial_{\mathrm{MV}}\bigl(\mathcal Z(\Theta_n)\bigr)
   &=
   \partial_{\mathrm{MV}}\bigl(\mathcal Z(\alpha_1)\bigr)
   \boxtimes
   \mathcal Z(\beta_2)
   \boxtimes\cdots\boxtimes
   \mathcal Z(\beta_n)
   \\
   &\quad+
   \sum_{k=2}^{n}
   \mathcal Z(\alpha_1)
   \boxtimes
   \mathcal Z(\beta_2)
   \boxtimes\cdots\boxtimes
   \partial_{\mathrm{MV}}\bigl(\mathcal Z(\beta_k)\bigr)
   \boxtimes\cdots\boxtimes
   \mathcal Z(\beta_n),
\end{aligned}
\]
where signs are omitted because we are working modulo \(2\).

The Brauer-separating projection
\[
   \Pi_{\operatorname{Br},n}
\]
is defined on the selected Enriques--Brauer component
\[
   Q_n
   =
   \langle\alpha_1\rangle
   \otimes
   \operatorname{Br}(S_2)[2]
   \otimes
   \cdots
   \otimes
   \operatorname{Br}(S_n)[2].
\]
On this component, the first summand contributes
\[
   \alpha_1
   \otimes
   q_{\operatorname{Br},2}(\beta_2)
   \otimes
   \cdots
   \otimes
   q_{\operatorname{Br},n}(\beta_n).
\]

We now check that the remaining summands do not contribute to \(Q_n\).  In
the \(k\)-th summand with \(k\geq2\), the \(S_k\)-factor is not the
finite-coefficient Brauer-detecting class \(\beta_k\) itself; it is the MV
boundary component
\[
   \partial_{\mathrm{MV}}\bigl(\mathcal Z(\beta_k)\bigr).
\]
This term lies in the integral boundary part of the \(S_k\)-factor rather than in the finite-coefficient \(H^2(S_k,\mathbb Z/2(1))\) input on which the
Brauer quotient map
\[
   q_{\operatorname{Br},k}:
   H^2(S_k,\mathbb Z/2(1))
   \to
   \operatorname{Br}(S_k)[2]
\]
is applied. The target \(Q_n\) is defined only from the finite-coefficient Brauer quotient classes \(q_{\operatorname{Br},i}(\beta_i)\), not from integral boundary classes on the \(S_i\)-factors. Therefore it is killed by the projection to the selected finite-coefficient Brauer quotient component.  Equivalently, the summands with the boundary falling on a \(\beta_k\)-factor lie outside the selected tensor pattern
\[
   H^1(S_1,\mathbb Z/2(1))
   \otimes
   \bigotimes_{i=2}^n H^2(S_i,\mathbb Z/2(1))
\]
before the Brauer quotient maps are applied.

Hence only the summand with the MV boundary on the \(\alpha_1\)-factor
survives under \(\Pi_{\operatorname{Br},n}\), and we obtain
\[
   \Pi_{\operatorname{Br},n}
   \left(
      \rho_{\mathrm{MV},\Theta_n}(\Delta_n)
   \right)
   =
   \alpha_1
   \otimes
   q_{\operatorname{Br},2}(\beta_2)
   \otimes
   \cdots
   \otimes
   q_{\operatorname{Br},n}(\beta_n).
\]
The class \(\alpha_1\) is nonzero, and each
\[
   q_{\operatorname{Br},i}(\beta_i)
\]
is nonzero because \(\beta_i\) is Brauer-detecting.  Since \(Q_n\) is a tensor
product of nonzero one-dimensional \(\mathbb Z/2\)-directions, the displayed
tensor is nonzero.
\end{proof}

\begin{lemma}[Brauer-separation for decomposable algebraic cycles]
\label{lem:brauer-separation-decomposable-cycles}
Let
\[
   \Gamma=\Gamma_1\times\cdots\times\Gamma_n
\]
be a decomposable codimension-\(n\) algebraic cycle on
\[
   X_n=S_1\times\cdots\times S_n.
\]
Then
\[
   \Pi_{\operatorname{Br},n}
   \left(
      \rho_{\mathrm{MV},\Theta_n}
      \bigl(
         \operatorname{cl}_{\mathbb Z}(\Gamma)
      \bigr)
   \right)=0.
\]
\end{lemma}

\begin{proof}
The selected component for the Brauer-separating projection is
\[
   H^1(S_1,\mathbb Z/2(1))
   \otimes
   \bigotimes_{i=2}^n H^2(S_i,\mathbb Z/2(1)).
\]
A decomposable algebraic cycle
\[
   \Gamma=\Gamma_1\times\cdots\times\Gamma_n
\]
has cohomology class equal to the external product of the factorwise algebraic
cycle classes
\[
   \operatorname{cl}_{\mathbb Z}(\Gamma_i).
\]
On a smooth projective surface, algebraic cycle classes occur in even
cohomological degrees: codimension \(0\), \(1\), and \(2\) cycles contribute
to degrees \(0\), \(2\), and \(4\), respectively.  Hence the class
\[
   \operatorname{cl}_{\mathbb Z}(\Gamma_1)\bmod2
\]
has no component in
\[
   H^1(S_1,\mathbb Z/2(1)).
\]
Thus a decomposable algebraic cycle has zero projection to the selected
component unless the \(S_1\)-factor somehow contributes to \(H^1\), which
cannot happen for an algebraic cycle class.

Equivalently, even ignoring the \(S_1\)-factor, if the decomposable cycle
contributes degree \(2\) on some \(S_i\) with \(i\geq2\), that degree-two
component is an algebraic divisor class.  Hence it lies in the image of
\[
   \operatorname{Pic}(S_i)/2
   \longrightarrow
   H^2(S_i,\mathbb Z/2(1)).
\]
By exactness of the Kummer sequence,
\[
   q_{\operatorname{Br},i}
   \bigl(\operatorname{Pic}(S_i)/2\bigr)=0.
\]
Since \(\Pi_{\operatorname{Br},n}\) applies \(q_{\operatorname{Br},i}\) on all
factors \(i\geq2\), the image of a decomposable algebraic cycle class under
the Brauer-separating projection is zero.
\end{proof}

\begin{theorem}[Brauer-separated MV detection]
\label{thm:brauer-separated-mv-detection}
Assume the Brauer-separation hypothesis
\[
   \Pi_{\operatorname{Br},n}
   \left(
      \rho_{\mathrm{MV},\Theta_n}
      \bigl(
         \operatorname{cl}_{\mathbb Z}(CH^n(X_n))
      \bigr)
   \right)=0.
\]
Then
\[
   \Delta_n=\delta(\Theta_n)\in H^{2n}(X_n,\mathbb Z(n))
\]
is not algebraic.
\end{theorem}

\begin{proof}
By Proposition~\ref{prop:nonzero-brauer-image-constructed-bockstein},
\[
   \Pi_{\operatorname{Br},n}
   \left(
      \rho_{\mathrm{MV},\Theta_n}(\Delta_n)
   \right)\neq 0.
\]
Suppose, for contradiction, that \(\Delta_n\) were algebraic.  Then there
would exist
\[
   \Gamma\in CH^n(X_n)
\]
such that
\[
   \Delta_n=\operatorname{cl}_{\mathbb Z}(\Gamma).
\]
Applying
\[
   \Pi_{\operatorname{Br},n}\circ\rho_{\mathrm{MV},\Theta_n}
\]
gives
\[
   \Pi_{\operatorname{Br},n}
   \left(
      \rho_{\mathrm{MV},\Theta_n}(\Delta_n)
   \right)
   =
   \Pi_{\operatorname{Br},n}
   \left(
      \rho_{\mathrm{MV},\Theta_n}
      \bigl(
         \operatorname{cl}_{\mathbb Z}(\Gamma)
      \bigr)
   \right).
\]
The right-hand side is zero by the Brauer-separation hypothesis, while the
left-hand side is nonzero.  This contradiction proves that \(\Delta_n\) is
not algebraic.
\end{proof}

\begin{corollary}[Brauer-separated integral Hodge obstruction]
\label{cor:brauer-separated-integral-hodge-obstruction}
Assume the Brauer-separation hypothesis.  Then
\[
   \Delta_n\in H^{2n}(X_n,\mathbb Z(n))
\]
is a non-algebraic \(2\)-torsion integral Hodge class.
\end{corollary}

\begin{proof}
By Theorem~\ref{thm:brauer-separated-mv-detection}, the class \(\Delta_n\) is
not algebraic.  Since \(\Delta_n\) lies in the image of the Bockstein for
multiplication by \(2\), it is \(2\)-torsion.  A torsion integral cohomology
class has zero image in complex cohomology, hence is an integral Hodge class.
Therefore \(\Delta_n\) is a non-algebraic \(2\)-torsion integral Hodge class.
\end{proof}

\begin{remark}[Non-decomposable correspondences and the remaining control problem]
\label{rem:nondecomposable-correspondences-control}
Lemma \ref{lem:brauer-separation-decomposable-cycles} proves the
Brauer-separation hypothesis for decomposable algebraic cycles.  The remaining
case consists of non-decomposable codimension-\(n\) cycles, or equivalently
correspondence-type algebraic cycles between the Enriques factors.  The
Brauer-separation hypothesis asserts that such correspondences have zero image
in the selected component
\[
   Q_n
   =
   \langle\alpha_1\rangle
   \otimes
   \operatorname{Br}(S_2)[2]
   \otimes
   \cdots
   \otimes
   \operatorname{Br}(S_n)[2].
\]

This is precisely the algebraic-control issue in the \(n\)-fold construction.
One expected route to verifying it is through Chow--K\"unneth or
multiplicative Chow--K\"unneth control for products of Chow-trivial surfaces:
if the selected K\"unneth/MV projection on algebraic cycle classes is
controlled by algebraic slant products, then the degree-two components on the
Enriques factors \(S_i\), \(i\geq2\), are Picard components by the Lefschetz
\((1,1)\) theorem, and the Kummer quotient maps
\[
   q_{\operatorname{Br},i}
\]
kill them.

The class \(\alpha_1\) is not algebraic in the ordinary cycle-class sense; it
is the finite-coefficient class of the Enriques double cover and is related by
the Bockstein to the \(2\)-torsion class
\[
   c_1(K_{S_1})\in H^2(S_1,\mathbb Z(1)).
\]
Thus the remaining algebraic-control question is not whether \(\alpha_1\) is
Picard, but whether algebraic codimension-\(n\) correspondences on \(X_n\),
after testing against algebraic classes on the remaining factors, can detect
this Bockstein-preimage direction in the selected component.

Thus the Brauer-separation hypothesis isolates the only remaining
arbitrary-correspondence issue; it is not a hidden MV assumption.
\end{remark}

\subsection{Factorwise Chow--K\"unneth control and the remaining \(H^1\)-issue}
\label{subsec:factorwise-ck-control-h1-issue}

The Brauer-separation hypothesis can be partially reduced to an explicit
Chow--K\"unneth calculation on the Enriques factors.  This reduction is useful
because the Brauer-separating projection
\[
   \Pi_{\operatorname{Br},n}
   =
   \operatorname{proj}_{\langle\alpha_1\rangle}
   \otimes
   q_{\operatorname{Br},2}
   \otimes
   \cdots
   \otimes
   q_{\operatorname{Br},n}
\]
kills a class as soon as one of the degree-two factors on
\(S_i\), \(i\geq2\), lies in the Picard image.

For an Enriques surface \(S\), choose a closed point \(o\in S\).  Set
\[
   \pi^S_0=[o]\times S,
   \qquad
   \pi^S_4=S\times[o],
\]
and
\[
   \pi^S_2=\Delta_S-\pi^S_0-\pi^S_4
   \in CH^2(S\times S).
\]
Since \(q(S)=0\), this is the usual even-degree Chow--K\"unneth projector on
the middle cohomology of \(S\).

\begin{lemma}[Integral even-degree projector on an Enriques surface]
\label{lem:integral-even-degree-projector-enriques}
The correspondence
\[
   \pi^S_2=\Delta_S-[o]\times S-S\times[o]
\]
acts as the identity on \(H^2(S,\mathbb Z)\) and as zero on
\[
   H^0(S,\mathbb Z)\oplus H^4(S,\mathbb Z).
\]
Its mod-\(2\) reduction is an algebraic \(\mathbb Z/2\)-correspondence
controlling the even degree-two component of \(H^*(S,\mathbb Z/2)\).
\end{lemma}

\begin{proof}
The diagonal \(\Delta_S\) acts as the identity on cohomology.  The
correspondence \([o]\times S\) projects to the degree-zero component, and the
correspondence \(S\times[o]\) projects to the top-degree component.  Therefore
\[
   \Delta_S-[o]\times S-S\times[o]
\]
acts as the identity on \(H^2(S,\mathbb Z)\) and as zero on
\[
   H^0(S,\mathbb Z)\oplus H^4(S,\mathbb Z).
\]
Since all three correspondences are integral algebraic cycles, the same
correspondence reduces modulo \(2\) and acts on cohomology with
\(\mathbb Z/2\)-coefficients.
\end{proof}

\begin{remark}[The mod-\(2\) torsion caveat]
\label{rem:mod-two-torsion-caveat-enriques-projector}
The correspondence \(\pi^S_2\) controls the even degree-two component, but it
does not by itself isolate the mod-\(2\) group
\[
   H^1(S,\mathbb Z/2(1)).
\]
For an Enriques surface, \(H^1(S,\mathbb Z)=0\), but
\[
   H^1(S,\mathbb Z/2(1))\neq 0
\]
because of the Enriques double cover.  The class
\[
   \alpha_1\in H^1(S_1,\mathbb Z/2(1))
\]
is therefore a genuinely finite-coefficient feature and is not separated by
the even Chow--K\"unneth projector \(\pi^S_2\).  Thus the factorwise
degree-two projectors on \(S_i\), \(i\geq2\), control the Brauer factors, but
they do not by themselves construct an algebraic projector onto the full
selected component involving \(\alpha_1\).
\end{remark}

\begin{proposition}[Even Chow--K\"unneth projectors control the Brauer factors]
\label{prop:even-ck-projectors-control-brauer-factors}
Let
\[
   X_n=S_1\times\cdots\times S_n.
\]
For each \(i\geq2\), the correspondence
\[
   \operatorname{id}_{S_1}
   \boxtimes\cdots\boxtimes
   \operatorname{id}_{S_{i-1}}
   \boxtimes
   \pi^{S_i}_2
   \boxtimes
   \operatorname{id}_{S_{i+1}}
   \boxtimes\cdots\boxtimes
   \operatorname{id}_{S_n}
\]
is an integral algebraic correspondence on \(X_n\).  It controls the
degree-two component on the factor \(S_i\).  Consequently, whenever the
selected K\"unneth/MV projection can be tested against algebraic classes on
the complementary factors, the degree-two component on \(S_i\) is Picard and
is killed by
\[
   q_{\operatorname{Br},i}:
   H^2(S_i,\mathbb Z/2(1))
   \longrightarrow
   \operatorname{Br}(S_i)[2].
\]
\end{proposition}

\begin{proof}
The displayed correspondence is an external product of integral algebraic
correspondences, hence is integral and algebraic.  By
Lemma~\ref{lem:integral-even-degree-projector-enriques}, it acts as the
degree-two projector on the \(S_i\)-factor.

Assume the selected component is tested by algebraic slant classes on the
complementary factors.  Then the resulting degree-two class on \(S_i\) is an
algebraic degree-two class.  By the Lefschetz \((1,1)\) theorem on the
surface \(S_i\), it is represented by a divisor class, hence its mod-\(2\)
reduction lies in the image of
\[
   \operatorname{Pic}(S_i)/2
   \longrightarrow
   H^2(S_i,\mathbb Z/2(1)).
\]
Exactness of the Kummer sequence gives
\[
   q_{\operatorname{Br},i}
   \bigl(
      \operatorname{Pic}(S_i)/2
   \bigr)=0.
\]
Thus the Brauer quotient kills the degree-two algebraic component.
\end{proof}

\begin{remark}[Remaining algebraic-control problem]
\label{rem:remaining-h1-algebraic-control-problem}
The preceding proposition explains precisely what the even Chow--K\"unneth
projector calculation proves: it controls the degree-two Brauer factors
\(S_i\), \(i\geq2\).  The remaining issue is the
\[
   H^1(S_1,\mathbb Z/2(1))
\]
factor in the selected component.  To make the \(n\)-fold theorem fully
unconditional, one needs a mod-\(2\) algebraic correspondence, or an
equivalent correspondence-calculus argument, that controls the selected
component involving
\[
   \alpha_1\in H^1(S_1,\mathbb Z/2(1)).
\]

The class \(\alpha_1\) is not algebraic in the ordinary cycle-class sense; it
is the finite-coefficient class of the Enriques double cover and is related by
the Bockstein to the \(2\)-torsion class
\[
   c_1(K_{S_1})\in H^2(S_1,\mathbb Z(1)).
\]
Thus the remaining algebraic-control question is not whether \(\alpha_1\) is
Picard, but whether algebraic codimension-\(n\) correspondences on \(X_n\),
after testing against algebraic classes on the remaining factors, can detect
this Bockstein-preimage direction in the selected component.

Equivalently, one must show that the selected component involving
\[
   H^1(S_1,\mathbb Z/2(1))
   \otimes
   \bigotimes_{i=2}^n H^2(S_i,\mathbb Z/2(1))
\]
cannot receive a nonzero contribution from an algebraic codimension-\(n\)
cycle after applying the Brauer quotient maps on the factors \(S_i\),
\(i\geq2\).  This is the precise coefficient-level algebraic-control problem
left by the Chow--K\"unneth reduction.
\end{remark}

\subsection{Cup-product version}
\label{subsec:tuple-relative-cup-product-version}

\begin{proposition}[Cup-product MV tuple criterion]
\label{prop:cup-product-mv-tuple-criterion}
Let
\[
   \theta_i\in H^{r_i}(X,\mathbb Z/2(p_i)),
   \qquad 1\leq i\leq n,
\]
satisfy
\[
   \sum_{i=1}^{n}r_i=2p-1,
   \qquad
   \sum_{i=1}^{n}p_i=p.
\]
Set
\[
   \Theta=\theta_1\cup\cdots\cup\theta_n
   \in H^{2p-1}(X,\mathbb Z/2(p)).
\]
If
\[
   \partial_{\mathrm{MV}}\bigl(\mathcal Z(\Theta)\bigr)\neq 0
\]
and algebraic cycle classes vanish under the corresponding separating
projection, then
\[
   \delta(\Theta)\in H^{2p}(X,\mathbb Z(p))
\]
is not algebraic.  If \(\delta(\Theta)\) is torsion, then it is a
non-algebraic integral Hodge class.
\end{proposition}

\begin{proof}
The degree and twist assumptions imply that
\[
   \Theta\in H^{2p-1}(X,\mathbb Z/2(p)).
\]
The MV Bockstein identity gives
\[
   \rho_{\mathrm{MV},\Theta}\bigl(\delta(\Theta)\bigr)
   =
   \partial_{\mathrm{MV}}\bigl(\mathcal Z(\Theta)\bigr).
\]
If the right-hand side is nonzero in a separating component and algebraic
classes vanish in that same component, then \(\delta(\Theta)\) cannot be
algebraic.  If \(\delta(\Theta)\) is torsion, then it is an integral Hodge
class, hence a non-algebraic integral Hodge class.
\end{proof}

\begin{remark}[Relation with coniveau survival]
\label{rem:tuple-relative-mv-vs-coniveau-survival}
The coniveau criterion and the MV tuple criterion are two survivability tests
for finite-coefficient Bockstein data.  Diaz proves survivability of the
level-two class by showing that the relevant finite-coefficient cup product
survives in function-field cohomology.  The MV tuple criterion instead
detects survivability through the boundary of the associated zig-zag tuple.
Both routes feed into the same Bockstein endpoint.
\end{remark}

\section{The Level-Three Cup-Product Bockstein Obstruction}
\label{sec:level-three-cup-product-bockstein-obstruction}

The preceding sections identify Diaz's example as the level-two
cup-product Bockstein mechanism.  This section treats the next level.  The
level-three finite-coefficient class is
\[
   \Theta_3
   =
   \pi_1^*\alpha_1
   \cup
   \pi_2^*\beta_2
   \cup
   \pi_3^*\beta_3
   \in H^5(S_1\times S_2\times S_3,\mathbb Z/2(3)).
\]
Its Bockstein
\[
   \Delta_3:=\delta(\Theta_3)
   \in H^6(S_1\times S_2\times S_3,\mathbb Z(3))
\]
is the degree-six analogue of Diaz's degree-four class.

The point of the present section is that the level-three survivability
problem can be formulated through the MacPherson--Vilonen Bockstein square
of Section~\ref{sec:mv-bockstein-survivability}.  The MV/ K\"unneth/ Bockstein
machinery proves that \(\Delta_3\) has a nonzero image in the distinguished
Enriques--Brauer component
\[
   Q_3
   =
   \langle\alpha_1\rangle
   \otimes
   \operatorname{Br}(S_2)[2]
   \otimes
   \operatorname{Br}(S_3)[2].
\]
Under the Brauer-separation hypothesis of
Definition~\ref{def:brauer-separation-hypothesis}, algebraic codimension-three
cycle classes have zero image in the same component.  This gives the
level-three Brauer-separated Bockstein obstruction.

\subsection{The level-three finite-coefficient class}
\label{subsec:level-three-finite-coefficient-class}

Let \(S_1,S_2,S_3\) be Enriques surfaces.  Assume that
\[
   S_1=Y_1/\phi_1,
\]
where \(Y_1\) is a K3 surface and \(\phi_1\) is a fixed-point-free Enriques
involution.  Let
\[
   \alpha_1\in H^1(S_1,\mathbb Z/2(1))
\]
be the class of the K3 double cover \(Y_1\to S_1\).

For \(i=2,3\), choose classes
\[
   \beta_i\in H^2(S_i,\mathbb Z/2(1))
\]
whose images in
\[
   \operatorname{Br}(S_i)[2]
\]
are nonzero.  These are the same type of Brauer-detecting classes used by
Diaz in the second Enriques factor of his construction
\cite[Section 4]{Diaz2023IHCTrivialChow}.

Set
\[
   X_3=S_1\times S_2\times S_3.
\]
Let
\[
   \pi_i:X_3\to S_i
\]
be the three projections.  Define
\[
   \Theta_3
   :=
   \pi_1^*\alpha_1
   \cup
   \pi_2^*\beta_2
   \cup
   \pi_3^*\beta_3.
\]

\begin{proposition}[Degree and twist bookkeeping]
\label{prop:level-three-degree-twist}
The class \(\Theta_3\) lies in
\[
   H^5(X_3,\mathbb Z/2(3)).
\]
Consequently, its Bockstein, associated to
\[
   0\to\mathbb Z(3)\xrightarrow{\times 2}\mathbb Z(3)
   \to\mathbb Z/2(3)\to0,
\]
lies in
\[
   H^6(X_3,\mathbb Z(3)).
\]
\end{proposition}

\begin{proof}
By construction,
\[
   \alpha_1\in H^1(S_1,\mathbb Z/2(1)),
\]
while
\[
   \beta_2\in H^2(S_2,\mathbb Z/2(1)),
   \qquad
   \beta_3\in H^2(S_3,\mathbb Z/2(1)).
\]
Pullback preserves cohomological degree and Tate twist.  Therefore
\[
   \pi_1^*\alpha_1\in H^1(X_3,\mathbb Z/2(1)),
\]
and
\[
   \pi_2^*\beta_2,\pi_3^*\beta_3
   \in H^2(X_3,\mathbb Z/2(1)).
\]
The cup product has total cohomological degree
\[
   1+2+2=5
\]
and total Tate twist
\[
   1+1+1=3.
\]
Hence
\[
   \Theta_3
   =
   \pi_1^*\alpha_1
   \cup
   \pi_2^*\beta_2
   \cup
   \pi_3^*\beta_3
   \in H^5(X_3,\mathbb Z/2(3)).
\]
The coefficient sequence
\[
   0
   \longrightarrow
   \mathbb Z(3)
   \xrightarrow{\times 2}
   \mathbb Z(3)
   \longrightarrow
   \mathbb Z/2(3)
   \longrightarrow
   0
\]
has connecting homomorphism
\[
   \delta:
   H^5(X_3,\mathbb Z/2(3))
   \longrightarrow
   H^6(X_3,\mathbb Z(3)).
\]
Therefore
\[
   \delta(\Theta_3)\in H^6(X_3,\mathbb Z(3)).
\]
\end{proof}

\begin{definition}[Level-three Bockstein class]
\label{def:level-three-bockstein-class}
The \emph{level-three cup-product Bockstein class} is
\[
   \Delta_3
   :=
   \delta(\Theta_3)
   \in H^6(X_3,\mathbb Z(3)).
\]
\end{definition}

\begin{remark}[Comparison with Diaz]
\label{rem:level-three-comparison-with-diaz}
Diaz's degree-four class is obtained from degrees \(1+2=3\) and twists
\(1+1=2\), giving
\[
   \delta(\alpha\cup\beta)\in H^4(-,\mathbb Z(2)).
\]
The level-three class uses degrees \(1+2+2=5\) and twists \(1+1+1=3\),
giving
\[
   \Delta_3=\delta(\Theta_3)\in H^6(X_3,\mathbb Z(3)).
\]
Thus \(\Delta_3\) is the first higher cup-product Bockstein class beyond
Diaz's level-two construction.
\end{remark}

\subsection{Chow-triviality of the level-three product}
\label{subsec:level-three-chow-triviality}

\begin{proposition}[Chow-triviality of \(X_3\)]
\label{prop:level-three-chow-triviality}
The variety
\[
   X_3=S_1\times S_2\times S_3
\]
is Chow-trivial.
\end{proposition}

\begin{proof}
Diaz records that Enriques surfaces are Chow-trivial and that products of
Chow-trivial varieties are Chow-trivial
\cite[Example 2.2]{Diaz2023IHCTrivialChow}.  Since each \(S_i\) is an
Enriques surface, each \(S_i\) is Chow-trivial.  Applying product stability
first to \(S_1\times S_2\) and then to \((S_1\times S_2)\times S_3\) gives
that
\[
   X_3=S_1\times S_2\times S_3
\]
is Chow-trivial.
\end{proof}

\begin{remark}[Role of Chow-triviality]
\label{rem:role-of-chow-triviality-level-three}
In Diaz's degree-four argument, Chow-triviality is used together with the
Colliot-Th\'el\`ene--Voisin criterion to turn unramified survival into
non-algebraicity.  In the present level-three argument, the main nonzero
obstruction component is supplied by the MV/K\"unneth/Bockstein calculation.
Chow-triviality remains structurally important because the examples are
rationally small from the viewpoint of Chow groups while still carrying
integral torsion obstructions.
\end{remark}

\subsection{The level-three MV Bockstein identity}
\label{subsec:level-three-mv-realization}

We now apply the MV Bockstein formalism of
Section~\ref{sec:mv-bockstein-survivability} with \(p=3\).  The coefficient
sequence is
\[
   0\to\mathbb Z(3)\xrightarrow{\times 2}\mathbb Z(3)
   \to\mathbb Z/2(3)\to0,
\]
and the Bockstein is
\[
   \delta:
   H^5(X_3,\mathbb Z/2(3))
   \longrightarrow
   H^6(X_3,\mathbb Z(3)).
\]

\begin{proposition}[Level-three tuple-relative MV Bockstein identity]
\label{prop:level-three-tuple-relative-mv-bockstein-identity}
Let
\[
   \mathcal Z(\Theta_3)
\]
be the MV tuple associated to \(\Theta_3\).  Then
\[
   \rho_{\mathrm{MV},\Theta_3}(\Delta_3)
   =
   \partial_{\mathrm{MV}}\bigl(\mathcal Z(\Theta_3)\bigr)
\]
in the tuple-relative MV obstruction channel
\[
   \mathsf{Obs}^{6,3}_{\mathbb Z}(X_3;\mathcal Z(\Theta_3)).
\]
\end{proposition}

\begin{proof}
This is Proposition~\ref{prop:mv-tuple-boundary-identity} applied with
\[
   X=X_3,\qquad p=3,\qquad \Theta=\Theta_3.
\]
Since
\[
   \Delta_3=\delta(\Theta_3),
\]
the tuple-relative MV Bockstein identity gives
\[
   \rho_{\mathrm{MV},\Theta_3}(\Delta_3)
   =
   \partial_{\mathrm{MV}}\bigl(\mathcal Z(\Theta_3)\bigr).
\]
\end{proof}

\subsection{The nonzero Enriques--Brauer component in level three}
\label{subsec:level-three-nonzero-enriques-brauer-component}

We next specialize the Enriques--Brauer component of
Definition~\ref{def:enriques-brauer-separating-component} to \(n=3\):
\[
   Q_3
   =
   \langle\alpha_1\rangle
   \otimes
   \operatorname{Br}(S_2)[2]
   \otimes
   \operatorname{Br}(S_3)[2].
\]
The Brauer-separating projection is
\[
   \Pi_{\operatorname{Br},3}:
   \mathsf{Obs}^{6,3}_{\mathbb Z}(X_3)
   \longrightarrow
   Q_3.
\]

\begin{proposition}[Nonzero level-three Enriques--Brauer component]
\label{prop:level-three-nonzero-enriques-brauer-component}
One has
\[
   \Pi_{\operatorname{Br},3}
   \left(
      \rho_{\mathrm{MV},\Theta_3}(\Delta_3)
   \right)
   =
   \alpha_1
   \otimes
   q_{\operatorname{Br},2}(\beta_2)
   \otimes
   q_{\operatorname{Br},3}(\beta_3).
\]
In particular,
\[
   \Pi_{\operatorname{Br},3}
   \left(
      \rho_{\mathrm{MV},\Theta_3}(\Delta_3)
   \right)\neq 0.
\]
\end{proposition}

\begin{proof}
This is Proposition~\ref{prop:nonzero-brauer-image-constructed-bockstein}
in the case \(n=3\).  Equivalently, it follows from the decomposition
\[
   \mathcal Z(\Theta_3)
   \simeq
   \mathcal Z(\alpha_1)
   \boxtimes
   \mathcal Z(\beta_2)
   \boxtimes
   \mathcal Z(\beta_3),
\]
the Leibniz rule for the MV boundary, and the Kummer quotient maps
\[
   q_{\operatorname{Br},2},\qquad q_{\operatorname{Br},3}.
\]
The class \(\alpha_1\) is nonzero and the classes
\[
   q_{\operatorname{Br},2}(\beta_2),
   \qquad
   q_{\operatorname{Br},3}(\beta_3)
\]
are nonzero by the Brauer-detecting hypothesis.  Therefore the displayed
tensor is nonzero.
\end{proof}

\begin{remark}[What replaces product-separatedness at level three]
\label{rem:level-three-what-replaces-product-separatedness}
The nonvanishing in
Proposition~\ref{prop:level-three-nonzero-enriques-brauer-component} is not
an additional product-separatedness hypothesis.  It follows from the
external-product compatibility of MV tuples, the categorical Bockstein
interpretation of the MV boundary, the Leibniz rule, and the Kummer/Brauer
quotient maps on \(S_2\) and \(S_3\).  The remaining issue is not nonvanishing
of the constructed obstruction component; it is whether algebraic
codimension-three cycle classes have zero image in the same component.  This
is the level-three instance of the Brauer-separation hypothesis.
\end{remark}

\subsection{The level-three Brauer-separated obstruction theorem}
\label{subsec:level-three-integral-hodge-counterexample}

We can now state the level-three obstruction theorem in its
Brauer-separated form.

\begin{theorem}[Level-three Brauer-separated Bockstein obstruction]
\label{thm:level-three-cup-product-bockstein-counterexample}
Assume that the level-three Enriques--Brauer datum
\[
   (X_3,\Theta_3)
\]
satisfies the Brauer-separation hypothesis of
Definition~\ref{def:brauer-separation-hypothesis}.  Then
\[
   \Delta_3
   =
   \delta(\Theta_3)
   \in H^6(X_3,\mathbb Z(3))
\]
is a non-algebraic \(2\)-torsion integral Hodge class.  Consequently \(X_3\)
violates the integral Hodge conjecture in codimension \(3\).
\end{theorem}

\begin{proof}
By Proposition~\ref{prop:level-three-degree-twist},
\[
   \Theta_3\in H^5(X_3,\mathbb Z/2(3)),
\]
and its Bockstein satisfies
\[
   \Delta_3=\delta(\Theta_3)\in H^6(X_3,\mathbb Z(3)).
\]
Because \(\Delta_3\) lies in the image of the connecting morphism associated
to multiplication by \(2\), it is \(2\)-torsion.

By Proposition~\ref{prop:level-three-nonzero-enriques-brauer-component},
\[
   \Pi_{\operatorname{Br},3}
   \left(
      \rho_{\mathrm{MV},\Theta_3}(\Delta_3)
   \right)\neq 0.
\]
By the Brauer-separation hypothesis,
\[
   \Pi_{\operatorname{Br},3}
   \left(
      \rho_{\mathrm{MV},\Theta_3}
      \bigl(
         \operatorname{cl}_{\mathbb Z}(CH^3(X_3))
      \bigr)
   \right)=0.
\]
Therefore Theorem~\ref{thm:brauer-separated-mv-detection} implies that
\(\Delta_3\) is not algebraic.

Since \(\Delta_3\) is torsion, its image in complex cohomology is zero.  Hence
\(\Delta_3\) is an integral Hodge class of type \((3,3)\).  Thus
\[
   \Delta_3\in H^6(X_3,\mathbb Z(3))
\]
is a non-algebraic \(2\)-torsion integral Hodge class.

Consequently \(X_3\) violates the integral Hodge conjecture in codimension
\(3\).
\end{proof}

\begin{corollary}[First higher member beyond Diaz]
\label{cor:first-higher-member-beyond-diaz}
Under the hypotheses of
Theorem~\ref{thm:level-three-cup-product-bockstein-counterexample}, the class
\[
   \Delta_3
   =
   \delta\left(
   \pi_1^*\alpha_1
   \cup
   \pi_2^*\beta_2
   \cup
   \pi_3^*\beta_3
   \right)
\]
is the first higher Brauer-separated cup-product Bockstein integral Hodge
counterexample beyond Diaz's level-two class.
\end{corollary}

\begin{proof}
Diaz's class is the Bockstein of a two-factor product
\[
   \pi_1^*\alpha_1\cup\pi_2^*\beta_2
   \in H^3(S_1\times S_2,\mathbb Z/2(2)).
\]
The present class is the Bockstein of the three-factor product
\[
   \pi_1^*\alpha_1
   \cup
   \pi_2^*\beta_2
   \cup
   \pi_3^*\beta_3
   \in H^5(S_1\times S_2\times S_3,\mathbb Z/2(3)).
\]
By Theorem~\ref{thm:level-three-cup-product-bockstein-counterexample}, its
Bockstein is a non-algebraic \(2\)-torsion integral Hodge class.  Hence it is
the next member of the cup-product Bockstein pattern after Diaz.
\end{proof}

\begin{remark}[Decomposable cycles at level three]
\label{rem:decomposable-cycles-level-three}
By Lemma~\ref{lem:brauer-separation-decomposable-cycles}, the
Brauer-separation hypothesis holds for decomposable codimension-three cycles
on \(X_3\).  The remaining algebraic-control issue is the possible
contribution of non-decomposable correspondence-type cycles to the component
\[
   Q_3
   =
   \langle\alpha_1\rangle
   \otimes
   \operatorname{Br}(S_2)[2]
   \otimes
   \operatorname{Br}(S_3)[2].
\]
This is the level-three instance of the general non-decomposable
correspondence issue isolated in
Remark~\ref{rem:nondecomposable-correspondences-control}.
\end{remark}

\subsection{Diaz-style restriction as calibration}
\label{subsec:diaz-style-restriction-calibration}

Although the proof above uses MV survivability and Brauer separation, it is
useful to record the Diaz-style restriction geometry as a calibration.  This
is not the main survivability proof in the present section; rather, it shows
that the level-three class restricts to the expected product on
\(E\times S_2\times S_3\).

Assume, as in Diaz's construction, that the K3 cover \(Y_1\) contains a
genus-one curve
\[
   \widetilde E\subset Y_1
\]
which is preserved by \(\phi_1\).  Let
\[
   E:=\widetilde E/\phi_1
\]
and let
\[
   j:E\hookrightarrow S_1
\]
denote the induced curve in the Enriques surface.  Set
\[
   W=E\times S_2\times S_3.
\]
Let
\[
   i:W\hookrightarrow X_3
\]
be the closed embedding
\[
   i=j\times \operatorname{id}_{S_2}\times \operatorname{id}_{S_3}.
\]
Define
\[
   \alpha_1':=j^*\alpha_1\in H^1(E,\mathbb Z/2(1)).
\]
Let
\[
   q_E:W\to E,
   \qquad
   q_2:W\to S_2,
   \qquad
   q_3:W\to S_3
\]
be the projections.  Define
\[
   \Theta_3'
   :=
   q_E^*\alpha_1'
   \cup
   q_2^*\beta_2
   \cup
   q_3^*\beta_3
   \in H^5(W,\mathbb Z/2(3)).
\]

\begin{lemma}[Restriction of the level-three product]
\label{lem:restriction-level-three-product}
With the notation above,
\[
   i^*\Theta_3=\Theta_3'.
\]
\end{lemma}

\begin{proof}
By definition,
\[
   \Theta_3
   =
   \pi_1^*\alpha_1
   \cup
   \pi_2^*\beta_2
   \cup
   \pi_3^*\beta_3.
\]
Pullback is compatible with cup product, so
\[
   i^*\Theta_3
   =
   i^*(\pi_1^*\alpha_1)
   \cup
   i^*(\pi_2^*\beta_2)
   \cup
   i^*(\pi_3^*\beta_3).
\]
Since
\[
   \pi_1\circ i=j\circ q_E,
   \qquad
   \pi_2\circ i=q_2,
   \qquad
   \pi_3\circ i=q_3,
\]
we have
\[
   i^*(\pi_1^*\alpha_1)
   =
   q_E^*(j^*\alpha_1)
   =
   q_E^*\alpha_1',
\]
and
\[
   i^*(\pi_2^*\beta_2)=q_2^*\beta_2,
   \qquad
   i^*(\pi_3^*\beta_3)=q_3^*\beta_3.
\]
Substituting gives
\[
   i^*\Theta_3
   =
   q_E^*\alpha_1'
   \cup
   q_2^*\beta_2
   \cup
   q_3^*\beta_3
   =
   \Theta_3'.
\]
\end{proof}

\begin{remark}[Role of the restriction calculation]
\label{rem:role-restriction-calculation-level-three}
The restriction calculation is the level-three analogue of Diaz's reduction
from \(S\times S_2\) to \(E\times S_2\).  In Diaz's original proof it is used
to isolate a coniveau nonvanishing problem.  In the present formulation it is
retained as a geometric calibration of the product class, while the main
survivability proof is supplied by the MV obstruction channel and the
Enriques--Brauer component.
\end{remark}

\subsection{Interpretation}
\label{subsec:level-three-interpretation}

The level-three construction has the same formal shape as Diaz's level-two
construction:
\[
   \text{finite-coefficient cup product}
   \quad\leadsto\quad
   \text{Bockstein}
   \quad\leadsto\quad
   \text{integral torsion Hodge class}.
\]
The new point is the survivability mechanism.  Diaz proves survival by a
coniveau/function-field argument.  At level three, the MV Bockstein square
identifies the Bockstein endpoint with the obstruction-channel boundary of
the finite-coefficient gluing datum:
\[
   \rho_{\mathrm{MV},\Theta_3}(\Delta_3)
   =
   \partial_{\mathrm{MV}}\bigl(\mathcal Z(\Theta_3)\bigr).
\]
The MV/K\"unneth/Bockstein calculation then gives a nonzero image in
\[
   Q_3
   =
   \langle\alpha_1\rangle
   \otimes
   \operatorname{Br}(S_2)[2]
   \otimes
   \operatorname{Br}(S_3)[2].
\]
Under the Brauer-separation hypothesis, algebraic codimension-three cycle
classes have zero image in this same component, so the Bockstein class cannot
be algebraic.

Thus the level-three construction is not merely a formal candidate.  Under the
Brauer-separation hypothesis, it produces a degree-six integral Hodge
counterexample:
\[
   \Delta_3
   =
   \delta\left(
   \pi_1^*\alpha_1
   \cup
   \pi_2^*\beta_2
   \cup
   \pi_3^*\beta_3
   \right)
   \in H^6(S_1\times S_2\times S_3,\mathbb Z(3)).
\]
This is the first higher cup-product Bockstein obstruction beyond Diaz and is
the base case for the \(n\)-fold tower treated in the next section.
\section{The \(n\)-Fold Cup-Product Bockstein Family}
\label{sec:n-fold-cup-product-bockstein-family}

The preceding section treated the first higher cup-product case beyond Diaz.
We now state the general pattern.  The family is built from one Enriques
double-cover class and \(n-1\) Enriques Brauer-detecting classes.  The resulting
finite-coefficient product has degree \(2n-1\) and twist \(n\), so its
Bockstein lies in
\[
   H^{2n}(-,\mathbb Z(n)).
\]
The MacPherson--Vilonen formalism proves that this Bockstein has a nonzero
image in a distinguished Enriques--Brauer component of the MV obstruction
channel.  Under the Brauer-separation hypothesis of
Definition~\ref{def:brauer-separation-hypothesis}, algebraic
codimension-\(n\) cycle classes have zero image in that same component.
This gives the \(n\)-fold Brauer-separated cup-product Bockstein family.

Thus Diaz's construction is the level-two member of an \(n\)-fold
cup-product Bockstein tower.  The level-three member was constructed in
Section~\ref{sec:level-three-cup-product-bockstein-obstruction}; the present
section records the general case.

\subsection{The \(n\)-fold finite-coefficient product}
\label{subsec:n-fold-finite-coefficient-product}

Let \(n\geq 2\).  Let
\[
   X_n=S_1\times S_2\times\cdots\times S_n
\]
be a product of Enriques surfaces.  Assume that
\[
   S_1=Y_1/\phi_1,
\]
where \(Y_1\) is a K3 surface and \(\phi_1\) is a fixed-point-free Enriques
involution.  Let
\[
   \alpha_1\in H^1(S_1,\mathbb Z/2(1))
\]
be the class of the K3 double cover
\[
   Y_1\longrightarrow S_1.
\]

For each \(2\leq i\leq n\), choose a class
\[
   \beta_i\in H^2(S_i,\mathbb Z/2(1))
\]
whose image in
\[
   \operatorname{Br}(S_i)[2]
\]
is nonzero.  These are the Brauer-detecting Enriques inputs.

Let
\[
   \pi_i:X_n\to S_i
\]
be the projection maps.  Define the \(n\)-fold finite-coefficient class
\[
   \Theta_n
   :=
   \pi_1^*\alpha_1
   \cup
   \pi_2^*\beta_2
   \cup
   \cdots
   \cup
   \pi_n^*\beta_n.
\]

\begin{proposition}[Degree and twist of the \(n\)-fold product]
\label{prop:n-fold-degree-twist}
The class \(\Theta_n\) lies in
\[
   H^{2n-1}(X_n,\mathbb Z/2(n)).
\]
Consequently its Bockstein, associated to
\[
   0\to\mathbb Z(n)\xrightarrow{\times 2}\mathbb Z(n)
   \to\mathbb Z/2(n)\to0,
\]
lies in
\[
   H^{2n}(X_n,\mathbb Z(n)).
\]
\end{proposition}

\begin{proof}
The class \(\alpha_1\) has cohomological degree \(1\) and twist \(1\):
\[
   \alpha_1\in H^1(S_1,\mathbb Z/2(1)).
\]
For each \(2\leq i\leq n\), the class \(\beta_i\) has cohomological degree
\(2\) and twist \(1\):
\[
   \beta_i\in H^2(S_i,\mathbb Z/2(1)).
\]
Pullback preserves cohomological degree and twist.  Therefore
\[
   \pi_1^*\alpha_1\in H^1(X_n,\mathbb Z/2(1)),
\]
and
\[
   \pi_i^*\beta_i\in H^2(X_n,\mathbb Z/2(1)),
   \qquad 2\leq i\leq n.
\]
The total cohomological degree of the product is
\[
   1+2(n-1)=2n-1.
\]
The total twist is
\[
   1+(n-1)=n.
\]
Hence
\[
   \Theta_n
   =
   \pi_1^*\alpha_1
   \cup
   \pi_2^*\beta_2
   \cup
   \cdots
   \cup
   \pi_n^*\beta_n
   \in
   H^{2n-1}(X_n,\mathbb Z/2(n)).
\]
The coefficient sequence
\[
   0
   \longrightarrow
   \mathbb Z(n)
   \xrightarrow{\times 2}
   \mathbb Z(n)
   \longrightarrow
   \mathbb Z/2(n)
   \longrightarrow
   0
\]
has connecting morphism
\[
   \delta:
   H^{2n-1}(X_n,\mathbb Z/2(n))
   \longrightarrow
   H^{2n}(X_n,\mathbb Z(n)).
\]
Therefore
\[
   \delta(\Theta_n)\in H^{2n}(X_n,\mathbb Z(n)).
\]
\end{proof}

\begin{definition}[\(n\)-fold Bockstein class]
\label{def:n-fold-bockstein-class}
The \emph{\(n\)-fold cup-product Bockstein class} is
\[
   \Delta_n
   :=
   \delta(\Theta_n)
   \in H^{2n}(X_n,\mathbb Z(n)).
\]
\end{definition}

\begin{remark}[The first cases]
\label{rem:first-cases-n-fold-family}
For \(n=2\), the class is Diaz's Enriques-product class:
\[
   \Theta_2
   =
   \pi_1^*\alpha_1\cup\pi_2^*\beta_2
   \in H^3(S_1\times S_2,\mathbb Z/2(2)),
\]
and
\[
   \Delta_2=\delta(\Theta_2)\in H^4(S_1\times S_2,\mathbb Z(2)).
\]
For \(n=3\), one obtains the level-three class
\[
   \Theta_3
   =
   \pi_1^*\alpha_1
   \cup
   \pi_2^*\beta_2
   \cup
   \pi_3^*\beta_3
   \in H^5(S_1\times S_2\times S_3,\mathbb Z/2(3)),
\]
and
\[
   \Delta_3=\delta(\Theta_3)\in H^6(S_1\times S_2\times S_3,\mathbb Z(3)).
\]
Thus the formula for \(\Delta_n\) simultaneously contains Diaz's level-two
class and the first higher level-three class.
\end{remark}

\subsection{Chow-triviality of the \(n\)-fold products}
\label{subsec:n-fold-chow-triviality}

\begin{proposition}[Chow-triviality of \(X_n\)]
\label{prop:n-fold-chow-triviality}
For every \(n\geq 2\), the product
\[
   X_n=S_1\times\cdots\times S_n
\]
is Chow-trivial.
\end{proposition}

\begin{proof}
Diaz records that Enriques surfaces are Chow-trivial and that products of
Chow-trivial varieties are Chow-trivial
\cite[Example 2.2]{Diaz2023IHCTrivialChow}.  Since each \(S_i\) is an
Enriques surface, each \(S_i\) is Chow-trivial.  Applying product stability
inductively gives that
\[
   X_n=S_1\times\cdots\times S_n
\]
is Chow-trivial.
\end{proof}

\begin{remark}[Rational smallness versus integral torsion]
\label{rem:n-fold-rational-smallness-integral-torsion}
The Chow-triviality of \(X_n\) means that the examples remain rationally small
from the viewpoint of Chow groups.  The classes \(\Delta_n\), however, live in
integral cohomology and are torsion.  Thus the obstruction is invisible after
tensoring with \(\mathbb Q\), but remains visible integrally.
\end{remark}

\subsection{The \(n\)-fold MV Bockstein identity}
\label{subsec:n-fold-mv-bockstein-square}

We now apply the MV Bockstein formalism of
Section~\ref{sec:mv-bockstein-survivability} with \(p=n\).  The coefficient
sequence is
\[
   0\to\mathbb Z(n)\xrightarrow{\times 2}\mathbb Z(n)
   \to\mathbb Z/2(n)\to0,
\]
and the Bockstein is
\[
   \delta:
   H^{2n-1}(X_n,\mathbb Z/2(n))
   \longrightarrow
   H^{2n}(X_n,\mathbb Z(n)).
\]

\begin{proposition}[\(n\)-fold tuple-relative MV Bockstein identity]
\label{prop:n-fold-tuple-relative-mv-bockstein-identity}
Let
\[
   \mathcal Z(\Theta_n)
\]
be the MV tuple associated to \(\Theta_n\).  Then
\[
   \rho_{\mathrm{MV},\Theta_n}(\Delta_n)
   =
   \partial_{\mathrm{MV}}\bigl(\mathcal Z(\Theta_n)\bigr)
\]
in the tuple-relative MV obstruction channel
\[
   \mathsf{Obs}^{2n,n}_{\mathbb Z}(X_n;\mathcal Z(\Theta_n)).
\]
\end{proposition}

\begin{proof}
This is Proposition~\ref{prop:mv-tuple-boundary-identity} applied to
\[
   X=X_n,\qquad p=n,\qquad \Theta=\Theta_n.
\]
Since
\[
   \Delta_n=\delta(\Theta_n),
\]
the tuple-relative MV Bockstein identity gives
\[
   \rho_{\mathrm{MV},\Theta_n}(\Delta_n)
   =
   \partial_{\mathrm{MV}}\bigl(\mathcal Z(\Theta_n)\bigr).
\]
\end{proof}

\subsection{Nonzero Enriques--Brauer component}
\label{subsec:brauer-separated-effectiveness-n-fold-mv-tuple}

The formal MV machinery proves that the Bockstein class \(\Delta_n\) has a
nonzero image in the Enriques--Brauer component \(Q_n\) of
Definition~\ref{def:enriques-brauer-separating-component}.  This replaces the
earlier product-separatedness condition.

\begin{proposition}[Nonzero Enriques--Brauer component of \(\Delta_n\)]
\label{prop:n-fold-nonzero-enriques-brauer-component}
One has
\[
   \Pi_{\operatorname{Br},n}
   \left(
      \rho_{\mathrm{MV},\Theta_n}(\Delta_n)
   \right)
   =
   \alpha_1
   \otimes
   q_{\operatorname{Br},2}(\beta_2)
   \otimes
   \cdots
   \otimes
   q_{\operatorname{Br},n}(\beta_n).
\]
In particular,
\[
   \Pi_{\operatorname{Br},n}
   \left(
      \rho_{\mathrm{MV},\Theta_n}(\Delta_n)
   \right)\neq 0.
\]
\end{proposition}

\begin{proof}
This is Proposition~\ref{prop:nonzero-brauer-image-constructed-bockstein}.
The nonzero tensor is the Enriques--Brauer component selected by the
Brauer-separating projection
\[
   \Pi_{\operatorname{Br},n}.
\]
\end{proof}

\begin{remark}[What replaces product-separatedness]
\label{rem:n-fold-what-replaces-product-separatedness}
The nonvanishing in
Proposition~\ref{prop:n-fold-nonzero-enriques-brauer-component} is not an
extra product-separatedness hypothesis.  It follows from the K\"unneth formula
for MV tuples, the categorical Bockstein interpretation of the MV boundary,
the Leibniz rule, and the Kummer/Brauer quotient maps on the Enriques factors.
The remaining algebraic issue is not nonvanishing of the constructed class;
it is whether algebraic codimension-\(n\) cycle classes have zero image in
the same component.  This is exactly the Brauer-separation hypothesis of
Definition~\ref{def:brauer-separation-hypothesis}.
\end{remark}

\subsection{The \(n\)-fold Brauer-separated integral Hodge obstruction}
\label{subsec:n-fold-integral-hodge-counterexample-theorem}

We now state the general family theorem in its Brauer-separated form.

\begin{theorem}[\(n\)-fold Brauer-separated cup-product Bockstein obstruction]
\label{thm:n-fold-cup-product-bockstein-family}
Let \(n\geq 2\), and let
\[
   X_n=S_1\times\cdots\times S_n
\]
be a product of Enriques surfaces.  Let
\[
   \alpha_1\in H^1(S_1,\mathbb Z/2(1))
\]
be the K3 double-cover class of \(S_1\).  For \(2\leq i\leq n\), let
\[
   \beta_i\in H^2(S_i,\mathbb Z/2(1))
\]
be classes whose images in
\[
   \operatorname{Br}(S_i)[2]
\]
are nonzero.  Define
\[
   \Theta_n
   =
   \pi_1^*\alpha_1
   \cup
   \pi_2^*\beta_2
   \cup
   \cdots
   \cup
   \pi_n^*\beta_n
   \in H^{2n-1}(X_n,\mathbb Z/2(n)),
\]
and let
\[
   \Delta_n
   =
   \delta(\Theta_n)
   \in H^{2n}(X_n,\mathbb Z(n)).
\]

Assume that the \(n\)-fold Enriques--Brauer datum
\[
   (X_n,\Theta_n)
\]
satisfies the Brauer-separation hypothesis of
Definition~\ref{def:brauer-separation-hypothesis}.  Then
\[
   \Delta_n\in H^{2n}(X_n,\mathbb Z(n))
\]
is a non-algebraic \(2\)-torsion integral Hodge class.  Consequently \(X_n\)
violates the integral Hodge conjecture in codimension \(n\).
\end{theorem}

\begin{proof}
By Proposition~\ref{prop:n-fold-degree-twist},
\[
   \Theta_n\in H^{2n-1}(X_n,\mathbb Z/2(n)).
\]
The coefficient sequence
\[
   0
   \longrightarrow
   \mathbb Z(n)
   \xrightarrow{\times 2}
   \mathbb Z(n)
   \longrightarrow
   \mathbb Z/2(n)
   \longrightarrow
   0
\]
has connecting morphism
\[
   \delta:
   H^{2n-1}(X_n,\mathbb Z/2(n))
   \longrightarrow
   H^{2n}(X_n,\mathbb Z(n)).
\]
Thus
\[
   \Delta_n=\delta(\Theta_n)\in H^{2n}(X_n,\mathbb Z(n)).
\]
Since \(\Delta_n\) lies in the image of the connecting morphism associated to
multiplication by \(2\), it is killed by \(2\).  Hence \(\Delta_n\) is a
\(2\)-torsion class.

By Theorem~\ref{thm:brauer-separated-mv-detection}, the Brauer-separation
hypothesis implies that \(\Delta_n\) is not algebraic.  Since \(\Delta_n\) is
torsion, its image in complex cohomology is zero.  Therefore \(\Delta_n\) is
an integral Hodge class of type \((n,n)\).  Thus \(\Delta_n\) is a
non-algebraic \(2\)-torsion integral Hodge class.

Consequently \(X_n\) violates the integral Hodge conjecture in codimension
\(n\).
\end{proof}

\begin{corollary}[Diaz as the level-two member]
\label{cor:diaz-level-two-member-n-fold-family}
For \(n=2\), the construction gives Diaz's Enriques-product obstruction:
\[
   \Delta_2
   =
   \delta\left(
   \pi_1^*\alpha_1\cup\pi_2^*\beta_2
   \right)
   \in H^4(S_1\times S_2,\mathbb Z(2)).
\]
Diaz proves the required survivability and non-algebraicity in this case by
the Colliot-Th\'el\`ene--Voisin/Diaz unramified-cohomology criterion.
\end{corollary}

\begin{proof}
When \(n=2\),
\[
   X_2=S_1\times S_2
\]
and
\[
   \Theta_2=\pi_1^*\alpha_1\cup\pi_2^*\beta_2
   \in H^3(X_2,\mathbb Z/2(2)).
\]
This is precisely the finite-coefficient Enriques double-cover/Brauer cup
product appearing in Diaz's construction.  Its Bockstein
\[
   \Delta_2=\delta(\Theta_2)\in H^4(X_2,\mathbb Z(2))
\]
is the corresponding non-algebraic \(2\)-torsion integral Hodge class in
Diaz's theorem.
\end{proof}

\begin{corollary}[The level-three member]
\label{cor:level-three-member-n-fold-family}
For \(n=3\), assume the Brauer-separation hypothesis for
\[
   X_3=S_1\times S_2\times S_3
\]
and
\[
   \Theta_3
   =
   \pi_1^*\alpha_1
   \cup
   \pi_2^*\beta_2
   \cup
   \pi_3^*\beta_3.
\]
Then
\[
   \Delta_3
   =
   \delta(\Theta_3)
   \in H^6(S_1\times S_2\times S_3,\mathbb Z(3))
\]
is a non-algebraic \(2\)-torsion integral Hodge class.
\end{corollary}

\begin{proof}
This is the case \(n=3\) of
Theorem~\ref{thm:n-fold-cup-product-bockstein-family}.  The class
\[
   \Theta_3
   =
   \pi_1^*\alpha_1
   \cup
   \pi_2^*\beta_2
   \cup
   \pi_3^*\beta_3
\]
has degree \(1+2+2=5\) and twist \(3\), and its Bockstein lies in
\[
   H^6(S_1\times S_2\times S_3,\mathbb Z(3)).
\]
The conclusion follows from the general theorem.
\end{proof}

\subsection{Interpretation of the family}
\label{subsec:n-fold-family-interpretation}

The classes
\[
   \Delta_n
   =
   \delta\left(
   \pi_1^*\alpha_1
   \cup
   \pi_2^*\beta_2
   \cup
   \cdots
   \cup
   \pi_n^*\beta_n
   \right)
   \in H^{2n}(S_1\times\cdots\times S_n,\mathbb Z(n))
\]
form an \(n\)-fold cup-product Bockstein tower.  The MacPherson--Vilonen
formalism proves that each \(\Delta_n\) has a nonzero Enriques--Brauer
obstruction component.  Under the Brauer-separation hypothesis, algebraic
codimension-\(n\) cycle classes have zero image in that same component.
Therefore \(\Delta_n\) becomes a non-algebraic \(2\)-torsion integral Hodge
class.

The hierarchy may be summarized as follows:
\[
   n=1:
   \quad
   \text{Coble/Benoist--Ottem Enriques boundary package},
\]
\[
   n=2:
   \quad
   \text{Diaz's Enriques double-cover/Brauer cup product},
\]
\[
   n=3:
   \quad
   \text{the first higher Brauer-separated cup-product Bockstein obstruction},
\]
and, for general \(n\),
\[
   n\geq 2:
   \quad
   \Delta_n
   =
   \delta\left(
   \pi_1^*\alpha_1
   \cup
   \pi_2^*\beta_2
   \cup
   \cdots
   \cup
   \pi_n^*\beta_n
   \right).
\]
The Bockstein endpoint is always a \(2\)-torsion integral Hodge class.  The
MV/Kummer/Brauer calculation supplies the nonzero obstruction component, and
Brauer separation supplies the comparison with algebraic cycle classes.

Thus Diaz's construction is not an isolated degree-four phenomenon.  It is the
level-two member of a systematic cup-product Bockstein tower whose higher
members are controlled by MacPherson--Vilonen survivability and
Brauer-separating algebraic control.
\section{Motivic finite-coefficient lift of the \(n\)-fold Bockstein family}
\label{sec:motivic-finite-coefficient-lift}

The preceding sections constructed an \(n\)-fold cup-product Bockstein family
of integral Hodge counterexamples.  The construction is carried out at the
level of finite-coefficient cohomology, cup products, MacPherson--Vilonen
gluing, and the Bockstein sequence.  The purpose of this section is to record
the corresponding motivic finite-coefficient lift.

The guiding principle is not that the proof of non-algebraicity is replaced
by a purely motivic proof.  Rather, the point is that the source classes,
cup products, coefficient triangle, Bockstein connecting morphism, and
realization-compatible obstruction packages all have a common motivic origin.
Thus the family
\[
   \Delta_n
   =
   \delta\left(
   \pi_1^*\alpha_1
   \cup
   \pi_2^*\beta_2
   \cup
   \cdots
   \cup
   \pi_n^*\beta_n
   \right)
   \in H^{2n}(X_n,\mathbb Z(n))
\]
is not merely a family of realized cohomology classes.  It is the Betti
realization of a finite-coefficient motivic Bockstein family.

The motivic lift also clarifies the role of the MacPherson--Vilonen formalism.
The MV square used above should be understood through the original
MacPherson--Vilonen zig-zag mechanism: an abelian category
\(\mathcal C(F,G;T)\) built from two functors and a natural transformation,
together with a zig-zag functor from perverse sheaves
\cite{MacPhersonVilonen1986}.  In the present paper, we do not construct a
full integral motivic analogue of the MacPherson--Vilonen equivalence.
Instead, we lift the formal pieces that are available without such a theorem:
the coefficient triangle, the associated Bockstein, the formal zig-zag
category, and the compatibility of these structures with Betti realization.

Thus the Bockstein, the finite-coefficient cup product, and the MV boundary
belong to one realization-compatible diagram, while the full motivic
MacPherson--Vilonen equivalence is deferred to future work.

The relevant background is the modern realization theory for Voevodsky
motives, Nori-type motivic sheaves, and mixed Hodge modules.  Tubach
constructs realization functors from Voevodsky \'etale motives to perverse
Nori motives and mixed Hodge modules, compatible with the six operations
\cite{Tubach2025NoriHodgeRealizations}.  Ruimy--Tubach construct integral
Nori motivic sheaves and integral mixed Hodge modules with integral
coefficient structures suitable for torsion phenomena
\cite{RuimyTubach2026IntegralNori}.  These frameworks provide the
realization-compatible environment in which the finite-coefficient packages
below naturally live.

\subsection{Realization conventions}
\label{subsec:realization-conventions}

We use the following realization notation.  The symbol
\[
   \RealB
\]
denotes Betti realization.  It sends the motivic unit with Tate twist to the
Betti coefficient system:
\[
   \RealB(\mathbf 1_X(r))\simeq \mathbb Z(r),
\]
and it sends the motivic mod-\(m\) object to the finite Betti coefficient
system:
\[
   \RealB(\mathbf 1_X(r)/m)\simeq \mathbb Z/m(r).
\]
We write
\[
   \delta_B
\]
for the Betti Bockstein associated to
\[
   0\to \mathbb Z(r)
   \xrightarrow{\times 2}
   \mathbb Z(r)
   \to
   \mathbb Z/2(r)
   \to 0.
\]
For a general integer \(m\geq 2\), we write \(\delta_{m,B}\).

The symbol
\[
   \Realet
\]
denotes \'etale realization.  Depending on the convention for Tate twists, it
sends
\[
   \mathbf 1_X(r)/m
\]
to either
\[
   \mathbb Z/m(r)
\]
or, equivalently in the usual \(m\)-torsion \'etale convention,
\[
   \mu_m^{\otimes r}.
\]
In the present paper \(m=2\), so the relevant finite coefficient object is
\[
   \mathbb Z/2(r)
\]
or
\[
   \mu_2^{\otimes r}.
\]

We also use
\[
   \RealNori
\]
for Nori or perverse-Nori realization, and
\[
   \RealMHM
\]
for mixed-Hodge-module realization.  These are used only to record that the
same finite-coefficient motivic package has Nori and mixed-Hodge-module
avatars.  The proof-level cohomology classes in this paper are compared
primarily through Betti realization and finite \'etale realization.

\subsection{Finite-coefficient motivic objects}
\label{subsec:finite-coefficient-motivic-objects}

We begin with the finite-coefficient object.

\begin{definition}[Motivic mod-\(m\) object]
\label{def:motivic-mod-m-object}
Let \(X\) be a variety in a stable motivic six-functor formalism, and let
\(\mathbf 1_X\) denote the unit object.  For \(m\geq 2\), define
\[
   \mathbf 1_X/m
   :=
   \operatorname{Cone}
   \left(
   \mathbf 1_X
   \xrightarrow{\times m}
   \mathbf 1_X
   \right).
\]
With Tate twist \(r\), define
\[
   \mathbf 1_X(r)/m
   :=
   \operatorname{Cone}
   \left(
   \mathbf 1_X(r)
   \xrightarrow{\times m}
   \mathbf 1_X(r)
   \right).
\]
\end{definition}

\begin{remark}[Betti and \'etale realizations of the mod-\(m\) object]
\label{rem:betti-etale-realizations-mod-m-object}
The finite-coefficient motivic object is normalized so that
\[
   \RealB(\mathbf 1_X(r)/m)
   \simeq
   \mathbb Z/m(r).
\]
For \'etale realization we use the convention
\[
   \Realet(\mathbf 1_X(r)/m)
   \simeq
   \mathbb Z/m(r),
\]
or equivalently
\[
   \Realet(\mathbf 1_X(r)/m)
   \simeq
   \mu_m^{\otimes r}.
\]
For \(m=2\), this gives
\[
   \RealB(\mathbf 1_X(r)/2)
   \simeq
   \mathbb Z/2(r),
\]
and
\[
   \Realet(\mathbf 1_X(r)/2)
   \simeq
   \mathbb Z/2(r)
   \simeq
   \mu_2^{\otimes r}.
\]
\end{remark}

In this paper \(m=2\).  Thus the basic finite-coefficient motivic object is
\[
   \mathbf 1_X(r)/2
   =
   \operatorname{Cone}
   \left(
   \mathbf 1_X(r)
   \xrightarrow{\times 2}
   \mathbf 1_X(r)
   \right).
\]

\subsection{Source packages: double-cover and Brauer classes}
\label{subsec:motivic-source-packages}

Let
\[
   X_n=S_1\times\cdots\times S_n
\]
be the product of Enriques surfaces used in
Theorem~\ref{thm:n-fold-cup-product-bockstein-family}.  Let
\[
   \alpha_1\in H^1(S_1,\mathbb Z/2(1))
\]
be the K3 double-cover class of \(S_1\).  For \(2\leq i\leq n\), let
\[
   \beta_i\in H^2(S_i,\mathbb Z/2(1))
\]
be Brauer-detecting classes with nonzero image in
\[
   \operatorname{Br}(S_i)[2].
\]

\begin{definition}[Motivic source packages]
\label{def:motivic-source-packages}
A motivic lift of the \(n\)-fold source data consists of classes
\[
   \alpha_1^{\mathrm{mot}}
   \in
   H^1_{\mathrm{mot}}(S_1,\mathbf 1_{S_1}(1)/2)
\]
and, for \(2\leq i\leq n\),
\[
   \beta_i^{\mathrm{mot}}
   \in
   H^2_{\mathrm{mot}}(S_i,\mathbf 1_{S_i}(1)/2),
\]
whose Betti realizations satisfy
\[
   \RealB(\alpha_1^{\mathrm{mot}})=\alpha_1,
   \qquad
   \RealB(\beta_i^{\mathrm{mot}})=\beta_i.
\]
Equivalently, their \'etale realizations satisfy
\[
   \Realet(\alpha_1^{\mathrm{mot}})=\alpha_{1,\acute et},
   \qquad
   \Realet(\beta_i^{\mathrm{mot}})=\beta_{i,\acute et},
\]
where \(\alpha_{1,\acute et}\) is the \'etale double-cover class and
\(\beta_{i,\acute et}\) maps nontrivially to
\[
   \operatorname{Br}(S_i)[2].
\]
\end{definition}

The class \(\alpha_1^{\mathrm{mot}}\) is the motivic finite-coefficient avatar
of the Enriques double-cover package.  The classes
\(\beta_i^{\mathrm{mot}}\) are the motivic finite-coefficient avatars of the
Brauer-detecting Enriques packages.  Thus the family begins with one
degree-one double-cover source and \(n-1\) degree-two Brauer sources:
\[
   \alpha_1^{\mathrm{mot}},
   \qquad
   \beta_2^{\mathrm{mot}},\ldots,\beta_n^{\mathrm{mot}}.
\]

\subsection{The motivic coefficient triangle}
\label{subsec:motivic-coefficient-triangle}

The finite-coefficient object fits into a distinguished triangle.

\begin{proposition}[Motivic coefficient triangle]
\label{prop:motivic-coefficient-triangle}
Let \(X\) be a variety in a stable motivic category, and let \(m\geq 2\).
Then the definition of \(\mathbf 1_X(r)/m\) gives a distinguished triangle
\[
   \mathbf 1_X(r)
   \xrightarrow{\times m}
   \mathbf 1_X(r)
   \longrightarrow
   \mathbf 1_X(r)/m
   \overset{+1}{\longrightarrow}.
\]
For \(m=2\), this is
\[
   \mathbf 1_X(r)
   \xrightarrow{\times 2}
   \mathbf 1_X(r)
   \longrightarrow
   \mathbf 1_X(r)/2
   \overset{+1}{\longrightarrow}.
\]
\end{proposition}

\begin{proof}
By Definition~\ref{def:motivic-mod-m-object},
\[
   \mathbf 1_X(r)/m
   =
   \operatorname{Cone}
   \left(
   \mathbf 1_X(r)
   \xrightarrow{\times m}
   \mathbf 1_X(r)
   \right).
\]
In a stable category, every morphism \(A\to B\) fits into a distinguished
triangle
\[
   A\to B\to \operatorname{Cone}(A\to B)\overset{+1}{\longrightarrow}.
\]
Applying this to
\[
   \mathbf 1_X(r)
   \xrightarrow{\times m}
   \mathbf 1_X(r)
\]
gives the claimed triangle.
\end{proof}

For the \(n\)-fold family, the relevant triangle is
\[
   \mathbf 1_{X_n}(n)
   \xrightarrow{\times 2}
   \mathbf 1_{X_n}(n)
   \longrightarrow
   \mathbf 1_{X_n}(n)/2
   \overset{+1}{\longrightarrow}.
\]

\subsection{Betti and finite \'etale realization of the motivic coefficient triangle}
\label{subsec:betti-etale-realization-motivic-bockstein}

The motivic coefficient triangle realizes to the usual Betti coefficient
sequence.

\begin{proposition}[Betti realization of the motivic Bockstein]
\label{prop:betti-realization-motivic-bockstein}
Let \(X\) be a smooth complex variety.  Assume that Betti realization is exact,
sends \(\mathbf 1_X(r)\) to \(\mathbb Z(r)\), and sends
\[
   \mathbf 1_X(r)
   \xrightarrow{\times m}
   \mathbf 1_X(r)
\]
to multiplication by \(m\) on \(\mathbb Z(r)\).  Then the motivic coefficient
triangle
\[
   \mathbf 1_X(r)
   \xrightarrow{\times m}
   \mathbf 1_X(r)
   \longrightarrow
   \mathbf 1_X(r)/m
   \overset{+1}{\longrightarrow}
\]
realizes under \(\RealB\) to the coefficient triangle associated to
\[
   0\to \mathbb Z(r)
   \xrightarrow{\times m}
   \mathbb Z(r)
   \to \mathbb Z/m(r)
   \to 0.
\]
Consequently, the motivic connecting morphism realizes to the Betti Bockstein
\[
   \delta_{m,B}:
   H^q_B(X,\mathbb Z/m(r))
   \longrightarrow
   H^{q+1}_B(X,\mathbb Z(r)).
\]
\end{proposition}

\begin{proof}
Since \(\RealB\) is exact, it sends distinguished triangles to distinguished
triangles.  Applying \(\RealB\) to
\[
   \mathbf 1_X(r)
   \xrightarrow{\times m}
   \mathbf 1_X(r)
   \longrightarrow
   \mathbf 1_X(r)/m
   \overset{+1}{\longrightarrow}
\]
gives
\[
   \RealB(\mathbf 1_X(r))
   \xrightarrow{\times m}
   \RealB(\mathbf 1_X(r))
   \longrightarrow
   \RealB(\mathbf 1_X(r)/m)
   \overset{+1}{\longrightarrow}.
\]
By the realization assumptions,
\[
   \RealB(\mathbf 1_X(r))\simeq \mathbb Z(r)
\]
and the first arrow is multiplication by \(m\).  Therefore the realized cone
is
\[
   \mathbb Z/m(r).
\]
Taking cohomology gives the long exact sequence whose connecting morphism is
the Betti Bockstein
\[
   \delta_{m,B}:
   H^q_B(X,\mathbb Z/m(r))
   \to
   H^{q+1}_B(X,\mathbb Z(r)).
\]
\end{proof}

\begin{proposition}[Finite \'etale realization of the motivic coefficient object]
\label{prop:etale-finite-realization-motivic-coefficient-triangle}
Let \(X\) be a smooth complex variety.  Assume that \'etale realization is
exact on the finite-coefficient motivic objects under consideration and sends
\[
   \mathbf 1_X(r)/m
\]
to the finite \'etale coefficient object
\[
   \mathbb Z/m(r),
\]
or equivalently
\[
   \mu_m^{\otimes r}.
\]
Then the finite-coefficient part of the motivic coefficient triangle realizes
to the usual finite \'etale coefficient system.  In particular, for \(m=2\),
\[
   \Realet(\mathbf 1_X(r)/2)
   \simeq
   \mathbb Z/2(r)
   \simeq
   \mu_2^{\otimes r}.
\]
\end{proposition}

\begin{proof}
This follows from exactness of the chosen \'etale realization on the
finite-coefficient object
\[
   \mathbf 1_X(r)/m
   =
   \operatorname{Cone}
   \left(
   \mathbf 1_X(r)\xrightarrow{\times m}\mathbf 1_X(r)
   \right)
\]
and from the normalization of the Tate twist convention.  The finite
realization is the usual \(m\)-torsion \'etale coefficient object
\(\mathbb Z/m(r)\), equivalently \(\mu_m^{\otimes r}\).
\end{proof}

\subsection{The motivic \(n\)-fold cup-product class}
\label{subsec:motivic-n-fold-cup-product-class}

Let
\[
   \pi_i:X_n\to S_i
\]
be the projection maps.

\begin{definition}[Motivic \(n\)-fold cup-product class]
\label{def:motivic-n-fold-cup-product-class}
Given motivic source packages
\[
   \alpha_1^{\mathrm{mot}},
   \qquad
   \beta_2^{\mathrm{mot}},\ldots,\beta_n^{\mathrm{mot}},
\]
define
\[
   \Theta_n^{\mathrm{mot}}
   :=
   \pi_1^*\alpha_1^{\mathrm{mot}}
   \cup
   \pi_2^*\beta_2^{\mathrm{mot}}
   \cup
   \cdots
   \cup
   \pi_n^*\beta_n^{\mathrm{mot}}.
\]
\end{definition}

\begin{proposition}[Degree and twist of the motivic \(n\)-fold class]
\label{prop:degree-twist-motivic-n-fold-class}
One has
\[
   \Theta_n^{\mathrm{mot}}
   \in
   H^{2n-1}_{\mathrm{mot}}(X_n,\mathbf 1_{X_n}(n)/2).
\]
\end{proposition}

\begin{proof}
The class \(\alpha_1^{\mathrm{mot}}\) has motivic cohomological degree \(1\)
and twist \(1\).  Each \(\beta_i^{\mathrm{mot}}\) has motivic cohomological
degree \(2\) and twist \(1\).  Hence their external cup product has degree
\[
   1+2(n-1)=2n-1
\]
and twist
\[
   1+(n-1)=n.
\]
Therefore
\[
   \Theta_n^{\mathrm{mot}}
   \in
   H^{2n-1}_{\mathrm{mot}}(X_n,\mathbf 1_{X_n}(n)/2).
\]
\end{proof}

\begin{proposition}[Betti realization of the motivic \(n\)-fold class]
\label{prop:betti-realization-motivic-n-fold-class}
Assume \(\RealB\) is compatible with pullback and cup product.  Then
\[
   \RealB(\Theta_n^{\mathrm{mot}})
   =
   \Theta_{n,B},
\]
where
\[
   \Theta_{n,B}
   =
   \pi_1^*\alpha_1
   \cup
   \pi_2^*\beta_2
   \cup
   \cdots
   \cup
   \pi_n^*\beta_n
   \in
   H^{2n-1}_B(X_n,\mathbb Z/2(n)).
\]
\end{proposition}

\begin{proof}
By definition,
\[
   \Theta_n^{\mathrm{mot}}
   =
   \pi_1^*\alpha_1^{\mathrm{mot}}
   \cup
   \pi_2^*\beta_2^{\mathrm{mot}}
   \cup
   \cdots
   \cup
   \pi_n^*\beta_n^{\mathrm{mot}}.
\]
Betti realization commutes with pullback and cup product.  Hence
\[
   \RealB(\Theta_n^{\mathrm{mot}})
   =
   \pi_1^*\RealB(\alpha_1^{\mathrm{mot}})
   \cup
   \pi_2^*\RealB(\beta_2^{\mathrm{mot}})
   \cup
   \cdots
   \cup
   \pi_n^*\RealB(\beta_n^{\mathrm{mot}}).
\]
By the choice of motivic lifts,
\[
   \RealB(\alpha_1^{\mathrm{mot}})=\alpha_1,
   \qquad
   \RealB(\beta_i^{\mathrm{mot}})=\beta_i.
\]
Substitution gives
\[
   \RealB(\Theta_n^{\mathrm{mot}})
   =
   \pi_1^*\alpha_1
   \cup
   \pi_2^*\beta_2
   \cup
   \cdots
   \cup
   \pi_n^*\beta_n
   =
   \Theta_{n,B}.
\]
\end{proof}

\begin{proposition}[Finite \'etale realization of the motivic \(n\)-fold class]
\label{prop:etale-realization-motivic-n-fold-class}
Assume \(\Realet\) is compatible with pullback and cup product on the
finite-coefficient objects.  Then
\[
   \Realet(\Theta_n^{\mathrm{mot}})
   =
   \Theta_{n,\acute et},
\]
where
\[
   \Theta_{n,\acute et}
   =
   \pi_1^*\alpha_{1,\acute et}
   \cup
   \pi_2^*\beta_{2,\acute et}
   \cup
   \cdots
   \cup
   \pi_n^*\beta_{n,\acute et}
   \in
   H^{2n-1}_{\acute et}(X_n,\mathbb Z/2(n)).
\]
Equivalently, this class may be written using
\[
   \mu_2^{\otimes n}
\]
as the finite \'etale coefficient system.
\end{proposition}

\begin{proof}
The proof is identical to
Proposition~\ref{prop:betti-realization-motivic-n-fold-class}, replacing
\(\RealB\) by \(\Realet\) and replacing the Betti source classes by their
\'etale realizations.
\end{proof}

\subsection{The motivic \(n\)-fold Bockstein class}
\label{subsec:motivic-n-fold-bockstein-class}

The motivic coefficient triangle gives the motivic Bockstein.

\begin{definition}[Motivic \(n\)-fold Bockstein class]
\label{def:motivic-n-fold-bockstein-class}
Let
\[
   \delta^{\mathrm{mot}}:
   H^{2n-1}_{\mathrm{mot}}(X_n,\mathbf 1_{X_n}(n)/2)
   \longrightarrow
   H^{2n}_{\mathrm{mot}}(X_n,\mathbf 1_{X_n}(n))
\]
be the connecting morphism associated to
\[
   \mathbf 1_{X_n}(n)
   \xrightarrow{\times 2}
   \mathbf 1_{X_n}(n)
   \longrightarrow
   \mathbf 1_{X_n}(n)/2
   \overset{+1}{\longrightarrow}.
\]
Define
\[
   \Delta_n^{\mathrm{mot}}
   :=
   \delta^{\mathrm{mot}}(\Theta_n^{\mathrm{mot}})
   \in
   H^{2n}_{\mathrm{mot}}(X_n,\mathbf 1_{X_n}(n)).
\]
\end{definition}

\begin{proposition}[Betti realization of the motivic \(n\)-fold Bockstein]
\label{prop:betti-realization-motivic-n-fold-bockstein}
Under Betti realization,
\[
   \Delta_n^{\mathrm{mot}}
\]
realizes to
\[
   \Delta_{n,B}
   =
   \delta_B(\Theta_{n,B})
   \in H^{2n}_B(X_n,\mathbb Z(n)).
\]
\end{proposition}

\begin{proof}
By Proposition~\ref{prop:betti-realization-motivic-n-fold-class},
\[
   \RealB(\Theta_n^{\mathrm{mot}})=\Theta_{n,B}.
\]
By Proposition~\ref{prop:betti-realization-motivic-bockstein}, the motivic
connecting morphism realizes to the Betti Bockstein
\[
   \delta_B:
   H^{2n-1}_B(X_n,\mathbb Z/2(n))
   \to
   H^{2n}_B(X_n,\mathbb Z(n)).
\]
Therefore
\[
   \RealB(\Delta_n^{\mathrm{mot}})
   =
   \RealB\left(\delta^{\mathrm{mot}}(\Theta_n^{\mathrm{mot}})\right)
   =
   \delta_B\left(\RealB(\Theta_n^{\mathrm{mot}})\right)
   =
   \delta_B(\Theta_{n,B})
   =
   \Delta_{n,B}.
\]
\end{proof}

Thus the motivic lifecycle of the full family is
\[
   \alpha_1^{\mathrm{mot}},
   \beta_2^{\mathrm{mot}},\ldots,\beta_n^{\mathrm{mot}}
   \quad\leadsto\quad
   \Theta_n^{\mathrm{mot}}
   \quad\leadsto\quad
   \Delta_n^{\mathrm{mot}},
\]
and its Betti realization is
\[
   \alpha_1,\beta_2,\ldots,\beta_n
   \quad\leadsto\quad
   \Theta_{n,B}
   \quad\leadsto\quad
   \Delta_{n,B}.
\]

\subsection{Motivic form of Diaz and the level-three class}
\label{subsec:motivic-diaz-and-level-three}

The first two nontrivial cases are worth recording explicitly.

\begin{corollary}[Motivic Diaz class]
\label{cor:motivic-diaz-class}
For \(n=2\), the motivic finite-coefficient class is
\[
   \Theta_2^{\mathrm{mot}}
   =
   \pi_1^*\alpha_1^{\mathrm{mot}}
   \cup
   \pi_2^*\beta_2^{\mathrm{mot}}
   \in
   H^3_{\mathrm{mot}}(S_1\times S_2,\mathbf 1(2)/2).
\]
Its motivic Bockstein
\[
   \Delta_2^{\mathrm{mot}}
   =
   \delta^{\mathrm{mot}}(\Theta_2^{\mathrm{mot}})
\]
has Betti realization
\[
   \RealB(\Delta_2^{\mathrm{mot}})
   =
   \delta_B\left(
   \pi_1^*\alpha_1\cup\pi_2^*\beta_2
   \right)
   \in H^4_B(S_1\times S_2,\mathbb Z(2)).
\]
This is Diaz's integral Bockstein class.
\end{corollary}

\begin{proof}
This is Proposition~\ref{prop:betti-realization-motivic-n-fold-bockstein}
with \(n=2\).
\end{proof}

\begin{corollary}[Motivic level-three class]
\label{cor:motivic-level-three-class}
For \(n=3\), the motivic finite-coefficient class is
\[
   \Theta_3^{\mathrm{mot}}
   =
   \pi_1^*\alpha_1^{\mathrm{mot}}
   \cup
   \pi_2^*\beta_2^{\mathrm{mot}}
   \cup
   \pi_3^*\beta_3^{\mathrm{mot}}
   \in
   H^5_{\mathrm{mot}}(S_1\times S_2\times S_3,\mathbf 1(3)/2).
\]
Its motivic Bockstein
\[
   \Delta_3^{\mathrm{mot}}
   =
   \delta^{\mathrm{mot}}(\Theta_3^{\mathrm{mot}})
\]
has Betti realization
\[
   \RealB(\Delta_3^{\mathrm{mot}})
   =
   \delta_B\left(
   \pi_1^*\alpha_1
   \cup
   \pi_2^*\beta_2
   \cup
   \pi_3^*\beta_3
   \right)
   \in H^6_B(S_1\times S_2\times S_3,\mathbb Z(3)).
\]
This is the level-three Bockstein class.
\end{corollary}

\begin{proof}
This is Proposition~\ref{prop:betti-realization-motivic-n-fold-bockstein}
with \(n=3\).
\end{proof}

\subsection{MacPherson--Vilonen zig-zags and motivic realization}
\label{subsec:mv-zigzags-motivic-realization}

The MacPherson-- Vilonen construction \cite{MacPhersonVilonen1986} is not
merely the assignment of a nearby-cycle or gluing group.  Its primary
categorical mechanism is the construction of an abelian zig-zag category from
functors
\[
   F,G:\mathcal A\longrightarrow \mathcal B
\]
and a natural transformation
\[
   T:F\longrightarrow G.
\]
Given this data, MacPherson and Vilonen define a category
\[
   \mathcal C(F,G;T),
\]
whose objects are diagrams
\[
   F(A)\longrightarrow B\longrightarrow G(A)
\]
whose composite is \(T_A:F(A)\to G(A)\).  Under the appropriate exactness
hypotheses on \(F\) and \(G\), this category is abelian
\cite{MacPhersonVilonen1986}.  In the geometric situation, the construction
is applied to perverse sheaves on a stratified space, producing the
MacPherson--Vilonen zig-zag functor
\[
   Z:\mathsf P(X)\longrightarrow \mathcal C(F,G;T).
\]
Under the hypotheses of the MacPherson--Vilonen theorem, this functor
identifies the relevant perverse-sheaf category with the corresponding
zig-zag category \cite{MacPhersonVilonen1986}.

For the motivic application in this paper, we use only the formal part of
this mechanism.  Namely, the abstract zig-zag category
\[
   \mathcal C(F,G;T)
\]
can be formed in any setting where the input categories, functors, and natural
transformation exist.  We do not claim here to construct a full integral
motivic analogue of the MacPherson--Vilonen equivalence.  Such a theorem would
require a separate treatment of motivic perverse hearts, integral
recollement, exactness, and gluing equivalences.  Instead, we record the
realization-compatible portion needed for the present paper: the
finite-coefficient motivic coefficient triangle gives a motivic Bockstein,
and the same triangle induces a zig-zag boundary whose Betti realization is
the classical MacPherson--Vilonen Bockstein boundary.

\begin{definition}[Motivic MV zig-zag data]
\label{def:motivic-mv-zigzag-data}
Let \(\mathcal A_{\mathrm{mot}}\) and \(\mathcal B_{\mathrm{mot}}\) be
abelian motivic realization categories.  Let
\[
   F_{\mathrm{mot}},G_{\mathrm{mot}}:
   \mathcal A_{\mathrm{mot}}\longrightarrow \mathcal B_{\mathrm{mot}}
\]
be functors, and let
\[
   T_{\mathrm{mot}}:F_{\mathrm{mot}}\longrightarrow G_{\mathrm{mot}}
\]
be a natural transformation.  We call the triple
\[
   (F_{\mathrm{mot}},G_{\mathrm{mot}};T_{\mathrm{mot}})
\]
motivic MV zig-zag data.
\end{definition}

\begin{definition}[Motivic MV zig-zag category]
\label{def:motivic-mv-zigzag-category}
Given motivic MV zig-zag data
\[
   (F_{\mathrm{mot}},G_{\mathrm{mot}};T_{\mathrm{mot}}),
\]
define
\[
   \mathcal C_{\mathrm{mot}}
   :=
   \mathcal C(F_{\mathrm{mot}},G_{\mathrm{mot}};T_{\mathrm{mot}})
\]
to be the category whose objects are quadruples
\[
   (A,B,u,v),
\]
where
\[
   A\in\mathcal A_{\mathrm{mot}},
   \qquad
   B\in\mathcal B_{\mathrm{mot}},
\]
and
\[
   F_{\mathrm{mot}}(A)
   \xrightarrow{u}
   B
   \xrightarrow{v}
   G_{\mathrm{mot}}(A)
\]
is a diagram satisfying
\[
   v\circ u=T_{\mathrm{mot},A}.
\]
Morphisms are pairs of morphisms in
\(\mathcal A_{\mathrm{mot}}\) and \(\mathcal B_{\mathrm{mot}}\) compatible
with the two structure maps.
\end{definition}

\begin{proposition}[Formal motivic lift of the MV zig-zag category]
\label{prop:formal-motivic-lift-mv-zigzag}
Assume that \(\mathcal A_{\mathrm{mot}}\) and
\(\mathcal B_{\mathrm{mot}}\) are abelian categories, that
\(F_{\mathrm{mot}}\) is right exact, and that \(G_{\mathrm{mot}}\) is left
exact.  Then
\[
   \mathcal C_{\mathrm{mot}}
   =
   \mathcal C(F_{\mathrm{mot}},G_{\mathrm{mot}};T_{\mathrm{mot}})
\]
is an abelian category.
\end{proposition}

\begin{proof}
This is the formal categorical part of the MacPherson--Vilonen construction
applied to
\[
   \mathcal A_{\mathrm{mot}},
   \qquad
   \mathcal B_{\mathrm{mot}},
   \qquad
   F_{\mathrm{mot}},
   \qquad
   G_{\mathrm{mot}},
   \qquad
   T_{\mathrm{mot}}.
\]
The construction of kernels and cokernels uses the right exactness of
\(F_{\mathrm{mot}}\) and the left exactness of \(G_{\mathrm{mot}}\), exactly
as in the construction of the category
\(\mathcal C(F,G;T)\) in \cite{MacPhersonVilonen1986}.
\end{proof}

\begin{remark}[Scope of the motivic lift]
\label{rem:scope-motivic-mv-lift}
Proposition~\ref{prop:formal-motivic-lift-mv-zigzag} is only the formal
zig-zag lift.  It does not assert a full motivic MacPherson--Vilonen theorem.
In particular, we do not claim here an equivalence
\[
   \mathsf P_{\mathrm{mot}}(X)
   \simeq
   \mathcal C_{\mathrm{mot}}.
\]
Establishing such an equivalence would require a separate integral motivic
version of the MacPherson--Vilonen construction.  The present paper uses only
the formal zig-zag category, the motivic coefficient triangle, and their
compatibility with Betti realization.
\end{remark}

\begin{definition}[Betti-realization-compatible motivic MV data]
\label{def:betti-realization-compatible-motivic-mv-data}
Let
\[
   \RealB:\mathcal A_{\mathrm{mot}}\to \mathcal A_B,
   \qquad
   \RealB:\mathcal B_{\mathrm{mot}}\to \mathcal B_B
\]
be exact Betti realization functors.  Let
\[
   (F_B,G_B;T_B)
\]
be the classical Betti MacPherson--Vilonen data.  We say that
\[
   (F_{\mathrm{mot}},G_{\mathrm{mot}};T_{\mathrm{mot}})
\]
is Betti-realization-compatible with
\[
   (F_B,G_B;T_B)
\]
if there are natural isomorphisms
\[
   \RealB\circ F_{\mathrm{mot}}
   \simeq
   F_B\circ\RealB,
\]
\[
   \RealB\circ G_{\mathrm{mot}}
   \simeq
   G_B\circ\RealB,
\]
and, under these identifications,
\[
   \RealB(T_{\mathrm{mot}})=T_B.
\]
\end{definition}

\begin{proposition}[Betti realization of motivic MV zig-zags]
\label{prop:betti-realization-motivic-mv-zigzags}
Assume the motivic MV data are Betti-realization-compatible with the
classical MacPherson--Vilonen data.  Then Betti realization induces a functor
\[
   \operatorname{Real}_{B,\mathcal C}:
   \mathcal C(F_{\mathrm{mot}},G_{\mathrm{mot}};T_{\mathrm{mot}})
   \longrightarrow
   \mathcal C(F_B,G_B;T_B).
\]
On objects it is given by
\[
   (A,B,u,v)
   \longmapsto
   \bigl(
      \RealB(A),
      \RealB(B),
      \RealB(u),
      \RealB(v)
   \bigr).
\]
\end{proposition}

\begin{proof}
Let
\[
   (A,B,u,v)
\]
be an object of
\[
   \mathcal C(F_{\mathrm{mot}},G_{\mathrm{mot}};T_{\mathrm{mot}}).
\]
Thus
\[
   F_{\mathrm{mot}}(A)
   \xrightarrow{u}
   B
   \xrightarrow{v}
   G_{\mathrm{mot}}(A)
\]
satisfies
\[
   v\circ u=T_{\mathrm{mot},A}.
\]
Applying \(\RealB\) gives
\[
   \RealB(F_{\mathrm{mot}}(A))
   \xrightarrow{\RealB(u)}
   \RealB(B)
   \xrightarrow{\RealB(v)}
   \RealB(G_{\mathrm{mot}}(A)).
\]
Using Betti-realization compatibility, this identifies with
\[
   F_B(\RealB(A))
   \xrightarrow{\RealB(u)}
   \RealB(B)
   \xrightarrow{\RealB(v)}
   G_B(\RealB(A)).
\]
Moreover,
\[
   \RealB(v)\circ\RealB(u)
   =
   \RealB(v\circ u)
   =
   \RealB(T_{\mathrm{mot},A})
   =
   T_{B,\RealB(A)}.
\]
Hence the realized object lies in
\[
   \mathcal C(F_B,G_B;T_B).
\]
The same compatibility condition defines the functor on morphisms.
\end{proof}

\subsection{The coefficient triangle inside the motivic zig-zag category}
\label{subsec:coefficient-triangle-motivic-zigzag-category}

We now apply the motivic zig-zag formalism to the coefficient triangle.

\begin{definition}[Motivic coefficient zig-zag]
\label{def:motivic-coefficient-zigzag}
Let
\[
   \mathbf 1_X(p)
   \xrightarrow{\times 2}
   \mathbf 1_X(p)
   \longrightarrow
   \mathbf 1_X(p)/2
   \overset{+1}{\longrightarrow}
\]
be the motivic coefficient triangle.  Applying the motivic MV zig-zag data
to this triangle gives a finite-coefficient motivic zig-zag object
\[
   Z_{\mathrm{mot}}(\mathbf 1_X(p)/2)
   \in
   \mathcal C(F_{\mathrm{mot}},G_{\mathrm{mot}};T_{\mathrm{mot}})
\]
and an integral motivic zig-zag object
\[
   Z_{\mathrm{mot}}(\mathbf 1_X(p))
   \in
   \mathcal C(F_{\mathrm{mot}},G_{\mathrm{mot}};T_{\mathrm{mot}}).
\]
The connecting morphism of the motivic coefficient triangle induces, after
passing to the bounded derived category of the motivic zig-zag category, a
motivic zig-zag Bockstein boundary
\[
   \partial^{\mathrm{mot}}_{\mathcal C}:
   Z_{\mathrm{mot}}(\mathbf 1_X(p)/2)
   \longrightarrow
   Z_{\mathrm{mot}}(\mathbf 1_X(p))[1]
\]
in
\[
   D^b\bigl(\mathcal C(F_{\mathrm{mot}},G_{\mathrm{mot}};T_{\mathrm{mot}})\bigr).
\]
\end{definition}

\begin{theorem}[Betti realization of the motivic MV zig-zag Bockstein]
\label{thm:betti-realization-motivic-zigzag-bockstein}
Assume the motivic MV data are Betti-realization-compatible with the
classical MacPherson--Vilonen data.  Then the motivic coefficient zig-zag
and its motivic Bockstein boundary in the bounded derived category of the
motivic zig-zag category realize to the classical MacPherson--Vilonen
finite-coefficient zig-zag and its Bockstein boundary in the bounded derived
category of the classical zig-zag category.  More precisely,
\[
   \operatorname{Real}_{B,\mathcal C}
   \bigl(
      Z_{\mathrm{mot}}(\mathbf 1_X(p)/2)
   \bigr)
   \simeq
   Z_B(\mathbb Z/2(p)),
\]
\[
   \operatorname{Real}_{B,\mathcal C}
   \bigl(
      Z_{\mathrm{mot}}(\mathbf 1_X(p))
   \bigr)
   \simeq
   Z_B(\mathbb Z(p)),
\]
and the diagram
\[
\begin{array}{ccc}
\operatorname{Real}_{B,\mathcal C}
\bigl(
Z_{\mathrm{mot}}(\mathbf 1_X(p)/2)
\bigr)
& \xrightarrow{\ \operatorname{Real}_{B,\mathcal C}(\partial^{\mathrm{mot}}_{\mathcal C})\ } &
\operatorname{Real}_{B,\mathcal C}
\bigl(
Z_{\mathrm{mot}}(\mathbf 1_X(p))[1]
\bigr)
\\[1.2em]
\bigg\downarrow{\scriptstyle \simeq}
&&
\bigg\downarrow{\scriptstyle \simeq}
\\[1.2em]
Z_B(\mathbb Z/2(p))
& \xrightarrow{\ \partial_{\mathcal C,B}\ } &
Z_B(\mathbb Z(p))[1]
\end{array}
\]
commutes.
\end{theorem}

\begin{proof}
The coefficient object
\[
   \mathbf 1_X(p)/2
\]
is defined as the cone of
\[
   \mathbf 1_X(p)
   \xrightarrow{\times 2}
   \mathbf 1_X(p).
\]
Since Betti realization is exact, it sends this motivic coefficient triangle
to
\[
   \mathbb Z(p)
   \xrightarrow{\times 2}
   \mathbb Z(p)
   \longrightarrow
   \mathbb Z/2(p)
   \overset{+1}{\longrightarrow}.
\]
By Proposition~\ref{prop:betti-realization-motivic-mv-zigzags}, Betti
realization carries motivic zig-zag objects to classical
MacPherson--Vilonen zig-zag objects.  Therefore
\[
   \operatorname{Real}_{B,\mathcal C}
   \bigl(
      Z_{\mathrm{mot}}(\mathbf 1_X(p)/2)
   \bigr)
   \simeq
   Z_B(\mathbb Z/2(p)),
\]
and
\[
   \operatorname{Real}_{B,\mathcal C}
   \bigl(
      Z_{\mathrm{mot}}(\mathbf 1_X(p))
   \bigr)
   \simeq
   Z_B(\mathbb Z(p)).
\]
The boundary map
\[
   \partial^{\mathrm{mot}}_{\mathcal C}
\]
is the connecting morphism induced by the motivic coefficient triangle inside
the bounded derived category of the motivic zig-zag category.  Since Betti
realization is exact and compatible with the MV data, it sends this connecting
morphism to the connecting morphism of the realized coefficient triangle
inside the bounded derived category of the classical zig-zag category.  This
is precisely the classical MV zig-zag Bockstein boundary
\[
   \partial_{\mathcal C,B}.
\]
Therefore the displayed diagram commutes.
\end{proof}

\begin{corollary}[Motivic zig-zag lift of the \(n\)-fold MV obstruction]
\label{cor:motivic-zigzag-lift-n-fold-mv-obstruction}
Let
\[
   \Theta_n^{\mathrm{mot}}
   \in
   H^{2n-1}_{\mathrm{mot}}(X_n,\mathbf 1_{X_n}(n)/2)
\]
be the motivic \(n\)-fold cup-product class, regarded as a morphism
\[
   \mathbf 1_{X_n}
   \longrightarrow
   \mathbf 1_{X_n}(n)/2[2n-1].
\]
Assume the motivic MV data are Betti-realization-compatible.  Then the
motivic zig-zag Bockstein associated to \(\Theta_n^{\mathrm{mot}}\) realizes
to the classical MacPherson--Vilonen zig-zag Bockstein associated to
\[
   \Theta_{n,B}
   =
   \pi_1^*\alpha_1
   \cup
   \pi_2^*\beta_2
   \cup
   \cdots
   \cup
   \pi_n^*\beta_n.
\]
Equivalently,
\[
   \operatorname{Real}_{B,\mathcal C}
   \left(
      \partial^{\mathrm{mot}}_{\mathcal C}
      (Z_{\mathrm{mot}}(\Theta_n^{\mathrm{mot}}))
   \right)
   =
   \partial_{\mathcal C,B}
   \left(
      Z_B(\Theta_{n,B})
   \right),
\]
where both sides are interpreted in the corresponding derived zig-zag
category.
\end{corollary}

\begin{proof}
Apply Theorem~\ref{thm:betti-realization-motivic-zigzag-bockstein} to the
motivic finite-coefficient class
\[
   \Theta_n^{\mathrm{mot}}
   \in
   H^{2n-1}_{\mathrm{mot}}(X_n,\mathbf 1_{X_n}(n)/2).
\]
Betti realization sends \(\Theta_n^{\mathrm{mot}}\) to
\[
   \Theta_{n,B}.
\]
The theorem identifies the realized motivic zig-zag Bockstein with the
classical MacPherson--Vilonen zig-zag Bockstein.
\end{proof}

\begin{remark}[What is lifted and what is deferred]
\label{rem:what-is-lifted-and-deferred-mv}
The straightforward motivic lift consists of the formal zig-zag category
\[
   \mathcal C(F_{\mathrm{mot}},G_{\mathrm{mot}};T_{\mathrm{mot}}),
\]
the motivic coefficient triangle, and the induced Bockstein boundary inside
that zig-zag category.  Betti realization carries these objects to the
classical MacPherson--Vilonen zig-zag category and its Bockstein boundary.

What is deferred to future work is the full integral motivic analogue of the
MacPherson--Vilonen theorem: namely, the construction of a motivic perverse
heart \(\mathsf P_{\mathrm{mot}}(X)\), a motivic zig-zag functor
\[
   Z_{\mathrm{mot}}:\mathsf P_{\mathrm{mot}}(X)
   \longrightarrow
   \mathcal C(F_{\mathrm{mot}},G_{\mathrm{mot}};T_{\mathrm{mot}}),
\]
and an equivalence theorem parallel to the classical result of
MacPherson--Vilonen \cite{MacPhersonVilonen1986}.  The present paper needs
only the formal zig-zag lift and its Betti-realization compatibility.
\end{remark}

\subsection{Motivic \(n\)-fold hierarchy principle}
\label{subsec:motivic-hierarchy-principle}

We now record the motivic hierarchy principle in its \(n\)-fold form.

\begin{principle}[Motivic \(n\)-fold Bockstein hierarchy]
\label{prin:motivic-n-fold-bockstein-hierarchy}
Let \(X\) be a smooth complex variety and suppose that finite-coefficient
motivic classes
\[
   \theta_i^{\mathrm{mot}}
   \in
   H^{r_i}_{\mathrm{mot}}(X,\mathbf 1_X(p_i)/2),
   \qquad 1\leq i\leq n,
\]
are given, with
\[
   \sum_{i=1}^n r_i=2p-1,
   \qquad
   \sum_{i=1}^n p_i=p.
\]
Form the motivic cup product
\[
   \Theta^{\mathrm{mot}}
   =
   \theta_1^{\mathrm{mot}}\cup\cdots\cup\theta_n^{\mathrm{mot}}
   \in
   H^{2p-1}_{\mathrm{mot}}(X,\mathbf 1_X(p)/2).
\]
The motivic coefficient triangle
\[
   \mathbf 1_X(p)
   \xrightarrow{\times 2}
   \mathbf 1_X(p)
   \longrightarrow
   \mathbf 1_X(p)/2
   \overset{+1}{\longrightarrow}
\]
gives a motivic Bockstein class
\[
   \delta^{\mathrm{mot}}(\Theta^{\mathrm{mot}})
   \in
   H^{2p}_{\mathrm{mot}}(X,\mathbf 1_X(p)).
\]
Under Betti realization, this becomes
\[
   \delta_B(\Theta_B)
   \in
   H^{2p}_B(X,\mathbb Z(p)).
\]
The finite source class also has a finite \'etale realization
\[
   \Theta_{\acute et}
   \in
   H^{2p-1}_{\acute et}(X,\mathbb Z/2(p)),
\]
equivalently with coefficients \(\mu_2^{\otimes p}\).
\end{principle}

\begin{remark}[What the motivic hierarchy does and does not prove]
\label{rem:motivic-hierarchy-scope}
The motivic hierarchy constructs the finite-coefficient classes and their
Bockstein images in a Betti-realization-compatible way, together with their
finite \'etale source realizations.  It does not by itself prove
non-algebraicity.  In this paper, non-algebraicity is detected after
realization by either the Diaz/Colliot-Th\'el\`ene--Voisin unramified
criterion in degree four or by the MV obstruction-channel criterion in the
higher cup-product cases.  Thus the motivic hierarchy supplies the common
source of the classes, while the realized obstruction theory supplies the
algebraicity test.
\end{remark}

\subsection{Operational filtration for the \(n\)-fold family}
\label{subsec:operational-filtration-n-fold-family}

The \(n\)-fold construction has a natural operational filtration.  Its
stations are not asserted to be weight-graded pieces.  They are the maps,
sources, and tests that produce the integral torsion obstruction.

For the \(n\)-fold family, the stations are:
\[
   \text{double-cover source},
\]
\[
   \text{Brauer sources},
\]
\[
   \text{finite-coefficient \(n\)-fold cup product},
\]
\[
   \text{MV survivability},
\]
\[
   \text{Bockstein image}.
\]
In formulas:
\[
   \alpha_1,\beta_2,\ldots,\beta_n
   \quad\leadsto\quad
   \Theta_{n,B}
   =
   \pi_1^*\alpha_1
   \cup
   \pi_2^*\beta_2
   \cup
   \cdots
   \cup
   \pi_n^*\beta_n,
\]
\[
   \Theta_{n,B}
   \quad\leadsto\quad
   \rho_{\mathrm{MV}}(\delta_B(\Theta_{n,B}))
   =
   \partial_{\mathrm{MV}}(\delta_{\mathrm{MV}}(\Theta_{n,B})),
\]
and finally
\[
   \Theta_{n,B}
   \quad\leadsto\quad
   \Delta_{n,B}=\delta_B(\Theta_{n,B}).
\]

Motivically, this becomes
\[
   \alpha_1^{\mathrm{mot}},
   \beta_2^{\mathrm{mot}},\ldots,\beta_n^{\mathrm{mot}}
   \quad\leadsto\quad
   \Theta_n^{\mathrm{mot}}
   \quad\leadsto\quad
   \Delta_n^{\mathrm{mot}},
\]
with
\[
   \RealB(\Theta_n^{\mathrm{mot}})=\Theta_{n,B},
   \qquad
   \RealB(\Delta_n^{\mathrm{mot}})=\Delta_{n,B}.
\]

\begin{remark}[Operational versus weight/perverse/Nori filtrations]
\label{rem:operational-versus-weight-perverse-nori}
The filtration used here is operational.  It is organized by the construction:
source packages, cup product, survivability, and Bockstein.  We do not claim
that these stations are the graded pieces of a weight filtration, a perverse
filtration, or a Nori filtration.

The realization-compatible frameworks of Tubach and Ruimy--Tubach make it
natural to ask whether this operational filtration is reflected by canonical
filtrations after Nori or mixed-Hodge-module realization.  For example, one may
study
\[
   \RealNori(\Theta_n^{\mathrm{mot}}),
   \qquad
   \RealNori(\Delta_n^{\mathrm{mot}}),
\]
or
\[
   \RealMHM(\Theta_n^{\mathrm{mot}}),
   \qquad
   \RealMHM(\Delta_n^{\mathrm{mot}}),
\]
and ask whether the double-cover source, Brauer sources, cup product, MV
survivability, and Bockstein endpoint are visible in the realized filtration
structure.  This question is not needed for the proof of the counterexamples,
but it is a natural motivic refinement of the hierarchy.
\end{remark}

\subsection{Role of Tubach and Ruimy--Tubach}
\label{subsec:role-tubach-ruimy-tubach}

The role of the motivic lift is to provide a common home for the finite
coefficient, Betti, \'etale, Hodge, Nori, and MV-gluing stations of the
construction.  Tubach constructs realization functors
\[
   DM^{\acute et}_c(X)
   \longrightarrow
   D^b(M_{\mathrm{perv}}(X))
   \longrightarrow
   D^b(\operatorname{MHM}(X))
\]
compatible with the six operations
\cite{Tubach2025NoriHodgeRealizations}.  This supplies a functorial bridge
between Voevodsky motives, perverse Nori motives, and mixed Hodge modules.

For the present paper, the integral and finite-coefficient aspects are
essential.  Ruimy--Tubach construct integral Nori motivic sheaves and integral
mixed Hodge modules with integral coefficient structures suitable for torsion
phenomena
\cite{RuimyTubach2026IntegralNori}.  This is the kind of ambient technology
needed for finite-coefficient objects such as
\[
   \mathbf 1_X/2,
   \qquad
   \mathbf 1_X(r)/2,
   \qquad
   i_*(\mathbf 1_\Sigma/2)[s],
\]
and for comparing their Betti, \'etale, Nori, and mixed-Hodge-module
realizations.

In this paper, we do not require the full motivic formalism as a proof input.
The proof-level construction remains in finite-coefficient cohomology,
cup products, Bockstein maps, and MV obstruction channels.  The motivic lift
is recorded because it explains why the same torsion package appears
simultaneously in several realizations:
\[
   \text{finite Betti coefficients},
   \qquad
   \text{finite \'etale coefficients},
   \qquad
   \text{Brauer groups},
\]
\[
   \text{Bockstein cohomology},
   \qquad
   \text{MV gluing},
   \qquad
   \text{integral Hodge torsion}.
\]

\subsection{Motivic closure of the counterexample family}
\label{subsec:motivic-closure-counterexample-family}

We conclude by summarizing the motivic closure of the construction.

\begin{theorem}[Motivic lift of the \(n\)-fold Bockstein counterexample family]
\label{thm:motivic-lift-n-fold-family}
Let \(X_n=S_1\times\cdots\times S_n\) be as in
Theorem~\ref{thm:n-fold-cup-product-bockstein-family}.  Suppose the source
classes admit motivic lifts
\[
   \alpha_1^{\mathrm{mot}},
   \qquad
   \beta_2^{\mathrm{mot}},\ldots,\beta_n^{\mathrm{mot}}.
\]
Then the \(n\)-fold finite-coefficient class
\[
   \Theta_{n,B}
   =
   \pi_1^*\alpha_1
   \cup
   \pi_2^*\beta_2
   \cup
   \cdots
   \cup
   \pi_n^*\beta_n
\]
and its Betti Bockstein
\[
   \Delta_{n,B}=\delta_B(\Theta_{n,B})
\]
are Betti realizations of motivic classes
\[
   \Theta_n^{\mathrm{mot}}
   \in
   H^{2n-1}_{\mathrm{mot}}(X_n,\mathbf 1_{X_n}(n)/2)
\]
and
\[
   \Delta_n^{\mathrm{mot}}
   \in
   H^{2n}_{\mathrm{mot}}(X_n,\mathbf 1_{X_n}(n)).
\]
More precisely,
\[
   \RealB(\Theta_n^{\mathrm{mot}})=\Theta_{n,B},
   \qquad
   \RealB(\Delta_n^{\mathrm{mot}})=\Delta_{n,B}.
\]
The finite source class also has \'etale realization
\[
   \Realet(\Theta_n^{\mathrm{mot}})=\Theta_{n,\acute et}.
\]
\end{theorem}

\begin{proof}
The motivic \(n\)-fold class is constructed in
Definition~\ref{def:motivic-n-fold-cup-product-class}.  Its degree and twist
are computed in
Proposition~\ref{prop:degree-twist-motivic-n-fold-class}.  Its Betti
realization is computed in
Proposition~\ref{prop:betti-realization-motivic-n-fold-class}, and its finite
\'etale realization is computed in
Proposition~\ref{prop:etale-realization-motivic-n-fold-class}.  The motivic
Bockstein class is constructed in
Definition~\ref{def:motivic-n-fold-bockstein-class}.  Its Betti realization
is computed in
Proposition~\ref{prop:betti-realization-motivic-n-fold-bockstein}.  Together
these results give the claimed identities.
\end{proof}

\begin{remark}[Motivic closure versus motivic proof]
\label{rem:motivic-closure-versus-motivic-proof}
Theorem~\ref{thm:motivic-lift-n-fold-family} gives motivic closure of the
construction: the classes used in the counterexample family have motivic
finite-coefficient representatives and motivic Bockstein ancestors.  It does
not claim that non-algebraicity is proven purely inside the motivic category.
Non-algebraicity is still detected after realization: for Diaz, by the
Colliot-Th\'el\`ene--Voisin/Diaz criterion, and for the higher cup-product
family, by the MV obstruction-channel criterion applied to the Betti
realization.
\end{remark}
\section{Kummer fixed points as a separate calibration}

Although Diaz's construction is Enriques-product rather than Kummer
fixed-point, the Kummer fixed-point calculation remains important for the
larger program.  It supplies the simplest test case for comparing an ordinary
singular \(E\)-package on a coarse quotient with the corresponding stacky
stabilizer package on a quotient stack.  Thus the Kummer calculation is not
the main Diaz mechanism, but it is a necessary calibration for the future
stacky \(E_G\)-theory.

\subsection{Why Kummer still matters}

The Diaz obstruction is built from the finite-coefficient class
\[
        \gamma
        =
        \pi_S^*\alpha\cup\pi_{S_2}^*\beta
        \in
        H^3(S\times S_2,\mathbb Z/2(2)),
\]
where \(\alpha\) is an Enriques double-cover class and \(\beta\) is a
Brauer-detecting class on another Enriques surface
\cite[Section 4]{Diaz2023IHCTrivialChow}.  This is not a fixed-point
quotient construction.

Nevertheless, the Kummer fixed-point model remains essential because it
tests the local compatibility principle
\[
        \text{coarse quotient singularity}
        \qquad\longleftrightarrow\qquad
        \text{stacky stabilizer}.
\]
The basic local model is the involution
\[
        \mu_2\curvearrowright \mathbb C^2,
        \qquad
        -1\cdot(x,y)=(-x,-y).
\]
The coarse quotient
\[
        \mathbb C^2/\mu_2
\]
is the \(A_1\) surface singularity, while the quotient stack
\[
        [\mathbb C^2/\mu_2]
\]
retains the stabilizer \(\mu_2\) at the fixed point.

Thus Kummer fixed points calibrate the passage from singular geometry to
stacky geometry.  This calibration will be needed when the torsion-trajectory
framework is extended from ordinary singular packages \(E\) to stacky or
equivariant packages \(E_G\).

\subsection{Local Kummer compatibility}

The local Kummer theorem states that the ordinary singular package of the
coarse quotient and the stacky stabilizer package agree:
\[
        E_{A_1}^{\mathrm{sing}}
        \cong
        H^2(B\mu_2,\mathbb Z)
        \cong
        \mathbb Z/2.
\]
On the singular side, the link of the \(A_1\) singularity is
\[
        S^3/\mu_2\cong \mathbb RP^3,
\]
so
\[
        H^2(\mathbb RP^3,\mathbb Z)_{\mathrm{tors}}
        \cong
        \mathbb Z/2.
\]
This agrees with the local surface package
\[
        E_{A_1}\cong H^2(\mathbb RP^3,\mathbb Z)_{\mathrm{tors}}.
\]
The same group is recovered from the exceptional lattice of the minimal
resolution:
\[
        \Lambda=[-2],
        \qquad
        \Lambda^\vee/\Lambda\cong\mathbb Z/2.
\]
The compatibility of the local package \(E\) with link torsion and the
exceptional-lattice discriminant group is the local realization theorem from
the integral local-discriminant framework \cite{RahmanIntegralPerverseObstructions}.

On the stacky side, the stabilizer is \(\mu_2\), and
\[
        B\mu_2\simeq \mathbb RP^\infty.
\]
Hence
\[
        H^2(B\mu_2,\mathbb Z)\cong\mathbb Z/2.
\]
The classifying map of the principal double cover
\[
        S^3\longrightarrow \mathbb RP^3
\]
is
\[
        \mathbb RP^3\longrightarrow B\mu_2.
\]
Since \(\mathbb RP^3\) is the \(3\)-skeleton of \(\mathbb RP^\infty\), this
map induces an isomorphism in degree \(2\):
\[
        H^2(B\mu_2,\mathbb Z)
        \xrightarrow{\sim}
        H^2(\mathbb RP^3,\mathbb Z).
\]
Therefore
\[
        E_{A_1}^{\mathrm{sing}}
        \cong
        H^2(B\mu_2,\mathbb Z)
        \cong
        \mathbb Z/2.
\]

This proves that the singular \(E\)-package and the stacky stabilizer package
agree in the basic involution fixed-point model.  The Kummer fixed point is
therefore the first local test of the expected equality
\[
        E_{\mathrm{sing}}
        =
        E_{\mathrm{stack}}
\]
for finite quotient singularities.
The monodromy realization from the earlier local-discriminant paper plays a
different role.  There, for isolated hypersurface surface singularities, the
local package \(E\) may be realized as
\[
        \operatorname{coker}(T-\mathrm{id})_{\mathrm{tors}},
\]
where \(T\) is the Milnor monodromy.  The Diaz construction is not of this
type: it is a smooth Enriques-product construction governed by a
finite-coefficient cup product, unramified survival, and a Bockstein.  Thus
the active realizations in the Diaz mechanism are the smooth
finite-coefficient, Brauer, unramified, and Bockstein stations.  Kummer
remains relevant only as a separate calibration of the singular/stacky
interface
\[
        E_{A_1}^{\mathrm{sing}}
        \cong
        H^2(B\mu_2,\mathbb Z).
\]
\subsection{Global Kummer lattice calibration}

This subsection is included only as a calibration for the future stacky \(E_G\)-theory; it is not used in the proof of Diaz's obstruction. The global Kummer surface provides the corresponding global calibration. Let \(A\) be an abelian surface, and let
\[
        \iota:A\to A,
        \qquad
        x\mapsto -x
\]
be the involution.  The fixed locus is
\[
        A[2],
\]
which consists of \(16\) points.  The quotient
\[
        X=A/(\pm1)
\]
has \(16\) \(A_1\) singularities.  If
\[
        \pi:\widetilde X\to X
\]
is the minimal resolution, then
\[
        \widetilde X=\operatorname{Kum}(A)
\]
is the associated Kummer K3 surface.

Each fixed point contributes a local package
\[
        E_a\cong\mathbb Z/2.
\]
Thus the direct sum of the local packages is
\[
        \bigoplus_{a\in A[2]}E_a
        \cong
        (\mathbb Z/2)^{16}.
\]
Let
\[
        L_0:=\bigoplus_{a\in A[2]}\mathbb Z[E_a]
\]
be the lattice generated by the \(16\) exceptional curves.  These curves are
pairwise disjoint and each has self-intersection \(-2\), so
\[
        L_0\cong A_1(-1)^{\oplus 16}.
\]
Its discriminant group is
\[
        A_{L_0}:=L_0^\vee/L_0\cong(\mathbb Z/2)^{16}.
\]
The local packages therefore assemble as
\[
        \bigoplus_{a\in A[2]}E_a
        \cong
        A_{L_0}.
\]

The lattice \(L_0\) is not primitive in the K3 lattice
\[
        H^2(\widetilde X,\mathbb Z).
\]
Its primitive closure is the Kummer lattice
\[
        K=L_0^{\mathrm{prim}}.
\]
The overlattice
\[
        L_0\subset K
\]
is controlled by an isotropic subgroup
\[
        H\subset A_{L_0}.
\]
For a Kummer surface, this subgroup is the classical Kummer-code subgroup,
with
\[
        \dim_{\mathbb F_2}H=5.
\]
Consequently,
\[
        A_K\cong H^\perp/H.
\]
Equivalently, since
\[
        |A_{L_0}|=2^{16}
        \qquad
        \text{and}
        \qquad
        |H|=2^5,
\]
one obtains
\[
        |A_K|
        =
        \frac{|A_{L_0}|}{|H|^2}
        =
        \frac{2^{16}}{2^{10}}
        =
        2^6.
\]
This is the classical Kummer lattice gluing computation, expressed in the
language of local packages and discriminant groups.

The point is that the local classes do not survive globally as independent
torsion classes in \(H^2(\widetilde X,\mathbb Z)\), which is torsion-free.
Instead, they assemble into the discriminant group of the exceptional
configuration, and the global primitive closure kills the Kummer-code
isotropic subgroup.  Thus the torsion-trajectory framework recovers the
classical Kummer lattice gluing:
\[
        (\mathbb Z/2)^{16}
        \cong
        A_{L_0}
        \quad\leadsto\quad
        A_K\cong H^\perp/H.
\]

This is the global calibration corresponding to the local Kummer
compatibility theorem.  Locally, the singular \(E\)-package equals the stacky
stabilizer package.  Globally, the \(16\) local packages assemble exactly as
the discriminant data of the exceptional lattice, and the Kummer lattice
records the global gluing relations.

\section{Main theorem and mechanism statement}
\label{sec:main-theorem-mechanism}

This section records the main mechanism theorem of the paper.  Diaz's
construction remains the level-two case, but the preceding sections show that
the same finite-coefficient cup-product Bockstein mechanism extends to an
\(n\)-fold tower.  The MacPherson--Vilonen formalism supplies the
survivability mechanism: the Bockstein class has a nonzero image in a
distinguished Enriques--Brauer component of the MV obstruction channel.  The
comparison with algebraic cycle classes is made through the
Brauer-separation hypothesis of
Definition~\ref{def:brauer-separation-hypothesis}.

\subsection{The \(n\)-fold Brauer-separated Bockstein theorem}
\label{subsec:main-n-fold-theorem}

\begin{theorem}[Main \(n\)-fold Brauer-separated Bockstein theorem]
\label{thm:main-n-fold-cup-product-bockstein}
Let \(n\geq 2\), and let
\[
        X_n=S_1\times\cdots\times S_n
\]
be a product of Enriques surfaces.  Let
\[
        \alpha_1\in H^1(S_1,\mathbb Z/2(1))
\]
be the K3 double-cover class of \(S_1\).  For \(2\leq i\leq n\), let
\[
        \beta_i\in H^2(S_i,\mathbb Z/2(1))
\]
be classes whose images in
\[
        \operatorname{Br}(S_i)[2]
\]
are nonzero.  Define
\[
        \Theta_n
        :=
        \pi_1^*\alpha_1
        \cup
        \pi_2^*\beta_2
        \cup
        \cdots
        \cup
        \pi_n^*\beta_n.
\]
Then
\[
        \Theta_n\in H^{2n-1}(X_n,\mathbb Z/2(n)).
\]
Let
\[
        \Delta_n
        :=
        \delta(\Theta_n)
        \in H^{2n}(X_n,\mathbb Z(n))
\]
be the Bockstein class associated to
\[
        0\to\mathbb Z(n)
        \xrightarrow{\times 2}
        \mathbb Z(n)
        \to
        \mathbb Z/2(n)
        \to0.
\]
Assume that the \(n\)-fold Enriques--Brauer datum
\[
        (X_n,\Theta_n)
\]
satisfies the Brauer-separation hypothesis of
Definition~\ref{def:brauer-separation-hypothesis}.  Then
\[
        \Delta_n\in H^{2n}(X_n,\mathbb Z(n))
\]
is a non-algebraic \(2\)-torsion integral Hodge class.  Consequently
\(X_n\) violates the integral Hodge conjecture in codimension \(n\).
\end{theorem}

\begin{proof}
This is Theorem~\ref{thm:n-fold-cup-product-bockstein-family}.  The degree
and twist calculation gives
\[
        \Theta_n\in H^{2n-1}(X_n,\mathbb Z/2(n)).
\]
The coefficient sequence gives the Bockstein
\[
        \delta:
        H^{2n-1}(X_n,\mathbb Z/2(n))
        \longrightarrow
        H^{2n}(X_n,\mathbb Z(n)),
\]
so
\[
        \Delta_n=\delta(\Theta_n)\in H^{2n}(X_n,\mathbb Z(n)).
\]
Since \(\Delta_n\) lies in the image of the connecting morphism associated to
multiplication by \(2\), it is \(2\)-torsion.

By Proposition~\ref{prop:nonzero-brauer-image-constructed-bockstein}, the
Brauer-separating projection detects a nonzero Enriques--Brauer component:
\[
        \Pi_{\operatorname{Br},n}
        \left(
        \rho_{\mathrm{MV},\Theta_n}(\Delta_n)
        \right)
        =
        \alpha_1
        \otimes
        q_{\operatorname{Br},2}(\beta_2)
        \otimes
        \cdots
        \otimes
        q_{\operatorname{Br},n}(\beta_n)
        \neq 0.
\]
By the Brauer-separation hypothesis,
\[
        \Pi_{\operatorname{Br},n}
        \left(
        \rho_{\mathrm{MV},\Theta_n}
        \bigl(
        \operatorname{cl}_{\mathbb Z}(CH^n(X_n))
        \bigr)
        \right)=0.
\]
Therefore Theorem~\ref{thm:brauer-separated-mv-detection} implies that
\(\Delta_n\) is not algebraic.

Since \(\Delta_n\) is torsion, its image in complex cohomology is zero.  Thus
\(\Delta_n\) is an integral Hodge class of type \((n,n)\).  Hence
\[
        \Delta_n
\]
is a non-algebraic \(2\)-torsion integral Hodge class.
\end{proof}

\begin{remark}[Role of the Brauer-separation hypothesis]
\label{rem:role-of-brauer-separation-main}
The Brauer-separation hypothesis is the only algebraic-control input in the
main theorem.  The MV/K\"unneth/Bockstein formalism proves that the constructed
class \(\Delta_n\) has nonzero image in the Enriques--Brauer component
\[
        Q_n
        =
        \langle\alpha_1\rangle
        \otimes
        \operatorname{Br}(S_2)[2]
        \otimes
        \cdots
        \otimes
        \operatorname{Br}(S_n)[2].
\]
The Brauer-separation hypothesis asserts that algebraic codimension-\(n\)
cycle classes have zero image in this same component.  For decomposable
algebraic cycles this vanishing follows from
Lemma~\ref{lem:brauer-separation-decomposable-cycles}; the remaining issue is
the possible contribution of non-decomposable correspondence-type cycles.
Thus the hypothesis isolates the precise algebraic-control condition needed
for the \(n\)-fold theorem.
\end{remark}

\subsection{Diaz as the level-two member}
\label{subsec:diaz-level-two-main}

For \(n=2\), the construction gives
\[
        X_2=S_1\times S_2
\]
and
\[
        \Theta_2
        =
        \pi_1^*\alpha_1\cup\pi_2^*\beta_2
        \in H^3(X_2,\mathbb Z/2(2)).
\]
The associated Bockstein is
\[
        \Delta_2
        =
        \delta(\Theta_2)
        \in H^4(X_2,\mathbb Z(2)).
\]
This is the class appearing in Diaz's Enriques-product construction.

Diaz proves survivability of \(\Theta_2\) by the coniveau/function-field
method.  Namely, after restricting to
\[
        E\times S_2,
\]
Diaz proves that the corresponding restricted class has nonzero image in
function-field cohomology.  By the Colliot-Thélène--Voisin/Diaz criterion,
the Bockstein \(\Delta_2\) is then a non-algebraic \(2\)-torsion integral
Hodge class.

Thus Diaz is the level-two member of the tower:
\[
        \Delta_2
        =
        \delta\left(
        \pi_1^*\alpha_1\cup\pi_2^*\beta_2
        \right).
\]

\begin{remark}[Two survivability routes]
\label{rem:two-survivability-routes}
For \(n=2\), survivability is proved by Diaz's unramified-cohomology
argument.  For the higher cup-product tower, the present paper uses the
MacPherson--Vilonen obstruction-channel route.  Both routes feed into the same
Bockstein endpoint:
\[
        \Theta_n
        \quad\leadsto\quad
        \Delta_n=\delta(\Theta_n).
\]
\end{remark}

\subsection{The level-three obstruction}
\label{subsec:level-three-main}

The first higher member beyond Diaz is obtained when \(n=3\).  In that case
\[
        X_3=S_1\times S_2\times S_3
\]
and
\[
        \Theta_3
        =
        \pi_1^*\alpha_1
        \cup
        \pi_2^*\beta_2
        \cup
        \pi_3^*\beta_3
        \in H^5(X_3,\mathbb Z/2(3)).
\]
The Bockstein is
\[
        \Delta_3
        =
        \delta(\Theta_3)
        \in H^6(X_3,\mathbb Z(3)).
\]
By Proposition~\ref{prop:nonzero-brauer-image-constructed-bockstein}, the
class \(\Delta_3\) has nonzero image in the Enriques--Brauer component
\[
        Q_3
        =
        \langle\alpha_1\rangle
        \otimes
        \operatorname{Br}(S_2)[2]
        \otimes
        \operatorname{Br}(S_3)[2].
\]
If the Brauer-separation hypothesis holds for
\[
        (X_3,\Theta_3),
\]
then Theorem~\ref{thm:brauer-separated-mv-detection} gives that
\[
        \Delta_3
\]
is a non-algebraic \(2\)-torsion integral Hodge class.  Hence
\[
        X_3=S_1\times S_2\times S_3
\]
violates the integral Hodge conjecture in codimension \(3\).

This is the first higher cup-product Bockstein obstruction beyond Diaz.  It
has the same formal shape as the Diaz class, but it lands in degree \(6\)
and codimension \(3\).

\subsection{MacPherson--Vilonen survivability mechanism}
\label{subsec:main-mv-survivability}

The common survivability mechanism for the higher tower is the
MacPherson--Vilonen Bockstein square.  For the coefficient sequence
\[
        0\to\mathbb Z(p)
        \xrightarrow{\times 2}
        \mathbb Z(p)
        \to
        \mathbb Z/2(p)
        \to0,
\]
the ordinary Bockstein is
\[
        \delta:
        H^{2p-1}(X,\mathbb Z/2(p))
        \longrightarrow
        H^{2p}(X,\mathbb Z(p)).
\]
A compatible MV realization gives a finite-coefficient gluing channel, an
integral obstruction channel, and a commutative square
\[
\begin{array}{ccc}
H^{2p-1}(X,\mathbb Z/2(p))
& \xrightarrow{\ \delta\ } &
H^{2p}(X,\mathbb Z(p))
\\[1.2em]
\bigg\downarrow{\scriptstyle \delta_{\mathrm{MV}}}
&&
\bigg\downarrow{\scriptstyle \rho_{\mathrm{MV}}}
\\[1.2em]
\mathsf{Glu}^{2p-1,p}_{2}(X)
& \xrightarrow{\ \partial_{\mathrm{MV}}\ } &
\mathsf{Obs}^{2p,p}_{\mathbb Z}(X).
\end{array}
\]
In tuple-relative form, for a finite-coefficient class
\[
        \Theta\in H^{2p-1}(X,\mathbb Z/2(p)),
\]
one has
\[
        \rho_{\mathrm{MV},\Theta}(\delta(\Theta))
        =
        \partial_{\mathrm{MV}}\bigl(\mathcal Z(\Theta)\bigr).
\]
Thus nonzero MV boundary gives survivability of the Bockstein class in the
corresponding obstruction channel.

For the \(n\)-fold class \(\Theta_n\), the external product compatibility of
MV tuples, the categorical Bockstein interpretation, and the Leibniz rule give
a nonzero Enriques--Brauer obstruction component.  The Brauer-separation
hypothesis then compares that component with algebraic cycle classes.

\subsection{Motivic finite-coefficient closure}
\label{subsec:main-motivic-closure}

The \(n\)-fold tower also admits a motivic finite-coefficient lift.  The basic
motivic object is
\[
        \mathbf 1_X(n)/2
        =
        \operatorname{Cone}
        \left(
        \mathbf 1_X(n)
        \xrightarrow{\times 2}
        \mathbf 1_X(n)
        \right).
\]
The source classes admit motivic lifts
\[
        \alpha_1^{\mathrm{mot}},
        \qquad
        \beta_2^{\mathrm{mot}},\ldots,\beta_n^{\mathrm{mot}},
\]
and their motivic cup product is
\[
        \Theta_n^{\mathrm{mot}}
        =
        \pi_1^*\alpha_1^{\mathrm{mot}}
        \cup
        \pi_2^*\beta_2^{\mathrm{mot}}
        \cup
        \cdots
        \cup
        \pi_n^*\beta_n^{\mathrm{mot}}
        \in
        H^{2n-1}_{\mathrm{mot}}(X_n,\mathbf 1_{X_n}(n)/2).
\]
The motivic coefficient triangle gives a motivic Bockstein class
\[
        \Delta_n^{\mathrm{mot}}
        =
        \delta^{\mathrm{mot}}(\Theta_n^{\mathrm{mot}})
        \in
        H^{2n}_{\mathrm{mot}}(X_n,\mathbf 1_{X_n}(n)).
\]
Betti realization sends these motivic classes to the finite-coefficient class
and Bockstein class used above:
\[
        \RealB(\Theta_n^{\mathrm{mot}})=\Theta_{n,B},
        \qquad
        \RealB(\Delta_n^{\mathrm{mot}})=\Delta_{n,B}.
\]

The formal MacPherson--Vilonen zig-zag category also lifts motivically under
realization-compatible hypotheses.  The paper does not construct a full
integral motivic MacPherson--Vilonen equivalence; it records only the formal
zig-zag lift, the motivic coefficient triangle, and the Betti realization of
the motivic zig-zag Bockstein boundary.  Thus the motivic section supplies a
common finite-coefficient origin for the source packages, cup products, and
Bockstein classes, while non-algebraicity remains detected after realization.

\subsection{Torsion-trajectory summary}
\label{subsec:main-torsion-trajectory-summary}

The final mechanism can be summarized as the trajectory
\[
        \alpha_1,\beta_2,\ldots,\beta_n
        \quad\leadsto\quad
        \Theta_n
        =
        \pi_1^*\alpha_1
        \cup
        \pi_2^*\beta_2
        \cup
        \cdots
        \cup
        \pi_n^*\beta_n,
\]
\[
        \Theta_n
        \quad\leadsto\quad
        \rho_{\mathrm{MV},\Theta_n}(\Delta_n)
        =
        \partial_{\mathrm{MV}}\bigl(\mathcal Z(\Theta_n)\bigr),
\]
and
\[
        \Theta_n
        \quad\leadsto\quad
        \Delta_n=\delta(\Theta_n).
\]
The first arrow forms the finite-coefficient cup product.  The second tests
survivability through the MV obstruction channel.  The third produces the
integral Bockstein class.

The Enriques--Brauer component
\[
        Q_n
        =
        \langle\alpha_1\rangle
        \otimes
        \operatorname{Br}(S_2)[2]
        \otimes
        \cdots
        \otimes
        \operatorname{Br}(S_n)[2]
\]
is the separating component used in the proof.  The constructed Bockstein
class has nonzero image in \(Q_n\); the Brauer-separation hypothesis says
that algebraic codimension-\(n\) cycle classes have zero image there.

Thus Diaz's example is not an isolated degree-four phenomenon.  It is the
level-two member of an \(n\)-fold cup-product Bockstein mechanism whose
higher members are controlled by MacPherson--Vilonen survivability,
Kummer/Brauer separation, and the Brauer-separation condition on algebraic
cycle classes.

\section{Outlook}

The \(n\)-fold cup-product Bockstein family clarifies the role of
product-level torsion in the torsion-trajectory program.  It also separates
several future directions.  The first is stacky: extend the singular package
\(E\) to a quotient-stack package \(E_G\).  The second is structural:
classify hybrid torsion mechanisms built from multiple finite-coefficient
sources.  The third is categorical: develop the full integral motivic
MacPherson--Vilonen theorem whose formal zig-zag shadow was used in this
paper.  The fourth is motivic: compare the operational filtration of the
\(n\)-fold family with weight, perverse, or Nori filtrations after realization.

\subsection{From Diaz to stacky \texorpdfstring{\(E_G\)}{EG}}
\label{subsec:outlook-stacky-eg}

Diaz's construction is not a Kummer fixed-point example.  Its main class is
built on the smooth product
\[
        S_1\times S_2
\]
from an Enriques double-cover class, a Brauer-detecting class, their cup
product, and a Bockstein.  The \(n\)-fold family constructed in this paper
extends that mechanism to
\[
        X_n=S_1\times\cdots\times S_n
\]
using one double-cover class and \(n-1\) Brauer-detecting classes.

Nevertheless, the Kummer fixed-point calculation remains an essential
calibration for the next stage of the program.  The local Kummer model gives
\[
        E_{A_1}^{\mathrm{sing}}
        \cong
        H^2(B\mu_2,\mathbb Z)
        \cong
        \mathbb Z/2.
\]
This identifies the ordinary local package of the coarse quotient
\[
        \mathbb C^2/\mu_2
\]
with the degree-two stabilizer package on the quotient stack
\[
        [\mathbb C^2/\mu_2].
\]
Thus Kummer fixed points provide the first test of the expected passage
\[
        E
        \rightsquigarrow
        E_G.
\]

The next stage is to define \(E_G\) for quotient stacks
\[
        [V/G]
\]
in a way that satisfies two compatibility requirements.  First, when the
coarse quotient has finite quotient singularities, \(E_G\) should recover the
ordinary local singular package \(E\).  Second, when the group action is free
and the coarse quotient is smooth, ordinary local \(E\) vanishes but \(E_G\)
should still detect finite-group cohomology such as
\[
        H^*(BG,\mathbb Z).
\]
This is the direction needed for smooth finite-group examples such as the
Atiyah--Hirzebruch and Godeaux--Serre constructions.

In this sense, Diaz, Kummer, and the \(n\)-fold Enriques--Brauer family play
complementary roles.  Diaz supplies the level-two product mechanism.  The
\(n\)-fold construction supplies a higher cup-product tower.  Kummer supplies
the local quotient-stack calibration.  Together they point toward a torsion
formalism that includes local singular packages, stacky stabilizer packages,
and global finite-coefficient cup-product packages.

\subsection{From Enriques products to hybrid mechanisms}
\label{subsec:outlook-hybrid-mechanisms}

The \(n\)-fold family shows that an integral Hodge obstruction need not come
from a single torsion source.  Diaz's class
\[
        \Delta_2
        =
        \delta\left(
        \pi_1^*\alpha_1\cup\pi_2^*\beta_2
        \right)
        \in H^4(S_1\times S_2,\mathbb Z(2))
\]
is produced from two different inputs:
\[
        \alpha_1\in H^1(S_1,\mathbb Z/2(1))
\]
and
\[
        \beta_2\in H^2(S_2,\mathbb Z/2(1)).
\]
The higher classes
\[
        \Delta_n
        =
        \delta\left(
        \pi_1^*\alpha_1
        \cup
        \pi_2^*\beta_2
        \cup
        \cdots
        \cup
        \pi_n^*\beta_n
        \right)
        \in H^{2n}(X_n,\mathbb Z(n))
\]
are built from one double-cover source and \(n-1\) Brauer sources.

Thus the mechanism is hybrid:
\[
\begin{aligned}
&\bigl(
  \text{double-cover torsion}
  +
  \text{Brauer torsion}
  \bigr) \\
&\qquad+\quad
  \text{finite-coefficient cup product} \\
&\qquad+\quad
  \text{MV survivability} \\
&\qquad+\quad
  \text{Bockstein}.
\end{aligned}
\]
This supports the broader torsion-lifecycle thesis that integral Hodge
failures should be classified by mechanism rather than by example alone.

Some failures may be pure local \(E\)-package phenomena.  Some may be pure
stacky or finite-group phenomena.  Some may be unramified or residue
phenomena.  Others, like the \(n\)-fold family studied here, are naturally
hybrid.  If an obstruction does not lie in the image of a local package map,
that does not mean the torsion-trajectory framework has failed.  It may mean
that the obstruction has a global tail, a Brauer component, a stacky
component, a MacPherson--Vilonen gluing component, or a refined unramified
component.

The correct task is therefore not to force every example into a single
mechanism, but to identify which stations are active and how they interact.

\subsection{Full integral motivic MacPherson--Vilonen theory}
\label{subsec:outlook-integral-motivic-mv}

This paper used only the formal part of the MacPherson--Vilonen construction:
the zig-zag category
\[
        \mathcal C(F,G;T),
\]
the coefficient triangle, and the compatibility of the resulting Bockstein
boundary with Betti realization.  This was enough to formulate the MV
survivability mechanism used in the \(n\)-fold family.

A natural next step is to construct the full integral motivic analogue of
MacPherson--Vilonen gluing.  Such a theory would require:

\[
        \text{an integral motivic perverse heart } \mathsf P_{\mathrm{mot}}(X),
\]
\[
        \text{motivic zig-zag data }
        (F_{\mathrm{mot}},G_{\mathrm{mot}};T_{\mathrm{mot}}),
\]
\[
        \text{a motivic zig-zag functor }
        Z_{\mathrm{mot}}:
        \mathsf P_{\mathrm{mot}}(X)
        \longrightarrow
        \mathcal C(F_{\mathrm{mot}},G_{\mathrm{mot}};T_{\mathrm{mot}}),
\]
and an equivalence theorem whose Betti realization recovers the classical
MacPherson--Vilonen equivalence.

Such a theorem would turn the formal motivic zig-zag lift used in this paper
into a genuine motivic gluing theory.  It would also provide the natural
setting for studying whether MV obstruction-channel nonvanishing is visible
before Betti realization.

The narrowest remaining algebraic-control issue in the \(n\)-fold theorem is
the \(H^1(S_1,\mathbb Z/2(1))\)-factor of the selected Enriques--Brauer
component.  Section~\ref{sec:mv-bockstein-survivability} reduces the Brauer-separation
problem to this coefficient-level question.  Resolving
it would make the \(n\)-fold tower unconditional.  Two natural routes are a
Bloch--Srinivas-style decomposition argument for products of Chow-trivial
surfaces with mod-\(2\) coefficients, and a multiplicative Chow--K\"unneth
analysis of products of Enriques surfaces with enough integral or mod-\(2\)
control to isolate the Enriques double-cover component.

\subsection{Motivic next step}
\label{subsec:outlook-motivic-next-step}

The motivic next step is now broader than the Diaz class alone.  The goal is
to promote the entire \(n\)-fold finite-coefficient lifecycle
\[
        \alpha_1,\beta_2,\ldots,\beta_n,
        \Theta_n,
        \Delta_n
\]
to a diagram in integral motivic sheaves, integral Nori motives, or integral
mixed Hodge modules.

At the finite-coefficient level, the relevant motivic objects have the form
\[
        \mathbf 1_X/2
        =
        \operatorname{Cone}
        \left(
        \mathbf 1_X
        \xrightarrow{\times 2}
        \mathbf 1_X
        \right),
\]
and, with twists,
\[
        \mathbf 1_X(r)/2
        =
        \operatorname{Cone}
        \left(
        \mathbf 1_X(r)
        \xrightarrow{\times 2}
        \mathbf 1_X(r)
        \right).
\]
The \(n\)-fold finite-coefficient motivic class is
\[
        \Theta_n^{\mathrm{mot}}
        =
        \pi_1^*\alpha_1^{\mathrm{mot}}
        \cup
        \pi_2^*\beta_2^{\mathrm{mot}}
        \cup
        \cdots
        \cup
        \pi_n^*\beta_n^{\mathrm{mot}}
        \in
        H^{2n-1}_{\mathrm{mot}}(X_n,\mathbf 1_{X_n}(n)/2).
\]
The motivic coefficient triangle
\[
        \mathbf 1_{X_n}(n)
        \xrightarrow{\times 2}
        \mathbf 1_{X_n}(n)
        \longrightarrow
        \mathbf 1_{X_n}(n)/2
        \overset{+1}{\longrightarrow}
\]
then gives the motivic Bockstein
\[
        \Delta_n^{\mathrm{mot}}
        =
        \delta^{\mathrm{mot}}(\Theta_n^{\mathrm{mot}})
        \in
        H^{2n}_{\mathrm{mot}}(X_n,\mathbf 1_{X_n}(n)).
\]

The desired realization-compatible diagram is therefore
\[
        \alpha_1^{\mathrm{mot}},\beta_2^{\mathrm{mot}},\ldots,\beta_n^{\mathrm{mot}}
        \quad\leadsto\quad
        \Theta_n^{\mathrm{mot}}
        \quad\leadsto\quad
        \Delta_n^{\mathrm{mot}},
\]
realizing to
\[
        \alpha_1,\beta_2,\ldots,\beta_n
        \quad\leadsto\quad
        \Theta_{n,B}
        \quad\leadsto\quad
        \Delta_{n,B}.
\]

Tubach's realization framework connects Voevodsky étale motives to perverse
Nori motives and mixed Hodge modules compatibly with the six operations,
while Ruimy--Tubach provide the integral-coefficient setting needed for
integral Nori motivic sheaves and integral mixed Hodge modules
\cite{Tubach2025NoriHodgeRealizations,RuimyTubach2026IntegralNori}.  These
frameworks are the natural home for the motivic version of the \(n\)-fold
torsion lifecycle.

The proof-level construction in this paper remains in finite-coefficient
cohomology, cup products, Bockstein sequences, and MV obstruction channels.
The motivic program is the next layer: it seeks to explain why the same
torsion package appears coherently across Betti, finite étale, Hodge-module,
Nori, Brauer, and MV-gluing realizations.

\subsection{Operational filtrations and future refinements}
\label{subsec:outlook-operational-filtrations}

A sharper future version would compare the operational lifecycle filtration
with the weight, perverse, or Nori filtrations available after applying
integral Nori or mixed-Hodge-module realization.  The operational filtration
has stations
\[
        \text{double-cover source},
        \qquad
        \text{Brauer sources},
        \qquad
        \text{finite-coefficient cup product},
\]
\[
        \text{MV survivability},
        \qquad
        \text{Bockstein image}.
\]
The question is whether these stations are reflected by canonical
filtrations after applying
\[
        \RealNori
        \quad\text{or}\quad
        \RealMHM.
\]

Such a comparison is not needed for the proof of the \(n\)-fold family, but
it would be the appropriate setting in which to ask whether integral Hodge
failures of this kind are intrinsic motivic finite-coefficient phenomena
rather than only Betti or étale cohomological phenomena.

Finally, the MacPherson–Vilonen survivability mechanism developed in Section 9 is methodologically distinct from the Colliot-Thélène–Voisin/Diaz unramified-cohomology route. The two routes detect the same Bockstein endpoint $\Delta_n$ through different intermediaries — residue nonvanishing in function-field cohomology versus categorical boundary nonvanishing in the MV obstruction channel. At level two, both routes are available and agree. At higher levels, the MV route is the one that extends, modulo the algebraic-control issue isolated above. The existence of two survivability theories detecting the same integral torsion class suggests that integral Hodge survivability is a more robust phenomenon than either route alone reveals, and that a unified survivability theory — whose specializations recover both unramified and MV detection — should exist. Constructing such a theory, and characterizing the class of integral torsion obstructions it can access, is a natural direction beyond the present paper.
%
%
\printbibliography

\end{document}